%% file: main.tex
\documentclass[11pt,reqno]{amsart}

\input{preamble}

\title{Khovanov homology and exotic $4$-manifolds}

\author{Qiuyu Ren}
\address{Department of Mathematics, University of California, Berkeley, Berkeley, CA 94720, USA}
\email{qiuyu\_ren@berkeley.edu}

\author{Michael Willis}
\address{Department of Mathematics, Texas A\&M University, College Station, TX 77843, USA}
\email{msw188@tamu.edu}

\begin{document}

\begin{abstract}
We show that the Khovanov-Rozansky $\mathfrak{gl}_2$ skein lasagna module distinguishes the exotic pair of knot traces $X_{-1}(-5_2)$ and $X_{-1}(P(3,-3,-8))$, an example first discovered by Akbulut. This gives the first analysis-free proof of the existence of exotic compact orientable $4$-manifolds. We also present a family of exotic knot traces that seem not directly recoverable from gauge/Floer-theoretic methods. Along the way, we present new explicit calculations of the Khovanov skein lasagna modules, and we define lasagna generalizations of the Lee homology and Rasmussen $s$-invariant, which are of independent interest. Other consequences of our work include a slice obstruction of knots in $4$-manifolds with nonvanishing skein lasagna module, a sharp shake genus bound for some knots from the lasagna $s$-invariant, and a construction of induced maps on Khovanov homology for cobordisms in $k\mathbb{CP}^2$.
\end{abstract}

\maketitle

\section{Introduction}
The revolutionary work of Donaldson and Freedman in the $1980$s revealed a drastic dichotomy between the topological and smooth categories of manifolds in dimension $4$. Since then, a central focus of smooth $4$–manifold topology has been the discovery of new exotic phenomena and the resulting refinement of our understanding of the geography of smooth $4$–manifolds, largely through the development of increasingly sophisticated smooth invariants.

A pair of smooth $4$-manifolds $X_1$ and $X_2$ is said to be \textit{exotic} if they are homeomorphic but not diffeomorphic. Under favorable conditions (e.g. trivial fundamental group), showing two $4$-manifolds are homeomorphic reduces to checking some algebraic topological invariants agree, thanks to Freedman's fundamental theorems \cite{freedman1982topology}. Showing that they are non-diffeomorphic usually involves showing some gauge-theoretic or Floer-theoretic smooth invariants (Donaldson invariants \cite{donaldson1990polynomial}, Seiberg-Witten invariants \cite{seiberg1994monopoles}, Heegaard Floer type invariants \cite{ozsvath2004holomorphic}, and various variations/generalizations) differ. These tools, powerful as they are, typically come with constraints on the topology of the $4$–manifolds (for instance, hypotheses on $b_2^+$ or related structural assumptions), and there is a sense in which their applicability to new phenomena is approaching its natural limits. To date, all known methods for detecting compact orientable exotic $4$-manifolds rely heavily on analysis, primarily elliptic PDEs.

The main result of this paper is the first analysis-free detection of compact orientable exotic $4$–manifolds.\footnote{For a noncompact example, an exotic $\R^4$ can be constructed from a topologically slice but not smoothly slice knot. Such knots can be detected by Freedman's result \cite{freedman1982topology} and Rasmussen's $s$-invariant \cite{rasmussen2010khovanov} from Khovanov homology. For nonorientable examples, Kreck \cite{kreck2006some} showed $K3\#\mathbb{RP}^4$ and $11(S^2\times S^2)\#\mathbb{RP}^4$ form an exotic pair using Rochlin's Theorem, and indeed exotic $\mathbb{RP}^4$'s were constructed by Cappell--Shaneson \cite{cappell1976some} (shown to be homeomorphic to $\mathbb{RP}^4$ later in \cite{hambleton1994nonorientable}).}\footnote{Exotic surfaces in $4$-manifolds were also known to be detectable by Khovanov homology \cite{hayden2021khovanov}.} The relevant smooth invariant in our story is the \textit{Khovanov-Rozansky $\mathfrak{gl}_2$ skein lasagna module} (Khovanov skein lasagna module for short) defined by Morrison--Walker--Wedrich \cite{morrison2022invariants}, which is an extension of a suitable variation of the classical Khovanov homology \cite{khovanov2000categorification} via a skein-theoretic construction. It assigns a tri-graded (by homological degree $h$, quantum degree $q$, and homology class degree $\alpha$) abelian group $$\mathcal{S}_{0,h,q}^2(X;L)=\bigoplus_{\alpha\in H_2(X,L)}\mathcal{S}_{0,h,q}^2(X;L;\alpha)$$ to every pair $(X,L)$ consisting of a compact oriented smooth $4$-manifold and a framed oriented link on its boundary. When $L$ is the empty link, we simply write $\mathcal{S}_0^2(X)$ for the skein lasagna module.

The $n$-trace on a knot $K$, denoted $X_n(K)$, is the $4$-manifold obtained by attaching an $n$-framed $2$-handle to $B^4$ along $K\subset S^3=\partial B^4$. Our main theorem shows that the $(-1)$-traces on the mirror image of the knot $5_2$ and the Pretzel knot $P(3,-3,-8)$, namely $$X_1:=X_{-1}(-5_2),\quad X_2:=X_{-1}(P(3,-3,-8)),$$ have non-isomorphic skein lasagna modules.

\begin{figure}
\centering
\raisebox{12pt}{\scalebox{.5}{\includegraphics{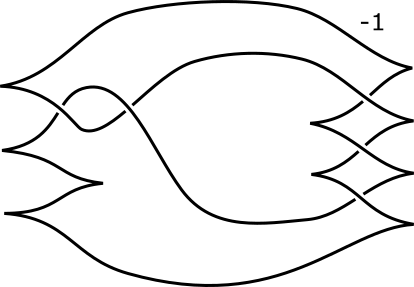}}}\quad\quad\quad
\scalebox{.5}{\includegraphics{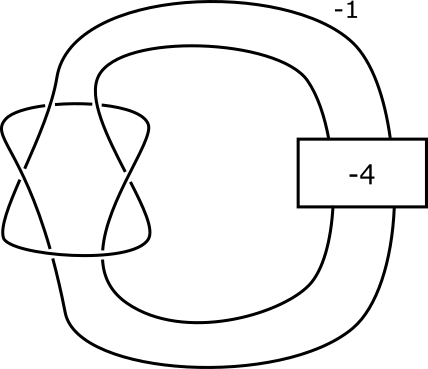}}
\caption{Kirby diagrams of $X_1,X_2$.}
\label{fig:kirby}
\end{figure}

\begin{Thm}\label{thm:exotic}
The graded modules $\mathcal{S}_0^2(X_1)$ and $\mathcal{S}_0^2(X_2)$ are non-isomorphic over $\Q$. Therefore, $X_1$ and $X_2$ form an exotic pair of $4$-manifolds. In fact, even their interiors are exotic.
\end{Thm}

The exotic pair $X_1,X_2$ was first discovered by Akbulut \cite{akbulut1991exotic,akbulut1991fake}. The fact that $X_1,X_2$ are homeomorphic is standard: a sequence of Kirby moves shows $\partial X_1$ and $\partial X_2$ are homeomorphic and a suitable extension of Freedman's Theorem \cite{freedman1982topology,boyer1986simply} then implies $X_1$ and $X_2$ are homeomorphic. On the other hand, the fact that $X_1,X_2$ are not diffeomorphic was detected by an intricate calculation of Donaldson invariants. Later, the result was reproved by Akbulut--Matveyev \cite[Theorem~3.2]{akbulut1997exotic} using a simplified argument which essentially says $X_{-1}(-5_2)$ is Stein (since the maximal Thurston-Bennequin number of $-5_2$ is $1$) while $X_{-1}(P(3,-3,-8))$ is not (since $P(3,-3,-8)$ is slice). This latter proof is elegant, yet still depends heavily on the calculation of Seiberg-Witten invariants of Stein surfaces, as well as deep results of Eliashberg \cite{eliashberg1990topological} and Lisca--Mati\'c \cite{lisca1997tight}. In comparison, our proof of Theorem~\ref{thm:exotic} depends on the much lighter package of Khovanov homology, which is combinatorial in nature.

Since the Khovanov skein lasagna module over a field of a connected sum or boundary connected sum of $4$-manifolds is equal to the tensor product of the Khovanov skein lasagna modules of the components, a fact proved in \cite{manolescu2022skein}, $X_1$ and $X_2$ remain exotic after (boundary) connected sums with any $4$-manifolds $W$ with nonvanishing skein lasagna module (at least when the support of $\mathcal S_0^2(W)$ is small in a suitable sense). Such tricks lead to many new exotic pairs. We write down some explicit examples in the following corollary to give the reader a flavor of what can be proven. For more choices of such candidates of connected summands, one may look into Section~\ref{sec:examples}.

\begin{Cor}\label{cor:more_exotic}
The $4$-manifolds $X_1^{\#a}\#X_2^{\#b}$ and $X_1^{\#a'}\#X_2^{\#b'}$ form an exotic pair if $a+b=a'+b'$ and $(a,b)\ne(a',b')$. The exotica remain under connected sums or boundary sums with any copies of $S^1\times S^3$, $S^2\times D^2$, $\overline{\CP^2}$, and the negative smooth $E8$ manifold. Moreover, the interiors of such pairs are also exotic.
\end{Cor}\vspace{-5pt}

Mukherjee \cite{mukherjeepersonal} informed us that most of these variations also appear to be achievable via gauge/Floer-theoretic methods and topological arguments. However, below we present a family of exotic knot traces generalizing Theorem~\ref{thm:exotic} that does not seem to be directly recoverable by other methods, in particular because Rasmussen's $s$-invariant $s(K)$ appears in the hypothesis of the theorem. Let $P_{n,m}$ and $Q_{n,m}$ be satellite patterns drawn as in Figure~\ref{fig:Yasui_patterns}, and $P_{n,m}(K)$, $Q_{n,m}(K)$ be the corresponding satellite knots with companion knot $K$. These satellite patterns were studied by Yasui \cite{yasui2015corks}.

\begin{Thm}\label{thm:exotic_family}
If a knot $K\subset S^3$ has a slice disk $\Sigma$ in $k\CP^2\backslash int(B^4)$ and $n,m$ are two integers satisfying $$s(K)=|[\Sigma]|-[\Sigma]^2,\ n<-[\Sigma]^2,\ m\ge0,$$ then $X_n(P_{n,m}(K))$ and $X_n(Q_{n,m}(K))$ form an exotic pair. Their interiors are also exotic.
\end{Thm}\vspace{-5pt}

In the special case when $K=U$ is the unknot, $k=0$, $\Sigma$ is the trivial slice disk of $K$ in $B^4$, and $n=-1$, $m=0$, Theorem~\ref{thm:exotic_family} recovers Theorem~\ref{thm:exotic}. One can construct more exotic pairs in the spirit of Corollary~\ref{cor:more_exotic} by taking connected sums of these examples and other standard manifolds appearing in that corollary. One may also deduce the existence of exotic pairs of compact contractible $4$-manifolds (see Remark~\ref{rmk:exotic_mazur}) from Theorem~\ref{thm:exotic_family}.\medskip

Now, we turn to some explicit calculations of $\mathcal{S}_0^2$ and other applications, some of which build towards a proof of Theorem~\ref{thm:exotic} presented in Section~\ref{sbsec:exotic_proof}. In the current section, many statements are presented in a specialized form. In such cases, we indicate the corresponding generalizations in parentheses after the statement numbers.

\textbf{Convention.} \textit{Throughout this paper, all $4$-manifolds are assumed to be smooth, compact, and oriented. All surfaces are oriented and smoothly and properly embedded.}

\subsection{A vanishing result and slice obstruction}
For a knot $K\subset S^3$, let $TB(K)$ denote its maximal Thurston-Bennequin number. We have the following vanishing criterion for the Khovanov skein lasagna module.

\begin{Thm}\label{thm:van}
If the knot trace $X_n(K)$ for some knot $K\subset S^3$ and framing $n\ge-TB(-K)$ embeds into a $4$-manifold $X$, then $\mathcal{S}_0^2(X;L)=0$ for any $L\subset\partial X$.
\end{Thm}\vspace{-5pt}

For example, since $S^2\times S^2$ contains an embedded sphere with self-intersection $2$, we see the $2$-trace of the unknot embeds into $S^2\times S^2$, implying $\mathcal S_0^2(S^2\times S^2)=0$. This answers a question of Manolescu (see \cite[Question~2]{sullivan2024kirby}), generalizing one main Theorem of Sullivan--Zhang \cite[Theorem~1.3]{sullivan2024kirby} and answering \cite[Question~3]{sullivan2024kirby} in the affirmative.\smallskip

The folklore trace embedding lemma says that $X_n(K)$ embeds into $X$ if and only if $-K$ is $(-n)$-slice in $X$, i.e. there is a framed slice disk in $X\backslash int(B^4)$ of the $(-n)$-framed knot $-K\subset S^3=-\partial B^4$. Therefore, the contrapositive of Theorem~\ref{thm:van} can be phrased as follows.

\begin{Cor}\label{cor:slice_obstruction}
Suppose $\mathcal{S}_0^2(X;L)\ne0$ for some $L\subset \partial X$. If a knot $K\subset S^3$ is $n$-slice in $X$, then $n>TB(K)$.\qed
\end{Cor}

\subsection{Lee skein lasagna modules and lasagna \texorpdfstring{$s$}{2}-invariants}\label{sbsec:intro_lee_lasagna}
In the construction of the Khovanov skein lasagna module, the relevant link homology theory (the Khovanov-Rozansky $\mathfrak{gl}_2$ homology) enters only at the very last step. We can replace this link homology theory with its Lee deformation \cite{lee2005endomorphism}, and define an analogous notion of a Lee skein lasagna module of a $4$-manifold $X$ with link $L\subset\partial X$, denoted $\mathcal{S}_0^{Lee}(X;L)$. It is a $\Q$-vector space equipped with a homological $\Z$-grading, a quantum $\Z/4$-grading, a homology class grading by $H_2(X,L)$, and a quantum $\Z$-filtration (where $0$ has filtration degree $-\infty$), with the caveat that a nonzero vector can have quantum filtration degree $-\infty$.

We show in Section~\ref{sec:Lee_lasagna} that the structure of $\mathcal{S}_0^{Lee}(X;L)$, except for the quantum filtration, is determined by simple algebraic data of $(X,L)$. Moreover, we extract numerical invariants from the quantum filtration structure\footnote{The authors are aware that the recent paper \cite{morrison2024invariants} by Morrison--Walker--Wedrich has an independent proof of a part of this structural result, stated in a more general form. They extract slightly different invariants from the filtration structure, which gives somewhat less information than ours in the $\mathfrak{gl}_2$ case. Moreover, our invariants generalize Rasmussen $s$-invariants for links and enjoy nice functorial properties.}. For the purpose of exposition, in the current section we assume $L$ is the empty link. Let $q\colon\mathcal S_0^{Lee}(X)\to\Z\cup\{-\infty\}$ denote the quantum filtration function.

\begin{Thm}\label{thm:intro_Lee_structure}(Theorem~\ref{thm:Lee_structure})
The Lee skein lasagna module of a $4$-manifold $X$ is $\mathcal S_0^{Lee}(X)\cong\Q^{H_2(X)^2},$ with a basis consisting of canonical generators $x_{\alpha_+,\alpha_-}$, one for each pair $\alpha_+,\alpha_-\in H_2(X)$.
\end{Thm}
In fact, the isomorphism $\mathcal S_0^{Lee}(X)\cong\Q^{H_2(X)^2}$ is as $(\Z\times\Z/4\times H_2(X,L))$-graded vector spaces; see Theorem~\ref{thm:Lee_structure} for grading information.

\begin{Def}(Definition~\ref{def:s})\label{def:intro_s}
The \textit{lasagna $s$-invariant} of $X$ at a homology class $\alpha\in H_2(X)$ is $$s(X;\alpha):=q(x_{\alpha,0}).$$
\end{Def}

The lasagna $s$-invariants of $X$ thus define a smooth invariant $s(X;\cdot):H_2(X)\to\Z\cup\{-\infty\}$.

The genus function of a $4$-manifold $X$, defined by
\[g(X,\alpha):=\min\{g(\Sigma)\colon\Sigma\subset X,[\Sigma]=\alpha\},\]
captures subtle information about the smooth topology of $X$.  Like the classical $s$-invariant \cite{rasmussen2010khovanov,beliakova2008categorification}, which provides a lower bound on the slice genus of links, the lasagna $s$-invariants provide a lower bound on the genus functions of $4$-manifolds.

\begin{Thm}\label{thm:intro_genus_bound}(Theorem~\ref{thm:s_prop}(6) and Corollaries after)
The genus function of a $4$-manifold $X$ has a lower bound given by $$g(X;\alpha)\ge\frac{s(X;\alpha)+\alpha^2}2.$$
\end{Thm}
Here, $\alpha^2=\alpha\cdot\alpha$ is computed via the intersection form on $H_2(X)$.

Unlike the classical $s$-invariant, the lasagna $s$-invariants are not directly computable from link diagrams. However, we are able to compute them in some special cases. We indicate one special case where Theorem~\ref{thm:intro_genus_bound} does provide useful information.

The \textit{$n$-shake genus} of a knot $K$ is defined by $g_{sh}^n(K):=g(X_n(K);1)$, where $1$ is a generator of $H_2(X_n(K))\cong\Z$. Clearly, $g_{sh}^n(K)\le g_4(K)$, but the inequality may not be strict \cite{akbulut19772,piccirillo2019shake}. Theorem~\ref{thm:intro_genus_bound} specializes to a shake genus bound 
\begin{equation}\label{eq:shake_genus_bound}
g_{sh}^n(K)\ge\frac{s(X_n(K);1)+n}2.
\end{equation}

\begin{Thm}\label{thm:shake_genus}
\begin{enumerate}[(1)]
\item If $K$ is concordant to a positive knot, then for $n\leq 0$ we have $s(X_n(K);1)=s(K)-n,$ where $s(K)$ is the Rasmussen $s$-invariant of $K$. Thus $$g_{sh}^n(K)=g_4(K)=\frac{s(K)}2,\ n\le0.$$
\item If $\Sigma\subset k\overline{\CP^2}\backslash(int(B^4)\sqcup int(B^4))$ is a concordance from $K$ to a positive knot $K'$, then for $n\le[\Sigma]^2$ we have $s(X_n(K);1)\ge s(K')+[\Sigma]^2+|[\Sigma]|-n$. Thus $$g_{sh}^n(K)\ge\frac{s(K')+[\Sigma]^2+|[\Sigma]|}2,\ n\le[\Sigma]^2.$$
\end{enumerate}
\end{Thm}

\begin{Rmk}
Theorem~\ref{thm:shake_genus}(1) gives a refinement of Rasmussen \cite[Theorem~4]{rasmussen2010khovanov} which was used to provide the first analysis-free proof of the Milnor conjecture about the slice genus of torus knots. The computation of shake genus in Theorem~\ref{thm:shake_genus}(1) can be also done using Seiberg-Witten theory:\vspace{-5pt}
\begin{enumerate}[(i)]
\item The adjunction inequality for Stein manifolds \cite{lisca1997tight} implies \begin{equation}\label{eq:Stein_shake_genus}
g_{sh}^n(K)\ge(tb(\mathcal K)+|rot(\mathcal K)|+1)/2
\end{equation}
for any Legendrian representative $\mathcal K$ of $K$ with $tb(\mathcal K)=n+1$. When $n<TB(K)$, take $\mathcal K$ to be a stabilization that maximizes $|rot|$, of a max tb representative of $K$. Then the right-hand side of \eqref{eq:Stein_shake_genus} agrees with $s(K)/2$ for positive knots \cite[Theorem~2]{tanaka1999maximal}. For example, Figure~\ref{fig:kirby} (left) shows a Legendrian representative of $-5_2$ with $tb=0$ and $|rot|=1$, which implies that $g_{sh}^{-1}(-5_2)\ge1$. This is used in the reproof of Theorem~\ref{thm:exotic} in \cite{akbulut1997exotic}; see also Section~\ref{sbsbsec:proof_2}.
\item Two concordant knots have the same shake genus (by an easy topological argument).
\end{enumerate}
\end{Rmk}

\subsection{Comparison results for \texorpdfstring{$2$}{2}-handlebodies}\label{sbsec:intro_comparison}
The Khovanov homology of a link admits a spectral sequence to its Lee homology. In particular, in every bidegree, the rank of the Khovanov homology is bounded below by the dimension of the associated graded vector space of the Lee homology. Although we are not able to define a spectral sequence from $\mathcal S_0^2$ to $\mathcal S_0^{Lee}$, we do have a rank inequality when $X$ is a $2$-handlebody. Let $gr_\bullet V$ denote the associated graded vector space of a filtered vector space $V$.
\begin{Thm}\label{thm:rank_ineq}
If $X$ is a $2$-handlebody, then in any tri-degree $(h,q,\alpha)\in\Z^2\times H_2(X,L)$, $$\mathrm{rank}(\mathcal S_{0,h,q}^2(X;L;\alpha))\ge\dim(gr_q(\mathcal S_{0,h}^{Lee}(X;L;\alpha))).$$
In particular, if the quantum filtration on $\mathcal{S}_0^{Lee}(X;L)$ is not identically $-\infty$, then $\mathcal{S}_0^2(X;L)\ne0$.
\end{Thm}

We state a comparison result for Lee skein lasagna modules of $2$-handlebodies, which can sometimes be used as a nonvanishing criterion for Khovanov skein lasagna modules in view of Theorem~\ref{thm:rank_ineq}; see e.g. Section~\ref{sbsbsec:proof_1}.

\begin{Prop}\label{prop:intro_Lee_comparison}(Proposition~\ref{prop:Lee_comparison})
Suppose $X$ (resp. $X'$) is a $2$-handlebody obtained by attaching $2$-handles to $B^4$ along a framed oriented link $K\subset S^3$ (resp. $K'\subset S^3$). If $K$ is framed concordant to $K'$, then $\mathcal{S}_0^{Lee}(X)\cong\mathcal{S}_0^{Lee}(X')$ as tri-graded, quantum filtered vector spaces. Under the identification $H_2(X)\cong H_2(X')$ induced by a link concordance, $X$ and $X'$ have equal lasagna $s$-invariants.
\end{Prop}

\subsection{A nonvanishing result for \texorpdfstring{$2$}{2}-handlebodies from link diagrams}\label{sbsec:intro_nonvan_diagram}
Using a diagrammatic argument, in Section~\ref{sbsec:diagram}, we prove a nonvanishing result for $2$-handlebodies with sufficiently negative framings. For ease of exposition, here we only state the case for knot traces. Let $n_-(K)$ be the minimal number of negative crossings among all diagrams of a knot $K$.
\begin{Thm}\label{thm:nonvan_diagram_knot}(Theorem~\ref{thm:nonvan_diagram})
Let $K$ be a knot and $X=X_n(K)$ be its $n$-trace.\vspace{-5pt}
\begin{enumerate}[(1)]
\item If $n<-2n_-(K)$, then $s(X;0)=0$ and $\mathcal S_{0,0,q}^2(X;0)\otimes\Q=gr_q(\mathcal S_{0,0}^{Lee}(X;0))=\begin{cases}\Q,&q=0\\0,&q\ne0.\end{cases}$
\item If $K$ is positive, then for $n\le0$, $s(X;0)=0$, $s(X;1)=s(K)-n$. For $n<0$,
$$\mathcal S_{0,0,q}^2(X;1)\otimes\Q=gr_q(\mathcal S_{0,0}^{Lee}(X;1))=\begin{cases}\Q,&q=s(X;1),s(X;1)-2\\0,&\text{otherwise}.\end{cases}$$
\end{enumerate}
\end{Thm}

\subsection{Proofs of Theorem~\ref{thm:exotic}}\label{sbsec:exotic_proof}
We give two proofs that $int(X_1)$ and $int(X_2)$ are not diffeomorphic. The first proof distinguishes them on the level of Khovanov/Lee skein lasagna modules, while the second proof distinguishes their smooth genus functions. The following facts will be used:
\begin{itemize}
\item $-5_2$ is a positive knot, with $s(-5_2)=2$.
\item $P(3,-3,-8)$ is a slice knot.
\end{itemize}
\subsubsection{First proof}\label{sbsbsec:proof_1}
Applying Theorem~\ref{thm:nonvan_diagram_knot}(2) to $-5_2$, we obtain $s(X_1;1)=3$, $$\mathcal S_{0,0,q}^2(X_1;1)\otimes\Q=gr_q(\mathcal S_{0,0}^{Lee}(X_1;1))=\begin{cases}\Q,&q=1,3\\0,&\text{otherwise}.\end{cases}$$
Applying Theorem~\ref{thm:nonvan_diagram_knot}(2) to the unknot, and then Proposition~\ref{prop:intro_Lee_comparison} to $P(3,-3,-8)$ and the unknot, we obtain $s(X_2;1)=1$, and
\[\mathcal S_{0,0,q}^2(X_2;1)\otimes\Q\supset gr_q(\mathcal S_{0,0}^{Lee}(X_2;1))=\begin{cases}\Q,&q=\pm1\\0,&\text{otherwise}.\end{cases}\]
where the first containment follows from Theorem~\ref{thm:rank_ineq}. In particular, $\mathcal S_{0,0,-1}^2(X_1;1)\otimes\Q\ne\mathcal S_{0,0,-1}^2(X_2;1)\otimes\Q$, hence $X_1$ and $X_2$ are not diffeomorphic. In fact, since skein lasagna modules are smooth invariants of interiors of $4$-manifold (cf. Lemma~\ref{lem:interior_invariant}), $int(X_1)$ and $int(X_2)$ are not diffeomorphic.\qed

\subsubsection{Second proof}\label{sbsbsec:proof_2}
Applying Theorem~\ref{thm:shake_genus}(1) to $-5_2$, we obtain $g_{sh}^{-1}(-5_2)=1$. On the other hand, capping off a slice disk of $P(3,-3,-8)$ by the core of the $2$-handle in $X_{-1}(P(3,-3,-8))$ shows that $g_{sh}^{-1}(P(3,-3,8))=0$. This proves $int(X_1)$ and $int(X_2)$ are non-diffeomorphic, without actually calculating $\mathcal S_0^\bullet(X_i)$.\qed

\subsection{Example: Khovanov and Lee skein lasagna module of \texorpdfstring{$\overline{\CP^2}$}{-CP2}}\label{sbsec:intro_CP2}
As an example, we note that the recent work by the first author \cite{ren2023lee} has the following implications for the Lee and Khovanov skein lasagna modules for $\overline{\CP^2}$. Note the products $\alpha^2,\beta^2$ below utilize the intersection form on $H_2(\overline{\CP^2})$.

\begin{Prop}\label{prop:-CP^2}
The lasagna $s$-invariants of $\overline{\CP^2}$ are given by $s(\overline{\CP^2};\alpha)=|\alpha|$. The associated graded vector space of the Lee skein lasagna module of $\overline{\CP^2}$ is given by $gr_q\mathcal S_{0,(\beta^2-\alpha^2)/2}^{Lee}(\overline{\CP^2};\alpha)=\Q$ for $$\beta\equiv\alpha\pmod2,\ q=\begin{cases}\frac{\alpha^2-\beta^2}2+|\beta|-1\pm1,&\beta\ne0\\\frac{\alpha^2}2,&\beta=0,\end{cases}$$ and $0$ in other tri-degrees. Moreover, for $\beta\equiv\alpha\pmod2$, $$\mathcal{S}_{0,(\beta^2-\alpha^2)/2,(\alpha^2-\beta^2)/2+|\beta|}^2(\overline{\CP^2};\alpha)=\Z.$$
\end{Prop}

In fact, in Conjecture~\ref{conj:-CP2}, we formulate an expected complete determination of $\mathcal S_0^2(\overline{\CP^2};\Q)$.

A consequence of Proposition~\ref{prop:-CP^2} is the following.
\begin{Cor}\label{cor:emd_-kCP2}
If a $4$-manifold $X$ admits an embedding $i\colon X\hookrightarrow k\overline{\CP^2}$ for some $k$, then $s(X;\alpha)\ge|i_*\alpha|$, with equality achieved if $\alpha=0$. If in addition, $X$ is a $2$-handlebody, then $\mathcal{S}_{0,0,q}^2(X;\alpha)\ne0$ for $q=s(X;\alpha)$, and for $q=s(X;\alpha)-2$ if $\alpha\ne0\in H_2(X)$.
\end{Cor}

If $\Sigma\subset k\CP^2\backslash(int(B^4)\sqcup int(B^4))$ is an oriented cobordism in $k\CP^2$ between two oriented links $L_0,L_1$, the last part of Proposition~\ref{prop:-CP^2} enables us to define an induced map $$Kh(\Sigma)\colon Kh(L_0)\to Kh(L_1)$$ between ordinary Khovanov homology of $L_0$ and $L_1$, at least over a field and up to sign, with bidegree $(0,\chi(\Sigma)-[\Sigma]^2+|[\Sigma]|)$. We carry out this construction in more detail in Section~\ref{sbsec:Kh_CP2_cob}.

\subsection{Organization of the paper}
In Section~\ref{sec:Kh_lasagna}, we review the theory of Khovanov skein lasagna modules, expanding on a few properties, as well as making some new comments on the skein-theoretic construction. In the short Section~\ref{sec:van}, we prove Theorem~\ref{thm:van} by applying a bound on Khovanov homology by Ng \cite{ng2005legendrian}, and we give some examples of $4$-manifolds with vanishing Khovanov skein lasagna module.

In Section~\ref{sec:Lee_lasagna}, we develop the theory of Lee skein lasagna modules and lasagna $s$-invariants, which can be thought of as a lasagna generalization of the classical theory by Lee \cite{lee2005endomorphism} and Rasmussen \cite{rasmussen2010khovanov} for links in $S^3$. A generalization of Theorem~\ref{thm:intro_Lee_structure}, in particular the definition of canonical Lee lasagna generators, is presented in Section~\ref{sbsec:Lee_structure}. The behavior of canonical Lee lasagna generators under morphisms is given in Theorem~\ref{thm:can_gen_morphism} in Section~\ref{sbsec:can_Lee_gen_mor}. Building on these results, in Section~\ref{sbsec:las_s} we define and examine the lasagna $s$-invariants, generalizing Definition~\ref{def:intro_s} and Theorem~\ref{thm:intro_genus_bound}. We have also included a discussion about the coefficient field $\Q$ at the end of Section~\ref{sec:Lee_lasagna}.

In Section~\ref{sec:2-hdby}, we make use of a $2$-handlebody formula by Manolescu--Neithalath \cite{manolescu2022skein} to prove nonvanishing results for Khovanov and Lee skein lasagna modules of $2$-handlebodies. In Section~\ref{sbsec:rank_ineq}, we prove Theorem~\ref{thm:rank_ineq} which shows that finiteness results for the quantum filtration of Lee skein lasagna modules imply nonvanishing results for Khovanov skein lasagna modules. In Section~\ref{sbsec:Lee_compare} we prove a generalization of Proposition~\ref{prop:intro_Lee_comparison} and in Section~\ref{sbsec:las_s_from_classical} we show that lasagna $s$-invariants can be expressed in terms of classical $s$-invariants of some cable links. Finally, in the somewhat technical Section~\ref{sbsec:diagram}, we prove a vast generalization of the diagrammatic nonvanishing criterion Theorem~\ref{thm:nonvan_diagram_knot}.

Finally, in Section~\ref{sec:examples}, we give examples and applications, mostly to the theory and tools developed in Section~\ref{sec:Lee_lasagna} and Section~\ref{sec:2-hdby}. This includes proofs of Corollary~\ref{cor:more_exotic}, Theorem~\ref{thm:exotic_family}, Theorem~\ref{thm:shake_genus}, and an expanded discussion of Section~\ref{sbsec:intro_CP2}. Other discussions include a proof that $s$-invariants of cables of the Conway knot can obstruct its sliceness, that the Khovanov skein lasagna module over $\Q$ does not detect potential exotica arising from Gluck twists in $S^4$, and a conjecture about skein lasagna modules of Stein $2$-handlebodies.

Most of our developments are independent of proving Theorem~\ref{thm:exotic}. Readers only interested in Theorem~\ref{thm:exotic} may want to look at its proof in Section~\ref{sbsec:exotic_proof} and delve into the corresponding sections. In particular, to distinguish the smooth structures of the pair $X_1,X_2$ in Theorem~\ref{thm:exotic}, one can read the first half of Section~\ref{sec:Kh_lasagna}, certain statements in Section~\ref{sec:Lee_lasagna} (mainly Definition~\ref{def:s} and Corollary~\ref{cor:genus_bound_simpliest}), Section~\ref{sbsec:diagram}, and Section~\ref{sbsec:shake_genus}. The minimal amount of theory needed to actually distinguish their Khovanov skein lasagna modules depends moreover on Section~\ref{sbsec:rank_ineq} and Section~\ref{sbsec:Lee_compare} (but not really on Corollary~\ref{cor:genus_bound_simpliest} and Section~\ref{sbsec:shake_genus}). Still, the statements in those sections are often much more general than needed for Theorem~\ref{thm:exotic}, which at times complicate the proofs.

\pgfdeclarelayer{background}
\pgfsetlayers{background,main}

\begin{figure}[h]
\centering
\begin{tikzpicture}[
  >=Latex,
  x=1.5cm, y=1.5cm, 
  result/.style={rectangle,draw,rounded corners,
                 align=center,inner sep=3pt},
  dep/.style={-Latex,thick},
  depdashed/.style={-Latex,thick,dashed},
  transform shape
]

\node[result] (sec24) at (0,5.4)  {\S\ref{sbsec:skein_construction}--\ref{sbsec:Kh_lasagna_properties}};
\node[result] (sec25) at (3.8,4.5) {\S\ref{sbsec:skein_structure}};
\node[result] (sec26) at (3.8,3.5) {\S\ref{sbsec:variation}};
\node[result] (sec3)  at (-3.8,4.2) {\S\ref{sec:van}\\Thm~\ref{thm:van}};

\node[result, text width=5.2cm] (sec4) at (0,4) {\S\ref{sec:Lee_lasagna}\\[2pt]
Thm~\ref{thm:Lee_structure}(\ref{thm:intro_Lee_structure}) Def~\ref{def:s}(\ref{def:intro_s}) Thm~\ref{thm:s_prop}(\ref{thm:intro_genus_bound})};

\node[result] (sec51) at (-4.0,2.4) {\S\ref{sbsec:rank_ineq}\\Thm~\ref{thm:rank_ineq}};
\node[result] (sec54) at (-1.3,2.4) {\S\ref{sbsec:diagram}\\Thm~\ref{thm:nonvan_diagram}(\ref{thm:nonvan_diagram_knot})};
\node[result] (sec52) at (1.4,2.4)  {\S\ref{sbsec:Lee_compare}\\Prop~\ref{prop:Lee_comparison}(\ref{prop:intro_Lee_comparison})};
\node[result] (sec53) at (3.8,2.4)  {\S\ref{sbsec:las_s_from_classical}\\Prop~\ref{prop:las_s_from_classical}};

\node[result] (sec61) at (-2.8,1.1) {\S\ref{sbsec:example_traces}};
\node[result] (sec62) at (-4,1.1) {\S\ref{sbsec:example_plumbing}};
\node[result] (sec64) at (2.1,1.1)  {\S\ref{sbsec:shake_genus}\\Thm~\ref{thm:shake_genus}(1)};
\node[result] (sec65) at (-0.5,1.1)  {\S\ref{sbsec:comparison_refined}\\Thm~\ref{thm:shake_genus}(2)\\Prop~\ref{prop:las_s_compare_-CP2}};
\node[result] (sec67) at (4,0.9)  {\S\ref{sbsec:conway}};

\node[result] (cor12) at (-3.1,-0.25)  {\S\ref{sbsec:more_exotica}\\Cor~\ref{cor:more_exotic}};
\node[result] (thm11) at (-1.5,-0.25)  {Thm~\ref{thm:exotic}};
\node[result] (thm13) at (0.1,-0.25){\S\ref{sbsec:yasui}\\Thm~\ref{thm:exotic_family}};

\node[result] (sec69)  at (-2.3,-1.5) {\S\ref{sbsec:-CP^2}\\Prop~\ref{prop:-CP^2}\\Cor~\ref{cor:emd_-kCP2}};
\node[result] (sec610) at (-4,-1.75) {\S\ref{sbsec:htp_S4}};
\node[result] (sec611) at (-0.3,-1.64) {\S\ref{sbsec:Kh_CP2_cob}\\Prop~\ref{prop:Kh_CP2_cob}}; 
\node[result] (sec68)  at (1.6,-1.4)  {\S\ref{sbsec:torus_traces}\\Prop~\ref{prop:torus_knot_trace_Lee}};
\node[result] (sec612) at (3.6,-1.3)  {\S\ref{sbsec:stein}};


\draw[dep] (sec24) -- (sec3);
\draw[dep] (sec24) -- (sec4);
\draw[dep] (sec24) -- (sec25);

\draw[dep] (sec25) -- (sec26);
\draw[depdashed] (sec25) -- (sec4);

\draw[dep, bend left=10] (sec3) to (sec61);

\draw[dep] (sec4) -- (sec54);
\draw[dep] (sec4) -- (sec53.150);
\draw[dep] (sec4) -- (sec52);
\draw[dep] (sec4) -- (sec51);

\draw[depdashed] (sec51) to (sec54);
\draw[depdashed, bend right=40] (sec51.215) to (sec69.165);
\draw[dep, bend left=30] (sec51) to (thm11);

\draw[dep] (sec52) -- (sec64);
\draw[depdashed] (sec52) -- (sec65);
\draw[dep] (sec52) -- (thm13);

\draw[dep] (sec53) -- (sec67);
\draw[dep] (sec53) to (sec68);
\draw[dep] (sec53) -- (sec65);
\draw[depdashed, bend right=20] (sec53) to (sec54);

\draw[dep] (sec54) -- (sec61);
\draw[dep] (sec54) -- (sec62);
\draw[dep] (sec54) -- (sec64);
\draw[dep] (sec54) -- (sec65);
\draw[dep, bend right=5] (sec54) to (thm11);

\draw[dep] (sec61) -- (cor12);
\draw[dep] (sec62) -- (cor12);

\draw[dep] (thm11) -- (cor12);
\draw[dep] (thm11) -- (thm13);

\draw[dep] (sec64) -- (thm11);
\draw[dep] (sec65) -- (thm13);
\draw[dep] (sec65) -- (sec61);

\draw[dep] (sec69) -- (sec610);
\draw[dep] (sec69) -- (sec611);

\draw[dep, bend right=10] (sec68) to (sec69);

\draw[dep, bend right=10] (sec69) to (sec61);

\draw[depdashed] (sec53.250) -- (sec612);
\draw[depdashed, bend left=5] (sec54.340) to (sec612);
\draw[depdashed] (sec64) -- (sec612);

\begin{pgfonlayer}{background}
  \node[
    draw,
    rounded corners,
    inner sep=8pt,
    fill=black!5,
    fit=(cor12) (thm11) (thm13),
  ] (group_main_results) {};
\end{pgfonlayer}

\begin{pgfonlayer}{background}
  \node[
    draw,
    rounded corners,
    inner sep=8pt,
    fill=black!5,
    fit=(sec61) (sec62),
  ] (group_main_results) {};
\end{pgfonlayer}

\begin{pgfonlayer}{background}
  \node[
    draw,
    rounded corners,
    inner sep=8pt,
    fill=black!5,
    fit=(sec64) (sec65),
  ] (group_main_results) {};
\end{pgfonlayer}

\begin{pgfonlayer}{background}
  \node[
    draw,
    rounded corners,
    inner sep=8pt,
    fill=black!5,
    fit=(sec68) (sec69) (sec610) (sec611),
  ] (group_main_results) {};
\end{pgfonlayer}

\end{tikzpicture}
\caption{Dependency graph. Arrows indicate logical dependence, and dashed arrows indicate weak logical dependence. Applications with a similar flavor are boxed together.}
\end{figure}

\subsection*{Acknowledgement}
We thank Ian Agol, Daren Chen, Kyle Hayden, Sungkyung Kang, Ciprian Manolescu, Marco Marengon, Anubhav Mukherjee, Gheehyun Nahm, Ikshu Neithalath, Lisa Piccirillo, Mark Powell, Yikai Teng, Joshua Wang, Paul Wedrich, and Hongjian Yang for helpful discussions. We thank the referees for their careful reading of the paper and their helpful suggestions. This work was supported in part by the Simons grant ``new structures in low-dimensional topology'' through the workshop on exotic $4$-manifolds held at Stanford University December 16--18, 2023.

\section{Khovanov and Lee skein lasagna modules}\label{sec:Kh_lasagna}
The Khovanov-Rozansky $\mathfrak{gl}_2$ skein lasagna modules were defined by Morrison--Walker--Wedrich \cite{morrison2022invariants}. In the first part of this section we review the construction, cast in a slightly different way, while defining the Lee deformation under the same framework. In Section~\ref{sbsec:skein_construction}, we begin with a general skein-theoretic construction of skein lasagna modules, using an arbitrary link homology theory for links in $S^3$. In Section~\ref{sbsec:lasagna_properties}, we state a few properties of skein lasagna modules, most of which are essentially proved in \cite{manolescu2022skein}. Only then, in Section~\ref{sbsec:KhR_2}, do we fix the link homology theory under consideration to be either the $\mathfrak{gl}_2$ Khovanov-Rozansky homology $KhR_2$ \cite{morrison2022invariants} or its Lee deformation $KhR_{Lee}$ \cite{lee2005endomorphism}. We state a few more computational results for these settings in Section~\ref{sbsec:Kh_lasagna_properties}.

The second part of the section is new, and can be skipped on a first reading. In Section~\ref{sbsec:skein_structure}, we identify the structure of the underlying space of skeins in the construction of Section~\ref{sbsec:skein_construction}. The proof will eventually be useful in Section~\ref{sbsec:Lee_structure} where we calculate the Lee skein lasagna module of a pair $(X,L)$. Finally, in Section~\ref{sbsec:variation}, we suggest a few variations of the skein-theoretic construction.

\subsection{Construction of skein lasagna modules}\label{sbsec:skein_construction}\footnote{The formulation of our construction was motivated by a conversation with Sungkyung Kang.}
Let $X$ be a $4$-manifold and $L\subset\partial X$ be a framed oriented link on its boundary. A \textit{skein} in $(X,L)$ or in $X$ rel $L$ is a properly embedded framed oriented surface in $X$ with the interiors of finitely many disjoint $4$-balls deleted, whose boundary on $\partial X$ is $L$. That is, a skein in $(X,L)$ can be written as
\begin{equation}\label{eq:skein}
\Sigma\subset X\backslash{\sqcup_{i=1}^kint(B_i)},\text{ with }\partial\Sigma|_{\partial X}=L,
\end{equation}
The deleted (unparametrized) balls $B_i$ are called \textit{input balls} of the skein $\Sigma$, and the framed links $K_i\:= \Sigma \cap \partial B_i$ in (unparametrized) 3-spheres are called \emph{input links} of $\Sigma$.

We define the \textit{category of skeins} in $(X,L)$, denoted $\mathcal{C}(X;L)$, as follows. The objects consist of all skeins in $X$ rel $L$. Suppose $\Sigma$ is an object as in \eqref{eq:skein} and $\Sigma'$ is another object with input balls $B_j'$, $j=1,\cdots,k'$. If $\sqcup B_i\subset\sqcup B_j'$, and each $B_j'$ either coincides with some $B_i$ or does not intersect any $B_i$ on its boundary, then the set of morphisms from $\Sigma$ to $\Sigma'$ consists of pairs of an isotopy class of surfaces\footnote{We formally allow some components of $S$ to be curves on the boundaries of those $B_i=B_j'$.} $S\subset\sqcup B_j'\backslash\sqcup\,int(B_i)$ rel boundary with $S|_{\partial B_i}=\Sigma|_{\partial B_i}$, $S|_{\partial B_j'}=\Sigma'|_{\partial B_j'}$, and an isotopy class of isotopies from $S\cup\Sigma'$\footnote{In order for the surfaces to glue smoothly, one may assume throughout that all embedded surfaces have a standard collar near each boundary, and all isotopies rel boundary respect these collars.} to $\Sigma$ rel boundary. We will suppress the isotopy and write $[S]\colon\Sigma\to\Sigma'$ for such a morphism. If the input balls $B_i,B_j'$ do not satisfy the condition above, there is no morphism from $\Sigma$ to $\Sigma'$. The identity morphism at $\Sigma$ is given by $[\Sigma|_{\sqcup\partial B_i}]$. The composition of two morphisms $[S]\colon\Sigma\to\Sigma'$ and $[S']\colon\Sigma'\to\Sigma''$ is $[S\cup S']$.

Next, we fix a commutative ground ring $R$ and introduce a functorial link homology theory $Z$ valued in $R$-modules for framed oriented links in $\R^3$, which satisfies the conditions of Theorem~2.1 of \cite{morrison2024invariants}. Explicitly,
\begin{itemize}
\item $Z$ is lax monoidal in the sense that there are canonical maps $\iota_{L_0,L_1}\colon Z(L_0)\otimes_RZ(L_1)\to Z(L_0\sqcup L_1)$ for links $L_0,L_1\subset\R^3$. In our later discussions, $Z$ will be strictly monoidal, meaning that the canonical maps $\iota_{L_0,L_1}$ are isomorphisms.
\item $Z$ induces the identity map for the sweep-around move \cite[(1-1)]{morrison2022invariants}, allowing it to extend to a functorial link homology theory $Z$ for links in $S^3$. That is to say, for a framed oriented link $L\subset S^3$ we have an $R$-module $Z(L)$, and for a framed oriented link cobordism $S\subset S^3\times I$ from $L_0$ to $L_1$ we have an $R$-module map $Z(S)\colon Z(L_0)\to Z(L_1)$ invariant under isotopies of $S$ rel boundary.
\item The trace of a $2\pi$-rotation on $\R^3$ induces the identity map $Z(L)\to Z(L)$ for any link $L\subset\R^3$. This, together with the triviality of the sweep-around move, allows one to upgrade $Z$ to a homology theory for framed oriented links in \textit{abstract} $S^3$ and link cobordisms in \textit{abstract} $S^3\times I$. Here, an abstract $S^3$ is an oriented $3$-manifold diffeomorphic to $S^3$, and an abstract $S^3\times I$ is an oriented $4$-manifold diffeomorphic to $S^3\times I$, with a chosen order of its two boundary components.
\end{itemize}

We then define a functor (also denoted $Z$) on the category of skeins $\mathcal{C}(X;L)$ as follows.  For an object $\Sigma$ of $\mathcal{C}(X;L)$ having $k$ input balls $B_i$ we define
\begin{equation}\label{eq:Z(Sigma)}
Z(\Sigma):=\bigotimes_{i=1}^k Z(K_i)
\end{equation}
where the $K_i$ are the input links of $\Sigma$.  Now let $[S]\colon \Sigma \to \Sigma'$ be a morphism in $\mathcal{C}(X;L)$ represented by the surface $S$. For each `outer' input ball $B'_j$ of $\Sigma'$ with input link $K_j'$, we define a map
\[Z(S)_j:\bigotimes_i Z(K_i) \to Z(K'_j),\]
where the index $i$ runs over the `inner' input balls $B_i$ of $\Sigma$ that are contained in $B'_j$.  First, if $B'_j$ coincides with some $B_i$, we define $Z(S)_j$ to be the identity.  Otherwise, we have a surface $S_j=S\cap B'_j$ and choose paths within $B'_j\backslash S_j$ to connect all of the inner $B_i$ within $B'_j$ to form one abstract input ball $\natural_iB_i$ with input link $\sqcup_i K_i$, and view $S_j$ as a cobordism within an abstract $S^3\times I$ from $\sqcup_i K_i$ to $K'_j$, to define $Z(S)_j = Z(S_j)\circ \iota$, where $\iota$ is defined as in the first bullet point above\footnote{The triviality of the sweep-around move ensures that this definition is independent of the choice of paths connecting the inner input balls.}.  We then collect all of the tensor factors together to define
\[Z([S]) = \bigotimes_j Z(S)_j \colon Z(\Sigma) \to Z(\Sigma').\]

\begin{Def}\label{def:skein lasagna module}
Let $X$ be a 4-manifold and $L\subset \partial X$, and fix a functorial link homology theory $Z$ as above.  Then the \emph{skein lasagna module} of $(X,L)$ with respect to $Z$, denoted $\mathcal{S}_0^Z(X;L)$, is the colimit
\begin{equation}\label{eq:colim}
\mathcal{S}_0^Z(X;L):=\mathrm{colim}(Z\colon\mathcal C(X;L)\to\mathbf{Mod}_R).
\end{equation}
\end{Def}

More explicitly,
\begin{equation}\label{eq:S_0^Z}
\mathcal S_0^Z(X;L)=\{(\Sigma,v)\colon\Sigma\subset X\backslash\sqcup_{i=1}^kint(B_i),\ v\in Z(\Sigma)\}/\sim
\end{equation}
where $\sim$ is the equivalence relation generated by linearity in $v$, and $(\Sigma,v)\sim(\Sigma',Z([S])(v))$ for $[S]\colon\Sigma\to\Sigma'$. A pair $(\Sigma,v)$ (or just $v$ if there's no confusion) as in \eqref{eq:S_0^Z} is called a \textit{lasagna filling} of $(X,L)$.

The skein lasagna module $\mathcal{S}_0^Z(X;L)$ is an invariant of the pair $(X,L)$ up to diffeomorphism. When $L=\emptyset$, we write for short $\mathcal{S}_0^Z(X)$ for $\mathcal{S}_0^Z(X;L)$. In this case, since all lasagna fillings and their relations appear in $int(X)$, we in fact have the following.
\begin{Lem}\label{lem:interior_invariant}
The skein lasagna module $\mathcal S_0^Z(X)$ is an invariant of $int(X)$, the interior of $X$.\qed
\end{Lem}

\subsection{Properties of skein lasagna modules}\label{sbsec:lasagna_properties}
We state some properties of the skein lasagna modules, mostly following Manolescu--Neithalath \cite{manolescu2022skein}. Although their paper dealt with the case where the link homology theory $Z$ is the Khovanov-Rozansky $\mathfrak{gl}_N$ homology, the proofs remain valid in a more general setup. The more recent paper \cite{manolescu2023skein} gives a formula for $\mathcal{S}_0^Z(X;L)$ in terms of a handlebody decomposition of $X$. However, as we will only be using this result occasionally, we refer interested readers to their paper.

Every object $\Sigma$ of $\mathcal C(X;L)$ represents a homology class in $H_2(X,L)$, whose image under the boundary homomorphism $\partial$ in the long exact sequence $$0\to H_2(X)\to H_2(X,L)\xrightarrow{\partial}H_1(L)\to H_1(X)\to\cdots$$ is $[L]$, the fundamental class of $L$. If $[S]\colon\Sigma\to\Sigma'$ is a morphism, then $[\Sigma]=[\Sigma']\in H_2(X,L)$. This means $\mathcal S_0^Z(X;L)$ admits a grading by $H_2^L(X):=\partial^{-1}([L])$, which we write as $$\mathcal S_0^Z(X;L)=\bigoplus_{\alpha\in H_2^L(X)}\mathcal S_0^Z(X;L;\alpha).$$
Of course, $H_2^L(X)$ is nonempty if and only if the image of $[L]$ in $H_1(X)$ is zero, i.e., $L$ is null-homologous in $X$. From now on this will be assumed, since otherwise $\mathcal S_0^Z(X;L)$ is trivially zero as there is no lasagna filling.

\begin{Prop}[$3,4$-handle attachments, {\cite[Proposition~2.1]{manolescu2022skein}}]\label{prop:34_hd}
If $(X',L)$ is obtained from $(X,L)$ by attaching a $3$-handle (resp. $4$-handle) away from $L$, then there is a natural surjection (resp. isomorphism) $\mathcal S_0^Z(X;L)\to\mathcal S_0^Z(X';L)$.\qed
\end{Prop}

\begin{Prop}[Connected sum formula, {\cite[Theorem~1.4,Corollary~7.3]{manolescu2022skein}}]\label{prop:cntd_sum}
If $Z$ is strictly monoidal, then there is a natural isomorphism $$\mathcal{S}_0^Z(X\natural X';L\sqcup L')\cong\mathcal{S}_0^Z(X;L)\otimes\mathcal{S}_0^Z(X';L').$$ The same formula holds with the boundary sum replaced by the connected sum or the disjoint union.\qed
\end{Prop}

Suppose $(X,L)$ is a pair and $Y\subset X$ is a properly embedded separating oriented $3$-manifold, possibly with transverse intersections with $L$ in some point set $P\subset\partial Y$. Then $Y$ cuts $(X,L)$ into two pairs $(X_1,T_1)$ and $(X_2,T_2)$, where $\partial X_1$ induces the orientation on $Y$ and $\partial X_2$ induces the reverse orientation, and $T_1,T_2$ are framed oriented tangles with endpoints $P$. The framing and the orientation of $L$ restrict to a framing and an orientation of $P\subset\partial Y$.

If $T_0\subset Y$ is a framed oriented tangle with $\partial T_0=P$ in the framed oriented sense, then $T_1\cup T_0$ is a framed oriented link in $\partial X_1$ and $-T_0\cup T_2$ is one in $\partial X_2$. A lasagna filling of $(X_1,T_1\cup T_0)$ and a lasagna filling of $(X_2,-T_0\cup T_2)$ glue to a lasagna filling of $(X,L)$ in a bilinear way, inducing a map 
\begin{equation}\label{eq:glue_2_summands}
\mathcal S_0^Z(X_1;T_1\cup T_0)\otimes\mathcal S_0^Z(X_2;-T_0\cup T_2)\to\mathcal S_0^Z(X;L)
\end{equation}

\begin{Prop}[Gluing]\label{prop:gluing}
The gluing map $$\bigoplus_{\substack{T_0\subset Y\\\partial T_0=P}}\mathcal S_0^Z(X_1;T_1\cup T_0)\otimes\mathcal S_0^Z(X_2;-T_0\cup T_2)\to\mathcal S_0^Z(X;L)$$ is a surjection.
\end{Prop}
\begin{proof}
Every lasagna filling of $(X,L)$ can be isotoped to intersect $Y$ transversely in some tangle $T_0$.
\end{proof}

Write $L_1=T_1\cup T_0$. A framed surface $S\subset X_2$ with $\partial S=-T_0\cup T_2$ is a lasagna filling of $(X_2;-T_0\cup T_2)$ (without input balls), thus putting its class in the second tensor summand of \eqref{eq:glue_2_summands} defines a gluing map 
\begin{equation}\label{eq:glue}
\mathcal S_0^Z(X_2;S)\colon\mathcal S_0^Z(X_1;L_1)\to\mathcal S_0^Z(X;L).
\end{equation}

The most useful cases of gluing might be when $Y$ is closed, where $P=\emptyset$, and $T_0,T_1,T_2$ are links. Still more specially, if $X\subset int(X')$ is an inclusion of $4$-manifolds, there is an induced map $\mathcal S_0^Z(X)\to\mathcal S_0^Z(X')$ by gluing the exterior of $X\subset X'$ to $X$. For example, the map in Proposition~\ref{prop:34_hd} is defined this way.

The skein lasagna module with respect to $Z$ refines the link homology theory $Z$ in the following sense.

\begin{Prop}[Recover link homology theory]\label{prop:recover_Z}
For a framed oriented link $L\subset S^3$ we have a natural isomorphism $\mathcal{S}_0^Z(B^4;L)\cong Z(L)$. If $S\subset S^3\times I$ is a cobordism from $L$ to $L'$, then the gluing map $\mathcal S_0^Z(S^3\times I;S)\colon\mathcal S_0^Z(B^4;L)\to\mathcal S_0^Z(B^4;L')$ agrees with $Z(S)\colon Z(L)\to Z(L')$ under the natural isomorphisms.
\end{Prop}
\begin{proof}
The statement on objects is proved in \cite[Example~5.6]{morrison2022invariants} and the statement on morphisms is a simple exercise.
\end{proof}

Finally, we remark that extra structures on the link homology theory $Z$ usually induce structures on the skein lasagna module $\mathcal S_0^Z$. For example, suppose $Z$ takes values in $R$-modules with a grading by an abelian group $A$, so that for a link cobordism $S$, $Z(S)$ is homogeneous with degree $\chi(S)a$ for some $a\in A$ independent of $S$. Then the lasagna module $\mathcal S_0^Z(X;L)$ inherits an $A$-grading, once we impose an extra grading shift of $\chi(\Sigma)a$ in \eqref{eq:Z(Sigma)}, the definition of $Z(\Sigma)$ for a skein $\Sigma$. Analogously, a filtration structure on $Z$ induces a filtration structure on $\mathcal S_0^Z$. The isomorphisms or morphisms stated or constructed in this section all respect the extra grading or filtration structure, with a warning that the gluing morphism $\mathcal{S}_0^Z(X_2;S)$ in \eqref{eq:glue} has degree or filtration degree $(\chi(S)-\chi(T_0))a$.

\begin{Rmk}\label{rmk:TQFT}
In view of the gluing morphism \eqref{eq:glue} (in the special case $\partial Y=\emptyset$ and $T_1=\emptyset$), one is tempted to define a link homology theory for framed oriented links in oriented $3$-manifolds and cobordisms between them. Namely, for a $3$-manifold $Y$ and link $L\subset Y$, define $$Z(Y,L):=\bigoplus_{\partial X=Y}\mathcal S_0^Z(X;L).$$ For a cobordism $(W,S)\colon(Y,L)\to(Y',L')$, define $Z(W,S)\colon Z(Y,L)\to Z(Y',L')$ by taking the direct sum of the gluing maps $\mathcal{S}_0^Z(X;L)\to\mathcal{S}_0^Z(X\cup_YW;L')$.

However, the resulting modules $Z(Y,L)$ are usually infinite dimensional. To obtain a finite theory, one has to either restrict the domain category (one extreme of which is to recover back the link homology theory $Z$ for links in $S^3$ again), or modulo suitable skein-theoretic relations on $Z(Y,L)$. We leave the exploration of these possibilities to future work.
\end{Rmk}

\subsection{Khovanov-Rozansky \texorpdfstring{$\mathfrak{gl}_2$}{gl2} homology and its Lee deformation}\label{sbsec:KhR_2}
In this paper we will take the link homology theory $Z$ in the skein lasagna construction to be the Khovanov-Rozansky $\mathfrak{gl}_2$ homology or its Lee deformation. In this section we review some features of these link homology theories, and establish renormalization conventions and notation. We assume the reader is familiar with the classical Khovanov homology as well as its Lee deformation. See \cite{khovanov2000categorification,lee2005endomorphism,rasmussen2010khovanov} for relevant background and \cite{bar2002khovanov} for an accessible overview.

We follow the convention in \cite{morrison2022invariants}. Khovanov-Rozansky $\mathfrak{gl}_2$ homology is a renormalization of the Khovanov-Rozansky $\mathfrak{sl}_2$ homology defined in \cite{khovanov2008matrix} (and extended to $\Z$-coefficients in \cite{blanchet2010oriented}) that is sensitive to the framing of a link. Since conventionally most calculations have been done using the classical Khovanov homology \cite{khovanov2000categorification}, we write down explicitly the renormalization differences between it and the $\mathfrak{gl}_2$ theory.

Let $KhR_2$ and $Kh$ denote the Khovanov-Rozansky $\mathfrak{gl}_2$ homology and the usual Khovanov homology, respectively. Then for a framed oriented link $L\subset S^3$, we have \cite{beliakova2023functoriality}
\begin{equation}\label{eq:KhR_2}
KhR_2^{h,q}(L)\cong Kh^{h,-q-w(L)}(-L),
\end{equation}
where $-L$ is the mirror image of $L$, and $w(L)$ is the writhe of $L$, defined as the linking number between $L$ and a slight pushoff of $L$ in the framing direction. When we want to emphasize that we are working over a commutative ring $R$, write $Kh(L;R)$ and $KhR_2(L;R)$ for the corresponding homology theories.

A dotted framed oriented cobordism $S\colon L_0\to L_1$ in $S^3\times I$ induces a map $KhR_2(S)\colon KhR_2(L_0)\to KhR_2(L_1)$ of bidegree $(0,-\chi(S)+2\#(\text{dots}))$ which only depends on the isotopy class of $S$ rel boundary. Here, a ``dot'' marks a multiplication-by-$X$ operation in the associated Frobenius algebra (see the paragraph below). Let $\bar S\colon-L_0\to-L_1$ denote the mirror image of $S$. Then the identification \eqref{eq:KhR_2} is functorial in the sense that $Kh(\bar S)\colon Kh(-L_0)\to Kh(-L_1)$ is equal to $KhR_2(S)$ under \eqref{eq:KhR_2} up to sign. Since the classical Khovanov homology is only known to be functorial in $\R^3$ up to sign, $KhR_2$ can be thought of as a sign fix and a functorial upgrade of $Kh$.

Both the constructions of $KhR_2$ and $Kh$ depend on the input of the Frobenius algebra $V=\Z[X]/(X^2)$\footnote{Strictly speaking, the construction of $KhR_2$ uses the formalism of foams, but for simplicity our exposition remains in the classical Bar-Natan formalism throughout.} with comultiplication and counit given by $$\Delta1=X\otimes 1+1\otimes X,\ \Delta X=X\otimes X,\ \epsilon1=0,\ \epsilon X=1.$$ Here $X$ has quantum degree $2$ in the $KhR_2$ case and $-2$ in the $Kh$ case.

The Lee deformation of $V$ over $\Q$ \cite{lee2005endomorphism} is the Frobenius algebra $V_{Lee}=\Q[X]/(X^2-1)$ with $$\Delta1=X\otimes1+1\otimes X,\ \Delta X=X\otimes X+1\otimes1,\ \epsilon1=0,\ \epsilon X=1.$$ Using $V_{Lee}$ instead of $V$ defines the Khovanov-Rozansky $\mathfrak{gl}_2$ Lee homology and the classical Lee homology, denoted $KhR_{Lee}$ and $Kh_{Lee}$, respectively, which are related by an equation analogous to \eqref{eq:KhR_2}. Since $V_{Lee}$ is no longer graded but only filtered, the resulting homology theories carry quantum filtration structures, instead of quantum gradings (although a compatible quantum $\Z/4$-grading is still present). We remark that $0\in KhR_{Lee}$ has quantum filtration degree $-\infty$ while $0\in Kh_{Lee}$ has quantum filtration degree $+\infty$, since the two filtrations go in different directions.

Either $KhR_2$ or $KhR_{Lee}$ changes only by a renormalization upon changing the orientation of the link $L$. More precisely, if $(L,\mathfrak o)$ denote the link $L$ equipped with a possibly different orientation $\mathfrak o$, then for $\bullet\in\{2,Lee\}$,
\begin{equation}\label{eq:KhR_ori_change}
KhR_\bullet(L,\mathfrak o)\cong KhR_\bullet(L)\left[\frac{w(L)-w(L,\mathfrak o)}{2}\right]\left\{\frac{w(L,\mathfrak o)-w(L)}{2}\right\}.
\end{equation}
Here, $[\cdot]$ and $\{\cdot\}$ indicate shifts in homological degree and quantum (filtration) degree, respectively.

It is sometimes more convenient to change from the basis $\{1,X\}$ of $V_{Lee}$ to the basis $\{\A=(1+X)/2,\,\B=(1-X)/2\}$. In this new basis, $$\A\B=0,\ \A^2=\A,\ \B^2=\B,\ 1=\A+\B,\ X=\A-\B,\ \Delta\A=2\A\otimes\A,\ \Delta\B=-2\B\otimes\B,\ \epsilon\A=\tfrac12,\ \epsilon\B=-\tfrac12.$$

Like the classical story in \cite{rasmussen2010khovanov}, the Khovanov-Rozansky $\mathfrak{gl}_2$ Lee homology of a framed oriented link $L$ has a basis consisting of canonical generators $x_\mathfrak o$, one for each orientation $\mathfrak o$ of $L$ as an unoriented link. The explicit renormalization convention of our choice of generators, together with the proof of the following properties, can be found in Appendix~\ref{sec:Lee_generator}.

\begin{Prop}\label{prop:Lee_generator_properties}\mbox{}\hfill\vspace{-5pt}
\begin{enumerate}[(1)]
\item The unique canonical generator of the empty link $\emptyset$ is $1\in\Q\cong KhR_{Lee}(\emptyset)$.
\item The $\mathfrak{gl}_2$ Lee homology of the $n$-framed oriented unknot $(U^n,\mathfrak o)$ is canonically isomorphic to $V_{Lee}\{-n-1\}$, with canonical generators given by $x_\mathfrak o=\A$, $x_{\bar{\mathfrak o}}=\B$, renormalized by $\{-n-1\}$.
\end{enumerate}
Let $L$ be a framed oriented link and $\mathfrak o$ be an orientation of $L$ as an unoriented link. Write $L=L_+(\mathfrak o)\cup L_-(\mathfrak o)$ as a union of sublinks on which $\mathfrak o$ agrees (resp. differs) with the orientation of $L$.
\begin{enumerate}[(1)]
\setcounter{enumi}{2}
\item $x_\mathfrak o$ is homogeneous with homological degree $$h(x_\mathfrak o)=-2\ell k(L_+(\mathfrak o),L_-(\mathfrak o))=(w(L,\mathfrak o)-w(L))/2.$$
\item $x_\mathfrak o\pm x_{\bar{\mathfrak o}}$ is homogeneous with $\Z/4$ quantum degrees $$q_{\Z/4}(x_\mathfrak o\pm x_{\bar{\mathfrak o}})=(-w(L)-\#L-1\pm1)\bmod4.$$
\item If $S\colon L\to L'$ is a framed oriented link cobordism, then
\begin{equation}\label{eq:can_gen}
KhR_{Lee}(S)(x_{\mathfrak o})=2^{n(S)}\sum_{\mathfrak O|_L=\mathfrak o}(-1)^{n(S_-(\mathfrak O))}x_{\mathfrak O|_{L'}},
\end{equation}
where $\mathfrak O$ runs over orientations on $S$ that are compatible with the orientation $\mathfrak o$ (thus each $(S,\mathfrak O)$ is an oriented cobordism from $(L,\mathfrak o)$ to $(L',\mathfrak O|_{L'})$), $S_-(\mathfrak O)$ is the union of components of $S$ whose orientations disagree with $\mathfrak O$, and 
\begin{equation}\label{eq:n_renormalization}
n(\Sigma)=\frac{-\chi(\Sigma)+\#\partial_+\Sigma-\#\partial_-\Sigma}{2}\in\Z
\end{equation}
for a surface $\Sigma$ regarded as a cobordism from the negative boundary $\partial_-\Sigma$ to the positive boundary $\partial_+\Sigma$.
\end{enumerate}
\end{Prop}
Note that the function $n(\cdot)$ in \eqref{eq:n_renormalization} is additive under gluing cobordisms.

The $\mathfrak{gl}_2$ $s$-invariant of $L$ is defined to be $$s_{\mathfrak{gl}_2}(L):=q(x_{\mathfrak o_L})-1$$ where $\mathfrak o_L$ is the orientation of $L$ and $q\colon KhR_{Lee}(L)\to\Z\sqcup\{-\infty\}$ is the quantum filtration function. It relates to the classical $s$-invariant of oriented links (as defined in \cite{beliakova2008categorification}, after \cite{rasmussen2010khovanov}) by
\begin{equation}\label{eq:s_renormalize}
s_{\mathfrak{gl}_2}(L)=-s(-L)-w(L).
\end{equation}

\subsection{Properties for Khovanov and Lee skein lasagna modules}\label{sbsec:Kh_lasagna_properties}
Letting $Z=KhR_2$ or $KhR_{Lee}$ as the link homology theory for the skein lasagna module, we recover the Khovanov and Lee skein lasagna modules, denoted $\mathcal{S}_0^2$ and $\mathcal{S}_0^{Lee}$, respectively, the first of which agrees with the construction in \cite{morrison2022invariants} (or more concisely in \cite{manolescu2022skein}) for the Khovanov-Rozansky $\mathfrak{gl}_2$ skein lasagna module.

\begin{Rmk}\label{rmk:dotted_skein}
Since $KhR_2$ and $KhR_{Lee}$ allow dots on cobordisms, in the definition of lasagna fillings one may also allow dots on skeins, which does not change the structure of $\mathcal{S}_0^2$ or $\mathcal{S}_0^{Lee}$ but makes the description of certain elements easier. One can think of a dot as an $X$-decoration on the unknot in the boundary of an extra deleted local $4$-ball near the dot.
\end{Rmk}

It follows from our general discussion in Section~\ref{sbsec:skein_construction} and Section~\ref{sbsec:lasagna_properties} that for a pair $(X,L)$ of $4$-manifold $X$ and framed oriented link $L\subset\partial X$, the Khovanov skein lasagna module $$\mathcal{S}_0^2(X;L)=\bigoplus\mathcal{S}_{0,h,q}^2(X;L;\alpha)$$ is an abelian group with a homological grading by $h\in\Z$, a quantum grading by $q\in\Z$, and a homology class grading by $\alpha\in H_2^L(X)$. Similarly, the Lee skein lasagna module $\mathcal{S}_0^{Lee}(X;L)$ is a $\Q$-vector space with a homological $\Z$-grading, a quantum filtration, a compatible quantum $\Z/4$-grading, and a homology class grading by $H_2^L(X)$. Incorporating the grading convention described in Section \ref{sbsec:lasagna_properties} with \eqref{eq:Z(Sigma)} we have
\begin{equation}\label{eq:KhR(Sigma)}
KhR_\bullet(\Sigma) = \bigotimes_{i=1}^k KhR_\bullet(K_i)\{-\chi(\Sigma)\}
\end{equation}
for skeins $\Sigma$ in $X$ rel $L$ with input links $K_i$, and $\bullet\in\{2,Lee\}$.  Since the space of Khovanov/Lee lasagna fillings with underlying skein $\Sigma$ is in one-one correspondence to $KhR_\bullet(\Sigma)$, this also defines homological gradings and quantum gradings/filtrations on Khovanov/Lee lasagna fillings. Thus, the class $[(\Sigma,v)]$ in $\mathcal{S}_0^\bullet(X;L)$ has homological degree $h(v)$ for homogeneous $v\in KhR_\bullet(\Sigma)$. If $\bullet=2$, $[(\Sigma,v)]$ has quantum degree $q(v)$ if $v$ is homogeneous; if $\bullet=Lee$, $[(\Sigma,v)]$ has filtration degree at most $q(v)$. Finally, when we work over a ring $R$, we write $\mathcal{S}_0^2(X;L;R)=\oplus_\alpha\mathcal{S}_0^2(X;L;\alpha;R)$ for the corresponding Khovanov skein lasagna module.

Next we describe a formula of Khovanov and Lee skein lasagna modules for $2$-handlebodies, developed in \cite{manolescu2022skein}. From now on till the end of this section, suppose $X$ is a $2$-handlebody obtained by attaching $2$-handles along a framed oriented link $K=K_1\cup\cdots\cup K_m\subset S^3=\partial B^4$ disjoint from a framed oriented link $L\subset S^3$. Then $L$ can be regarded as a link in $\partial X$. A class $\alpha\in H_2^L(X)$ intersects the cocore of the $2$-handle on $K_i$ algebraically some $\alpha_i$ times, and the assignment $\alpha\mapsto(\alpha_i)_i$ gives an isomorphism $H_2^L(X)\cong\Z^m$. Write $\alpha_i^+=\max(0,\alpha_i)$ and $\alpha_i^-=\max(0,-\alpha_i)$. Write $|\cdot|$ for the $L^1$-norm on $\Z^m$.

For $k^+,k^-\in\Z_{\ge0}^m$, let $K(k^+,k^-)\cup L\subset S^3$ denote the framed oriented link obtained from $K\cup L$ by replacing each $K_i$ by $k_i^++k_i^-$ parallel copies of itself, with orientations on $k_i^-$ of them reversed.

\begin{Prop}[$2$-handlebody formula]\label{prop:2_hdby}
Let $(X,L)$ be given as above. Then 
\begin{equation}\label{eq:2-hdby}
\mathcal S_0^\bullet(X;L;\alpha)\cong\left(\bigoplus_{r\in\Z_{\ge0}^m}KhR_\bullet(K(\alpha^++r,\alpha^-+r)\cup L)\{-|\alpha|-2|r|\}\right)\bigg/\sim
\end{equation}
where $\bullet\in\{2,Lee\}$ and $\sim$ is the linear equivalence relation generated by\vspace{-5pt}
\begin{enumerate}[(i)]
\item $\sigma_i\cdot x\sim x$ for $x\in KhR_\bullet(K(\alpha^++r,\alpha^-+r)\cup L)\{-|\alpha|-2|r|\}$ and $\sigma_i\in S_{|\alpha_i|+2r_i}$, where $\sigma_i$ acts by permuting the cables for $K_i$ (cf. \cite{grigsby2018annular} and \cite[Remark 3.7]{manolescu2022skein})\footnote{Strictly speaking, elements in $S_{\alpha_i+2r_i}$ may permute cables of $K_i$ with different orientations, so it does not act on $KhR_\bullet(K(\alpha^++r,\alpha^-+r)\cup L)\{-|\alpha|-2|r|\}$. Nevertheless, since changing orientation only affects $KhR_\bullet$ by a renormalization, we can first work with $K(\alpha^++\alpha^-+2r,0)\cup L$ and then renormalize.};
\item $\phi_i(x)\sim x$ for $x\in KhR_\bullet(K(\alpha^++r,\alpha^-+r)\cup L)\{-|\alpha|-2|r|\}$ where 
\begin{align*}
\phi_i\colon KhR_\bullet(K(\alpha^++r,&\alpha^-+r)\cup L)\{-|\alpha|-2|r|\}\to\\
&\,KhR_\bullet(K(\alpha^++r+e_i,\alpha^-+r+e_i)\cup L)\{-|\alpha|-2|r|-2\}
\end{align*}
is the map induced by the dotted annular cobordism that creates two parallel, oppositely oriented cables of $K_i$ (which is bidegree-preserving), and where $e_i$ is the $i$\textsuperscript{th} coordinate vector.\vspace{-5pt}
\end{enumerate}
Here, when $\bullet=2$, we assume we work over a base ring where $2$ is invertible; otherwise, $\mathcal{S}_0^2(X;L;\alpha)$ is given by a further quotient of the right-hand side of \eqref{eq:2-hdby}.
\end{Prop}
\begin{proof}
The $\mathcal S_0^2$ version follows from a slight generalization of Theorem~1.1 and Proposition~3.8\footnote{The $Kh$ in Proposition~3.8 in \cite{manolescu2022skein} should be $KhR_2$ instead.} in \cite{manolescu2022skein}, stated as Theorem~3.2 in \cite{manolescu2023skein}. The claim for the case when $2$ is not invertible follows from the proof of Proposition~3.8 in \cite{manolescu2022skein}. The version for $\mathcal S_0^{Lee}$ has an identical proof.
\end{proof}

Changing the orientation of a link only affects $KhR_\bullet$ by a grading shift, given by \eqref{eq:KhR_ori_change}. Therefore for $2$-handlebodies $X$, by \eqref{eq:2-hdby}, $\mathcal S_0^\bullet(X;L)$ is independent of the orientation of $L$ (as long as it is null-homologous) up to renormalizations, and $\mathcal S_0^\bullet(X;L;\alpha)$ only depends on the mod $2$ reduction of $\alpha\in H_2^L(X)\subset H_2(X,L)$ up to renormalizations. Explicitly, suppose $\beta\in H_2(X,L)$ has the same mod $2$ reduction as $\alpha\in H_2^L(X)$ such that $\partial\beta\in H_1(L)$ is the fundamental class of $L$ with a possibly different orientation. Let $(L,\partial\beta)$ denote the link $L$ with orientation given by $\partial\beta$. Then $\beta\in H_2^{(L,\partial\beta)}(X)$, and we have
\begin{equation}\label{eq:2-hdby_homology_mod_2}
\mathcal S_0^\bullet(X;(L,\partial\beta);\beta)\cong\mathcal S_0^\bullet(X;L;\alpha)\left[\frac{\alpha^2-\beta^2}{2}\right]\left\{\frac{\beta^2-\alpha^2}{2}\right\}.
\end{equation}
for any $2$-handlebody $X$. Also, when $\bullet=2$ in \eqref{eq:2-hdby_homology_mod_2}, we assume $2$ is invertible in the base ring.

The bilinear product on $H_2^L(X)$ in \eqref{eq:2-hdby_homology_mod_2} (similarly for $H_2^{(L,\partial\beta)}(X)$) is defined by counting intersections of two surfaces in $X$ with boundary $L$, one copy of which is pushed off $L$ via its framing. In particular, it is independent of the choice of $K\cup L\subset S^3$.

\begin{Ex}\label{ex:S2D2}
For the $2$-handlebody $X=S^2\times D^2$, $K$ is the $0$-framed unknot. It follows from Proposition~\ref{prop:2_hdby} that $\mathcal{S}_0^\bullet(S^2\times D^2)$ can be computed from the Khovanov or Lee homology of unlinks and dotted cobordism maps between them. Explicitly, \cite[Theorem~1.2]{manolescu2022skein} calculated that for every $\alpha\in H_2(S^2\times D^2)$, $$\mathcal{S}_{0,h,q}^2(S^2\times D^2;\alpha)=\begin{cases}\Z,&h=0,\,q=-2k\le0\\0,&\text{otherwise}.\end{cases}$$ Similarly, $gr_q\mathcal{S}_{0,h}^{Lee}(S^2\times D^2;\alpha)$ is given by the same formula with $\Z$ replaced by $\Q$.

Using Proposition~\ref{prop:2_hdby}, \cite[Theorem~1.3]{manolescu2022skein} also obtained partial information on Khovanov skein lasagna modules for $D^2$-bundles over $S^2$ with Euler number $n\ne0$, denoted $D(n)$, namely that $$\mathcal S_{0,0,*}^2(D(n);0)=\begin{cases}\Z,&n<0,\ *=0\\
0,&\text{otherwise.}\end{cases}$$
For $n=\pm1$, this also yields the corresponding calculation for $\CP^2$, $\overline{\CP^2}$, in view of Proposition~\ref{prop:34_hd}. These results will be improved for $n>0$ in Example~\ref{ex:unknot_van}, and for $n<0$ in Example~\ref{ex:unknot_trace_nonvan}, Section~\ref{sbsec:torus_traces}, and Section~\ref{sbsec:-CP^2}.
\end{Ex}

If $L',L$ are two framed links in an oriented $3$-manifold with the same underlying unframed link, their framings differ by a tuple of integers, one for each link component. The total framing difference between them, denoted $w(L',L)$, is the sum of these integers. For example, if $U^n$ denotes the $n$-framed unknot in $S^3$, then $w(U^n,U^m)=n-m$.

\begin{Prop}[Framing change]\label{prop:framing_change}
If $L,L'\subset\partial X$ are framed oriented links that are equal as unframed oriented links, then for $\bullet\in\{2,Lee\}$, $$\mathcal S_0^\bullet(X;L')\cong\mathcal S_0^\bullet(X;L)\{-w(L',L)\}.$$
\end{Prop}
\begin{proof}
For a given lasagna filling of $(X,L)$, create new input balls near each component of $L$ that intersect the skein in disks. Change the framing of $L$ to $L'$ and push the framing change into the intersection disks. Then the sum of framings on the input unknots is $w(L',L)$. Assigning these unknots the Khovanov or Lee class $1$ (the element corresponds to $1\in KhR_\bullet(U)$ under the Reidemeister I induced isomorphisms, where $U$ is the $0$-framed unknot) produces a lasagna filling of $(X,L')$. This assignment induces the claimed graded isomorphism.
\end{proof}

Finally, in view of the trick for framing changes in the proof of Proposition~\ref{prop:framing_change}, we can state a slight improvement of the gluing construction \eqref{eq:glue}. In the general setup when gluing $(X_2,S)$ to $(X_1,L_1)$ to obtain $(X,L)$, in order to get an induced map on skein lasagna modules, we needed to assume that $S$ was framed, so that it could be regarded as a skein in $X_2$ rel $L_2$. In the Khovanov or Lee skein lasagna module setup, however, we can drop the framing assumption, at the expense of deleting extra input balls to compensate. Thus, for any (unframed) surface $S\subset X_2$ with the previous assumptions, there is a gluing morphism
\begin{equation}\label{eq:glue_Kh_ver}
\mathcal S_0^\bullet(X_2;S)\colon\mathcal S_0^\bullet(X_1;L_1)\to\mathcal S_0^\bullet(X;L),
\end{equation}
which has homological degree $0$ and quantum (filtration) degree $-\chi(S)+\chi(T_0)-[S]^2$. If $S$ is already framed, this map agrees with the previous construction.

\subsection{Structure of space of skeins}\label{sbsec:skein_structure}
The skein lasagna module $\mathcal S_0^Z(X;L)$ decomposes as a direct sum over $\Omega_{\ge0}(X;L):=\pi_0(\mathcal C(X;L))$, the set of connected components of the category $\mathcal C(X;L)$. We call $\Omega_{\ge0}(X;L)$ the \textit{space of skeins} in $X$ rel $L$. In this section, we show that $\Omega_{\ge0}(X;L)$ contains the same information as $H_2^L(X)$.

\begin{Thm}\label{thm:space_of_skeins}
Taking homology defines a bijection
\begin{equation}\label{eq:skein_homology_iso}
\Omega_{\ge0}(X;L)\xrightarrow{\cong}H_2^L(X).
\end{equation}
\end{Thm}

We recall that $\Omega_{\ge0}(X;L)$ is the set of (properly embedded, oriented) framed surfaces $\Sigma$ in $X$ outside finitely many open balls with $\partial\Sigma|_{\partial X}=L$ (such $\Sigma$ are called skeins in $X$ rel $L$) up to\vspace{-5pt}
\begin{enumerate}[(a)]
\item isotopy rel boundary;
\item enclosing some small balls by some large balls and deleting the interior of the large balls, or its inverse.\vspace{-5pt}
\end{enumerate}
Although\vspace{-5pt}
\begin{enumerate}[(a)]
\setcounter{enumi}{2}
\item isotopy of input balls in $X$ (that drags surfaces along)\vspace{-5pt}
\end{enumerate}
is not directly allowed in the relations, it can be recovered using a combination of (a)(b).

In the context of studying skeins in $X$ rel $L$, it is natural to consider the following cobordism relation. We say two skeins $\Sigma_0,\Sigma_1$ in $X$ rel $L$ with input balls $\sqcup_iint(B_i^{(0)}),\sqcup_iint(B_i^{(1)})$ are \textit{framed cobordant rel $L$} if there is a properly embedded framed oriented $3$-manifold $Y\subset(X\times I)\backslash\nu(\Gamma)$ for some tubular neighborhood $\nu(\Gamma)$ of some $1$-complex $\Gamma\subset int(X)\times I$, such that $\partial Y|_{X\times\{k\}}=\Sigma_k\times\{k\}$, $k=0,1$, and $\partial Y|_{\partial X\times I}=L\times I$. It is understood that $\nu(\Gamma)\cap(X\times\{k\})=\sqcup_iint(B_i^{(k)})\times\{k\}$, $k=0,1$, and that $Y$ can have boundary on $\partial\nu(\Gamma)$.\medskip

\begin{Lem}\label{lem:skein_framed_cobordant}
Two skeins are equal in $\Omega_{\ge0}(X;L)$ if and only if they are framed cobordant.
\end{Lem}
\begin{proof}
One direction is clear: (a) is clearly contained in the relation of framed cobordisms, and (b) is realized by taking the identity cobordism but gradually enlarging the small deleted balls so that they merge with each other and become the large balls.

For the converse, after an isotopy we can decompose every framed cobordism of skeins in $X$ rel $L$ into the following elementary moves:\vspace{-5pt}
\begin{enumerate}[(i)]
\item Isotopy of the surface and deleted balls rel $\partial X$;
\item Morse move away from deleted balls;
\item Pushing a Morse critical point into one deleted ball;
\item Deleting a local ball away from the surface, or its inverse;
\item Choosing a path $\gamma$ between two deleted balls $B_1,B_2$ disjoint from the surface and replacing $B_1,B_2$ by a small tubular neighborhood of $B_1\cup B_2\cup\gamma$, or its inverse.\vspace{-5pt}
\end{enumerate}
That these moves are sufficient can be proved by using (i)(iv)(v) to account for the topology change of the $1$-complex $\Gamma$ as one moves along the time ($I$-)direction, and (i)(ii)(iii) to account for the topology change of the cobordism $3$-manifold, assuming $\Gamma$ is constant in time.

We see immediately that (i) is realized by (a) and (c), while (iii),(iv), and (v) are realized by (b). Since a Morse move is local, we can use (b) twice to delete a local ball and reglue it back to realize (ii).
\end{proof}

\begin{proof}[Proof of Theorem~\ref{thm:space_of_skeins}]
First we show the map is surjective. Any class $\alpha\in H_2^L(X)$ can be represented by an oriented embedded surface $\Sigma\subset X$ with $\partial\Sigma=L$. The obstruction to putting a framing on $\Sigma$ lies in $H^2(\Sigma,L)$. Once we delete local $4$-balls at points on each component of $\Sigma$, this obstruction vanishes and we get a skein $\Sigma$ with $[\Sigma]=\alpha$.

Next we show the map is injective. Without loss of generality, assume $X$ is connected. Suppose two skeins $\Sigma_0$, $\Sigma_1$ have the same homology class. By enclosing all input balls by one large ball, we may assume $\Sigma_0,\Sigma_1\in X^\circ:=X\backslash int(B^4)$ for some fixed ball $B^4$ in $X$. Since $[\Sigma_1]-[\Sigma_0]=0\in H_2(X^\circ\times I,\partial X\times I)$, we can find a properly immersed oriented $3$-manifold $Y\looparrowright X^\circ\times I$ that cobounds $\Sigma_0,\Sigma_1$. By general position we may assume that the self-intersection of $Y$ is an embedded $1$-manifold in $X^\circ\times I$ with endpoints on $\partial(X^\circ\times I)$ and interior singularities of the immersion modeled on the Whitney umbrella \cite{whitney1944singularities}. Deleting a tubular neighborhood along the self intersection makes $Y$ an embedded cobordism between $\Sigma_0$ and $\Sigma_1$ in $(X\times I)\backslash\nu(\Gamma)$ for some $1$-complex $\Gamma$. The obstruction of putting a compatible framing on $Y$ lies in $H^2(Y,\partial Y|_{(X^\circ\times\partial I)\cup(\partial X\times I)})\cong H_1(Y,\partial Y|_{\partial\nu(\Gamma)})$, which can be made zero upon a further deletion of the tubular neighborhood of a $1$-complex. This yields a framed cobordism between $\Sigma_0,\Sigma_1$, showing that they are equal in $\Omega_{\ge0}(X;L)$ by Lemma~\ref{lem:skein_framed_cobordant}.
\end{proof}

\begin{Rmk}
An alternative proof of Theorem~\ref{thm:space_of_skeins} can be deduced from a relative and noncompact version of \cite[Theorem~2]{kirby2012cohomotopy} applied to $(X^\circ,L)$ together with the fact that a framed link translation (as defined in their paper) can be undone in $\Omega_{\ge0}(X;L)$ by a neck-cutting. However, for our purpose in Section~\ref{sec:Lee_lasagna}, we need a similar result for the space of a pair of skeins, which can be proved in the same way as we demonstrated, but is inaccessible by directly applying \cite[Theorem~2]{kirby2012cohomotopy}.
\end{Rmk}

\subsection{Variations of the construction}\label{sbsec:variation}
In this section we assume $X$ to be connected and nonempty. The skein lasagna module $\mathcal{S}_0^Z(X;L)$ is the colimit of $Z$ along the category $\mathcal C_{\ge0}(X;L):=\mathcal C(X;L)$ of skeins. One can restrict to the subcategory $\mathcal C_1(X;L)\subset\mathcal C_{\ge0}(X;L)$ of skeins with exactly $1$ input ball. The colimit of $Z$ along $\mathcal C_1(X;L)$ produces a slightly different version of skein lasagna module, say $\mathcal S_0^{'Z}(X;L)$. It comes with a map
\begin{equation}\label{eq:refine_S_0^Z}
\mathcal S_0^{'Z}(X;L)\twoheadrightarrow\mathcal S_0^Z(X;L)
\end{equation}
which is surjective since every object in $\mathcal C_{\ge0}(X;L)$ has a successor in $\mathcal C_1(X;L)$.

In fact, \eqref{eq:refine_S_0^Z} is an isomorphism whenever $X$ has no $1$-handles. This is because the relevant formulas of $\mathcal S_0^Z$ for attaching $2,3,4$-handles (generalized appropriately for $2$-handles when $Z\ne KhR_\bullet$), developed by Manolescu--Walker--Wedrich \cite{manolescu2023skein}, still hold for $\mathcal S_0^{'Z}$. In particular, most calculations done in this paper, including the proof of Theorem~\ref{thm:exotic}, remain valid with $\mathcal S_0^\bullet$ replaced by $\mathcal S_0^{'\bullet}$, $\bullet\in\{2,Lee\}$.

However, when $H_1(X)\ne0$, \eqref{eq:refine_S_0^Z} is usually not an injection. In fact, a relative and noncompact version of \cite[Theorem~2]{kirby2012cohomotopy} shows that the relevant space of skeins $\Omega_1(X;L):=\pi_0(\mathcal{C}_1(X;L))$ admits a surjection onto $H_2^L(X)\cong\Omega_{\ge0}(X;L)$ by taking homology, whose fiber is usually nontrivial.

We illustrate the problem in the following example. Take $Z=KhR_2$ and let $K$ be a knot. By neck-cutting \cite[Lemma~7.2]{manolescu2022skein}, unwinding and regluing, we see that the lasagna filling $(K\times S^1,1)$ of $S^3\times S^1$ (with any framing on $K\times S^1$) is equivalent to some lasagna filling $(T,1)$ where $T$ is a local torus in $S^3\times S^1$. Since every torus in $B^4$ evaluates to $2$ in $KhR_2$ (see \cite{rasmussen2005khovanov} for a proof in the $Kh$ case, up to sign), $(T,1)$, and hence $(K\times S^1,1)$, is equivalent to $(\emptyset,2)$. However, the neck-cutting relation no longer holds for $\mathcal{S}_0^{'2}$, so $(K\times S^1,1)$ cannot be simplified there in general. More precisely, if the framing on $K\times S^1$ has twists along the $K$ factor, then $K\times S^1$ and $\emptyset$ do not represent the same element in $\Omega_1(S^3\times S^1)$ \cite{kirby2012cohomotopy}, so $(K\times S^1,1)$ is not equivalent to $(\emptyset,2)$ in $\mathcal S_0^{'2}(S^3\times S^1)$, as the latter represents a class that is nonzero in the quotient $\mathcal S_0^2(S^3\times S^1)$ \cite[Corollary~4.2]{manolescu2023skein}.

Since the proof for the connected sum formula (Proposition~\ref{prop:cntd_sum}) requires a neck-cutting, we also lose it for $\mathcal{S}_0^{'2}$. Thus, although $\mathcal{S}_0^{'Z}$ is a refinement of $\mathcal{S}_0^Z$, its computation could be more challenging. We suspect that when $X$ is simply-connected, $\mathcal C_1(X;L)$ is cofinal in $\mathcal C_{\ge0}(X;L)$ and thus \eqref{eq:refine_S_0^Z} is an isomorphism.

One may also consider versions of skein lasagna modules with various conditions on framings and orientations. For example, if one drops the assumption that surfaces are framed, one can obtain a skein lasagna module for every input link homology theory of unframed oriented links in $S^3$. The underlying space of skeins in $X$ rel $L$, no matter whether one allows only one input ball or finitely many (or none), is again isomorphic to $H_2^L(X)$. The many-balls version of the unframed skein lasagna module for the Khovanov or Lee link homology theory seems to agree with its framed counterpart up to renormalizations, in view of the framing-change trick in the proof of Proposition~\ref{prop:framing_change}. In fact, it could be notationally simpler to work with this version throughout the paper, but we follow the existing convention. The structure of the one-ball version of the unframed Khovanov skein lasagna module is less clear.

\section{Vanishing criterion and examples}\label{sec:van}
In this section we prove Theorem~\ref{thm:van} and give a few examples where the theorem applies. The proof is a straightforward application of Ng's maximal Thurston-Bennequin bound from Khovanov homology \cite{ng2005legendrian} and the $2$-handlebody formula \eqref{eq:2-hdby}.

\begin{Lem}[{\cite[Theorem~1]{ng2005legendrian}}]\label{lem:Ng}
For an oriented link $L$, the Khovanov homology $Kh^{h,q}(L)$ is supported in the region $$TB(L)\le q-h\le-TB(-L).$$
\end{Lem}

\begin{proof}[Proof of Theorem~\ref{thm:van}]
We set out to show that $\mathcal{S}_0^2(X;L)=0$ for any $X$ containing an embedded knot trace $X_n(K)$ with framing $n\ge-TB(-K)$.  By the gluing formula Proposition~\ref{prop:gluing}, it suffices to prove that $\mathcal{S}_0^2(X_n(K);L)$ vanishes for such $n$ (and for any boundary link $L$). The strategy will be to use the framing assumption and Lemma \ref{lem:Ng} to show that, in any fixed $(q+h)$-degree, the homology groups of the cablings in the 2-handlebody formula \eqref{eq:2-hdby} must eventually go to zero.

We give $K$ framing $n$ and an arbitrary orientation. Any $L\subset\partial X$ can be isotoped to miss the $2$-handle, hence we may henceforth assume $L$ is a framed oriented link in $S^3$ missing $K$. Since $-n\le TB(-K)$, we can find a Legendrian representative $\mathcal K\cup\mathcal L$ of $-K\cup-L$ with $\mathcal K$ having Thurston-Bennequin number $tb(\mathcal K)=-n$. In view of Proposition~\ref{prop:framing_change}, we may assume without loss of generality that the framing of $-L$ induced from $L$ agrees with the contact framing of $\mathcal L$.

Let $\mathcal{K}(k^+,k^-)$ denote the Legendrian $(k^++k^-)$-cable of $\mathcal K$ with the orientation on $k^-$ of the strands reversed. In other words, $\mathcal K(k^+,k^-)$ is the link obtained by taking $k^++k^-$ pushoffs of $\mathcal K$ by the Reeb flow associated to a compatible contact form on $S^3$, and reversing the orientation on $k^-$ of the components. The underlying link type of $\mathcal K(k^+,k^-)\cup\mathcal L$ is $-(K(k^+,k^-)\cup L)$, and the contact framing on the former agrees with the framing on the latter induced from $K$ and $L$. We thus have
\begin{equation}\label{eq:TB_KL}
TB(-(K(k^+,k^-)\cup L))\ge tb(\mathcal{K}(k^+,k^-)\cup\mathcal L)=w(-(K(k^+,k^-)\cup L))=-w(K(k^+,k^-)\cup L).
\end{equation}
By Lemma~\ref{lem:Ng} and \eqref{eq:TB_KL}, $Kh(-(K(k^+,k^-)\cup L))$ is supported in $q-h\ge-w(K(k^+,k^-)\cup L)$. By \eqref{eq:KhR_2}, $KhR_2(K(k^+,k^-)\cup L)$ is supported in $q+h\le w(K(k^+,k^-)\cup L)-w(K(k^+,k^-)\cup L)=0$. Now Proposition~\ref{prop:2_hdby} implies that $\mathcal{S}_0^2(X;L;\alpha)=0$ for any $\alpha$, as $KhR_2(K(r+\alpha^+,r+\alpha^-)\cup L)\{-|\alpha|-2r\}$ is supported in $q+h\le-|\alpha|-2r\to-\infty$ as $r\to\infty$ yet the map $\phi$ in Proposition~\ref{prop:2_hdby}(ii) preserves the grading $q+h$.
\end{proof}

\begin{Ex}\label{ex:unknot_van}
The unknot $K=U$ has maximal $tb$ number $-1$. The $n$-trace on $U$ is $D(n)$, the $D^2$-bundle over $S^2$ with Euler number $n$, which is also a tubular neighborhood of an embedded sphere in a $4$-manifold with self-intersection $n$. Therefore in this case, Theorem~\ref{thm:van} says $\mathcal{S}_0^2(X;L)=0$ whenever $X$ contains an embedded sphere with positive self-intersection. This is an adjunction-type obstruction for $\mathcal S_0^2$ to be nontrivial.

For example, $D(n)$ with $n>0$, $\CP^2$, $S^2\times S^2$, or the mirror image of the $K3$ surface have vanishing Khovanov skein lasagna modules.
\end{Ex}

\begin{Ex}
The right handed trefoil $-3_1$ has $TB(-3_1)=1$. Therefore, the $(\pm1)$-trace on $3_1$ both have vanishing Khovanov skein lasagna modules. The boundary of this $4$-manifold is the Poincar\'e homology sphere $\Sigma(2,3,5)$ for $-1$, and the mirror image of the Brieskorn sphere $\Sigma(2,3,7)$ for $+1$.
\end{Ex}

One way to see embedded traces is through a Kirby diagram of the $4$-manifold $X$. If there is an $n$-framed $2$-handle attached along some knot $K$ not going over $1$-handles, then the $2$-handle together with the $4$-ball is a copy of $X_n(K)$ inside $X$. Consequently, if we see such a $2$-handle with $n\ge-TB(-K)$, then $\mathcal{S}_0^2(X;L)=0$ for any $L$.\smallskip

\begin{Rmk}
As proved by \cite[Theorem~1.3]{manolescu2022skein} (and will be reproved more generally in Section~\ref{sbsec:-CP^2}), $k\overline{\CP^2}$ has nonvanishing skein lasagna module. Therefore, we obtain a slice obstruction from Corollary~\ref{cor:slice_obstruction}, which states that if $\Sigma$ is a slice disk in $k\overline{\CP^2}\backslash int(B^4)$ for a knot $K$, then $$TB(K)+[\Sigma]^2<0.$$ In this special case, however, the obstruction is somewhat weaker than the obstruction coming from the Rasmussen $s$-invariant from \cite[Corollary~1.5]{ren2023lee} which implies $s(K)\le-[\Sigma]^2-|[\Sigma]|$. Since $TB(K)<s(K)$ \cite[Proposition~4]{plamenevskaya2006transverse}, this implies $TB(K)+[\Sigma]^2+|[\Sigma]|<0$.
\end{Rmk}

\section{Lee skein lasagna modules and lasagna \texorpdfstring{$s$}{s}-invariants}\label{sec:Lee_lasagna}
In this section, we explore the structure of the Lee skein lasagna module for an arbitrary pair $(X,L)$ of $4$-manifold $X$ and framed oriented link $L\subset\partial X$. In particular, we prove the statements in Section~\ref{sbsec:intro_lee_lasagna} (except Theorem~\ref{thm:shake_genus}) in a more general setup.

Recall from Section~\ref{sbsec:Kh_lasagna_properties} that the Lee skein lasagna module is defined to be the skein lasagna module with $KhR_{Lee}$ as the link homology theory input. For a pair $(X,L)$, its Lee skein lasagna module $\mathcal S_0^{Lee}(X;L)$ is a vector space over $\Q$ with a homological $\Z$-grading, a quantum $\Z/4$-grading, a homology class grading by $H_2^L(X)$, and an (increasing) quantum filtration (so that $0$ has filtration degree $-\infty$). Since the filtration degree of an element can decrease under a filtered map, a nonzero element in the colimit $\mathcal S_0^{Lee}(X;L)$ may also have filtration degree $-\infty$. Let $q\colon\mathcal S_0^{Lee}(X;L)\to\Z\sqcup\{-\infty\}$ denote the quantum filtration function.

\subsection{Structure of Lee skein lasagna module} \label{sbsec:Lee_structure}

For a pair $(X,L)$, let $L_i$ be the components of $L$ and $[L_i]\in H_1(L)$ be the fundamental class of $L_i$. Let $\partial\colon H_2(X,L)\to H_1(L)$ be the boundary homomorphism. Define the set of \textit{double classes} in $X$ rel $L$ to be $$H_2^{L,\times2}(X):=\left\{(\alpha_+,\alpha_-)\in H_2(X,L)^2\colon\partial\alpha_\pm=\sum_i\epsilon_{i,\pm}[L_i],\,\epsilon_{i,\pm}\in\{0,1\},\,\epsilon_{i,+}+\epsilon_{i,-}=1\right\}.$$
Double classes are precisely those pairs of second homology classes which can be represented by double skeins, to be defined below. First though we present the following generalization of Theorem~\ref{thm:intro_Lee_structure}, the proof of which will occupy the remainder of this section.

\begin{Thm}\label{thm:Lee_structure}
The Lee skein lasagna module of $(X,L)$ is $\mathcal{S}_0^{Lee}(X;L)\cong\Q^{H_2^{L,\times2}(X)}$, with a basis consisting of canonical generators $x_{\alpha_+,\alpha_-}$, one for each double class $(\alpha_+,\alpha_-)\in H_2^{L,\times2}(X)$. Moreover,\vspace{-5pt}
\begin{enumerate}[(1)]
\item $x_{\alpha_+,\alpha_-}$ has homological degree $-2\alpha_+\cdot\alpha_-$;
\item $x_{\alpha_+,\alpha_-}$ has homology class degree $\alpha:=\alpha_++\alpha_-$;
\item $x_{\alpha_+,\alpha_-}\pm x_{\alpha_-,\alpha_+}$ has quantum $\Z/4$ degrees $-\alpha^2-\#L-1\pm1$;
\item $q(x_{\alpha_+,\alpha_-})=q(x_{\alpha_-,\alpha_+})=\max(q(x_{\alpha_+,\alpha_-}\pm x_{\alpha_-,\alpha_+}))$, which also equals $\min(q(x_{\alpha_+,\alpha_-}\pm x_{\alpha_-,\alpha_+}))+2$ if $\alpha_+\ne\alpha_-$ in $H_2(X,L;\Q)$ (in particular if $L\ne\emptyset$).
\end{enumerate}
\end{Thm}
Here, the bilinear product on $H_2(X,L)$ is defined using the framing of $L$.  When $X=B^4$, the bilinear product computes the linking numbers of the boundary links and thus, in view of Proposition~\ref{prop:recover_Z}, Theorem~\ref{thm:Lee_structure} recovers classical structural results on $KhR_{Lee}(L)$ by Lee \cite[Proposition~4.3]{lee2005endomorphism} and Rasmussen \cite[Proposition~2.3,3.3,Lemma~3.5]{rasmussen2010khovanov} \cite[Section~6.1]{beliakova2008categorification}.\medskip

Throughout this section we will make use of the following notations. Given a skein $\Sigma$ in $X$ rel $L$, we continue to use the notation $KhR_{Lee}(\Sigma)$ as in \eqref{eq:KhR(Sigma)}. We also let $\Sigma^\circ$ denote the skein obtained from $\Sigma$ by deleting a local ball around a point on each component of $\Sigma$ that intersected no input balls. In this way every component of $\Sigma^\circ$ now intersects some input ball and moreover we have a morphism $[D]:\Sigma\to \Sigma^\circ$ in $\mathcal{C}(X;L)$ consisting of a disc in each of the new input balls.  We define $v^\circ$ to be $KhR_{Lee}([D])(v)$, i.e. $v^\circ$ assigns the label $1$ in the homology of each new input unknot in $\Sigma^\circ$. Thus we have $[(\Sigma,v)]=[(\Sigma^\circ,v^\circ)]$ in $\mathcal{S}_0^{Lee}(X;L)$.

We also define a \textit{double skein} in $X$ rel $L$ to be a skein $\Sigma$ together with a partition of its components $\Sigma=\Sigma_+\cup\Sigma_-$. Equivalently, a double skein is a pair $(\Sigma,\mathfrak O)$ where $\Sigma$ is a skein (which is already oriented) and $\mathfrak O$ is an orientation of $\Sigma$ as an unoriented surface. The equivalence is given by defining $\Sigma_+$ to be the union of components of $\Sigma$ on which $\mathfrak O$ matches the given orientation of $\Sigma$. A double skein $\Sigma=\Sigma_+\cup\Sigma_-$ represents a double class $([\Sigma_+],[\Sigma_-])\in H_2^{L,\times2}(X)$.  Finally, a double skein structure $\mathfrak O$ on $\Sigma$ determines one on $\Sigma^\circ$ in the obvious way, still denoted $\mathfrak O$.

\begin{Def}\label{def:can Lee lasagna filling}
The \textit{canonical Lee lasagna filling} of a double skein $(\Sigma,\mathfrak O)$ is the lasagna filling
\[x(\Sigma,\mathfrak O) = x(\Sigma_+,\Sigma_-) :=(\Sigma^\circ, 2^{-n(\Sigma^\circ)}(-1)^{-n(\Sigma_-^\circ)} \bigotimes_i x_{\mathfrak O|_{K_i}})\]
which assigns to each input link $K_i$ of $\Sigma^\circ$ the (rescaled) canonical Lee generator associated to the orientation $\mathfrak O|_{K_i}$ before adding a global rescaling as indicated (here $n(\cdot)$ is defined as in \eqref{eq:n_renormalization}; note that $n(\Sigma^\circ)=n(\Sigma)$, $n(\Sigma_\pm^\circ)=n(\Sigma_\pm)$), as well as a filtration degree shift $-\chi(\Sigma^\circ)$ as in \eqref{eq:KhR(Sigma)} (suppressed from the notation).
\end{Def}

\begin{Lem}\label{lem:can_Lee_las_generates}
Let $\Sigma$ be a skein in $X$ rel $L$.  For every $v\in KhR_{Lee}(\Sigma)$, we have a unique decomposition
\begin{equation}\label{eq:filling as lin combo of pures}
(\Sigma^\circ,v^\circ) = \sum_{\mathfrak{O}} \lambda_{\mathfrak{O}} x(\Sigma,\mathfrak{O}) + \cdots
\end{equation}
where the sum runs over various double skein structures (orientations) $\mathfrak{O}$ on $\Sigma$ (or equivalently on $\Sigma^\circ$), the $\lambda_\mathfrak{O}$'s are scalars, and $\cdots$ indicates terms that are zero in $\mathcal{S}_0^{Lee}(X;L)$.

In particular, any class in $\mathcal{S}_0^{Lee}(X;L)$ is a linear combination of classes of canonical Lee lasagna fillings, and there are no relations among such fillings arising from relations within $KhR_{Lee}(\Sigma^\circ)$ for any fixed skein $\Sigma$.
\end{Lem}
\begin{proof}
Since the Lee homology of a link has canonical generators as a basis, we may write $(\Sigma^\circ,v^\circ)$ as a unique linear combination of fillings of the form $(\Sigma^\circ, \bigotimes_i x_{\mathfrak{o}_i})$ where each $\mathfrak{o}_i$ is some orientation of the input link $K_i$ for $\Sigma^\circ$ viewed as an unoriented link; we shall refer to such fillings as \emph{pure Lee lasagna fillings} of $\Sigma^\circ$.

Then if for some pure Lee lasagna filling $(\Sigma^\circ, \bigotimes_i x_{\mathfrak{o}_i})$ there is no orientation $\mathfrak{O}$ of $\Sigma^\circ$ (or equivalently, of $\Sigma$) that gives rise to the various $\mathfrak{o}_i$, then there must be two components $C_1,C_2$ on some input links $K_{i_1},K_{i_2}$ (possibly $i_1=i_2$) on the same component of $\Sigma^\circ$ with incompatible orientations.  We then puncture this component of $\Sigma^\circ$ once more to arrive at a new skein, denoted $\Sigma^{\circ\circ}$, with
\[
[(\Sigma^\circ,\bigotimes_i x_{\mathfrak{o}_i})] = [(\Sigma^{\circ\circ},1\otimes \bigotimes_i x_{\mathfrak{o}_i})] = [(\Sigma^{\circ\circ},\A \otimes \bigotimes_i x_{\mathfrak{o}_i})] + [(\Sigma^{\circ\circ},\B \otimes \bigotimes_i x_{\mathfrak{o}_i})],
\]
where the labels $1,\A,\B$ indicate labels on the new input unknot $U$ in $\Sigma^{\circ\circ}$ (see Section~\ref{sbsec:KhR_2}). Then by choosing a path along $\Sigma^{\circ\circ}$ from $U$ to $C_1$ (respectively to $C_2$), we can enclose $U$ with $K_{i_1}$ (respectively with $K_{i_2}$) in a new input ball inducing a morphism in $\mathcal{C}(X;L)$ which connects $U$ with $C_1$ (respectively with $C_2$).  Since $\A$ and $\B$ correspond to two different orientations of $U$, \eqref{eq:can_gen} shows that one of these morphisms will force $[(\Sigma^{\circ\circ},\A \otimes \bigotimes_i x_{\mathfrak{o}_i})]=0$, and the other will force $[(\Sigma^{\circ\circ},\B \otimes \bigotimes_i x_{\mathfrak{o}_i})]=0$, in $\mathcal{S}_0^{Lee}(X;L)$.
\end{proof}

If $[S]\colon\Sigma_1\to\Sigma_2$ is a morphism in the category of skeins $\mathcal C(X;L)$, then a double skein structure $\mathfrak{O}$ on $\Sigma_1$ (or equivalently on $\Sigma_1^\circ$) induces a double skein structure $\mathfrak{O}|_2$ on $\Sigma_2$ (or equivalently on $\Sigma_2^\circ$) by restricting the orientation, with corresponding decompositions $\Sigma_1=\Sigma_{1+}\cup\Sigma_{1-}$ and $\Sigma_2=\Sigma_{2+}\cup \Sigma_{2-}$ satisfying $([\Sigma_{1+}],[\Sigma_{1-}])=([\Sigma_{2+}],[\Sigma_{2-}])\in H_2^{L,\times 2}(X)$.

\begin{Lem}\label{lem:can_Lee_las_morphism}
Let $[S]\colon\Sigma_1\to\Sigma_2$ be a morphism in $\mathcal C(X;L)$, and let $v_1\in KhR_{Lee}(\Sigma_1)$ with $v_2=KhR_{Lee}([S])(v_1)\in KhR_{Lee}(\Sigma_2)$.  Then if we decompose the lasagna filling $(\Sigma_1^\circ,v_1^\circ)$ as in \eqref{eq:filling as lin combo of pures}
\begin{equation}\label{eq:pures for v1}
(\Sigma_1^\circ,v_1^\circ) = \sum_{\mathfrak{O}} \lambda_{\mathfrak{O}} x(\Sigma_1,\mathfrak{O}) + \cdots,
\end{equation}
we have
\begin{equation}\label{eq:pures for v2}
(\Sigma_2^\circ,v_2^\circ) = \sum_{\mathfrak{O}} \lambda_{\mathfrak{O}} x(\Sigma_2,\mathfrak{O}|_2) + \cdots,
\end{equation}
where the sum is still over double skein structures (orientations) $\mathfrak{O}$ of $\Sigma_1$, the scalars $\lambda_{\mathfrak{O}}$ are equal to those in the decomposition for $(\Sigma_1^\circ,v_1^\circ)$, and $\cdots$ continues to indicate terms which are zero in $\mathcal{S}_0^{Lee}(X;L)$.

In particular, if we view $\mathcal{S}_0^{Lee}(X;L)$ as generated by the classes represented by canonical Lee lasagna fillings, then all relations between such classes are generated by relations of the form $[x(\Sigma_1,\mathfrak{O})] = [x(\Sigma_2,\mathfrak{O}|_2)]$.
\end{Lem}
\begin{proof}
The morphism $[S]$ is represented by a surface $S$ together with an isotopy rel boundary $\phi$ from $\Sigma_2\cup S$ to $\Sigma_1$. We may decompose any such morphism into a composition $\Sigma_1 \xrightarrow{[\mathrm{id},\phi]}\Sigma_2\cup S \xrightarrow{[S,\mathrm{id}]} \Sigma_2$ where $[\mathrm{id},\phi]$ is represented by the `trivial surface' formally consisting of the input links (and maintaining all input balls) and the isotopy $\phi$, while $[S,\mathrm{id}]$ is represented by the surface $S$ and the identity isotopy. The trivial surface morphism acts by identity on $KhR_{Lee}(\Sigma_1)=KhR_{Lee}(\Sigma_2\cup S)$, and thus for the remainder of the proof we will assume our morphism $[S]$ gives rise to a surface $S$ such that $\Sigma_1=\Sigma_2\cup S$.

We let $D_i\colon \Sigma_i\to \Sigma_i^\circ$ denote the surface consisting of discs in the new input balls as before.  Now by definition the lasagna filling $(\Sigma^\circ_2,v^\circ_2)$ arises from $(\Sigma_1,v_1)$ via the following sequence of morphisms (we abuse notation here slightly):
\[(\Sigma_1,v_1) \xmapsto{KhR_{Lee}([S])} (\Sigma_2,v_2) \xmapsto{KhR_{Lee}([D_2])} (\Sigma_2^\circ,v_2^\circ).\]
However, since $D_2\subset\Sigma_2$ is disjoint from $S$, we may instead write \[(\Sigma_1,v_1) \xmapsto{KhR_{Lee}([D_2])} (\Sigma_2^\circ \cup S, z) \xmapsto{KhR_{Lee}([S])} (\Sigma_2^\circ,v_2^\circ)\]
for $z=KhR_{Lee}([D_2])(v_1)\in KhR_{Lee}(\Sigma_2^\circ \cup S)$.  Finally, since any component of $\Sigma_1$ intersecting no input ball is either disjoint from $S$ (and so is a component of $\Sigma_2$ intersecting no input ball), or intersects $S$ (in which case the corresponding component(s) of $\Sigma_2$ intersect input balls), we can choose our punctures such that $D_1\cap \Sigma_2 = D_2$.  This allows us to further decompose
\[
(\Sigma_1,v_1) \xmapsto{KhR_{Lee}([D_2])} (\Sigma_2^\circ \cup S, z) \xmapsto{KhR_{Lee}([D_1 \backslash D_2])} (\Sigma_1^\circ,v_1^\circ) 
\xmapsto{KhR_{Lee}([S^\circ])} (\Sigma_2^\circ,v_2^\circ),
\]
where $S^\circ$ is by definition $\overline{S\backslash(D_1 \cap S)}$.  Thus to prove the lemma we may assume without loss of generality that $(\Sigma_i,v_i)=(\Sigma_i^\circ,v_i^\circ)$ and that $S=S^\circ$.

This is now a more or less straightforward consequence of \eqref{eq:can_gen}. First, for any pure Lee lasagna filling $(\Sigma_1,\bigotimes_i x_{\mathfrak{o}_i})$ that does not arise from a double skein structure, and thus contributes to the ellipsis $\cdots$ in \eqref{eq:pures for v1} (as in the proof of Lemma \ref{lem:can_Lee_las_generates}), either $S$ connects two incompatibly oriented components of input links, or it does not.  In the first case, $KhR_{Lee}([S])(\bigotimes_i x_{\mathfrak{o}_i}) = 0$ via \eqref{eq:can_gen}; in the second case, $KhR_{Lee}([S])(\bigotimes_i x_{\mathfrak{o}_i})$ is a linear combination of terms which still contain two incompatibly oriented components of input links, all of which contribute to the ellipsis $\cdots$ in \eqref{eq:pures for v2} (again as in the proof of Lemma \ref{lem:can_Lee_las_generates} using \eqref{eq:can_gen}).

Finally, for any fixed double skein structure (orientation) $\mathfrak{O}$ of $\Sigma_1$, \eqref{eq:can_gen} shows $KhR_{Lee}([S]) (x(\Sigma_1,\mathfrak{O}))$ is a linear combination of terms coming from various orientations of $S$ whose restrictions on the input links of $\Sigma_1$ agree with the restriction of $\mathfrak O$. The restriction $\mathfrak{O}|_{S}$ gives rise to some scalar multiple of the desired term $x(\Sigma_2,\mathfrak{O}|_2)$ in $KhR_{Lee}([S])(x(\Sigma_1,\mathfrak{O}))$; one can use the additivity of $n(\cdot)$ to show that the scalar is indeed one.  We thus have from \eqref{eq:can_gen}
\[KhR_{Lee}([S]) (x(\Sigma_1,\mathfrak{O})) = x(\Sigma_2,\mathfrak{O}|_2) + \sum_{\mathfrak{O}'_S} \lambda_{\mathfrak{O}'_S} \left(\Sigma_2 \,,\, \bigotimes _jx_{\mathfrak{O}'_S|_{K_j'}} \right)
\]
where the sum runs over allowable orientations $\mathfrak{O}_S'$ of $S$ different from $\mathfrak O|_S$, the $\lambda_{\mathfrak{O}'_S}$ are scalars, and the $K_j'$ denote input links of $\Sigma_2$ (recall that we allow some components of $S$ to consist of curves where input balls of $\Sigma_1$ and $\Sigma_2$ coincide; on such components, the only allowable orientation is the restriction of $\mathfrak{O}$). To complete the proof, one uses the fact that every component $\Sigma_1$ intersects some input ball to conclude that, for any such allowable $\mathfrak{O}'_S$, the filling $(\Sigma_2 \,,\, \bigotimes _jx_{\mathfrak{O}'_S|_{K_j'}})$ has two input link components which have been assigned incompatible orientations, and thus represents zero in $\mathcal{S}_0^{Lee}(X;L)$ as in the proof of Lemma \ref{lem:can_Lee_las_generates}.
\end{proof}

\begin{Rmk}\label{rmk:can Lee filling equivalent to any number of puncture}
One consequence of Lemma \ref{lem:can_Lee_las_morphism} is that, for any double skein $(\Sigma,\mathfrak{O})$, the canonical Lee lasagna filling $x(\Sigma,\mathfrak{O})$ is equivalent in $\mathcal{S}_0^{Lee}(X;L)$ to the canonical Lee lasagna filling $x(\Sigma^{\circ\circ},\mathfrak{O})$ where $\Sigma^{\circ\circ}$ is any skein obtained from $\Sigma^\circ$ by introducing any number of additional local input balls (with new input unknots), with orientation inherited from $\mathfrak{O}$ (and given same notation).  We note that each of these input unknots is assigned the label $\A$ or $\B$ according to whether it occurs on $\Sigma^{\circ\circ}_+(\mathfrak{O})$ or $\Sigma^{\circ\circ}_-(\mathfrak{O})$ (and that this partition of $\Sigma^{\circ\circ}$ is also directly inherited from that of $\Sigma$), which satisfy $\A^k=\A,\B^k=\B$, and $\A\B=0$.
\end{Rmk}

\begin{Rmk}
With Remark \ref{rmk:can Lee filling equivalent to any number of puncture} in mind, one may also view labels on input unknots as decorations on skeins (which may be multiplied upon collision), in turn viewing any canonical Lee lasagna filling $x(\Sigma, \mathfrak O) = x(\Sigma_+,\Sigma_-)$ as a lasagna filling with skein $\Sigma$, which assigns the canonical generators to each input link of $\Sigma$ as in Definition \ref{def:can Lee lasagna filling}, rescaled by $2^{-n(\Sigma)}(-1)^{-n(\Sigma_-)}$ overall, and where each component of $\Sigma_+$ (resp. $\Sigma_-$) carries an extra decoration (indeed any positive number of extra decorations) by $\A$ (resp. $\B$). This perspective can be made precise by allowing dotted skeins. (Note that $\A$ and $\B$ are linear combinations of $1$ and the dot decoration; cf. Remark~\ref{rmk:dotted_skein}.) See also \cite{morrison2024invariants}.
\end{Rmk}

Now we are ready to prove Theorem~\ref{thm:Lee_structure} on the structure of the Lee lasagna module. We begin by establishing the claimed isomorphism.

\begin{proof}[Proof of the isomorphism in Theorem~\ref{thm:Lee_structure}]
We define an augmentation map $$\varepsilon\colon\mathcal S_0^{Lee}(X;L)\to\Q^{H_2^{L,\times2}(X)}$$ by requiring it to send the class represented by a canonical Lee lasagna filling $x(\Sigma_+,\Sigma_-)$ to $e_{([\Sigma_+],[\Sigma_-])}$, the generator of the $([\Sigma_+],[\Sigma_-])$\textsuperscript{th} coordinate of the codomain. Lemmas~\ref{lem:can_Lee_las_generates} and \ref{lem:can_Lee_las_morphism} imply that $\varepsilon$ is well-defined.

We show that $\varepsilon$ is surjective. Any pair $(\alpha_+,\alpha_-)\in H_2^{L,\times2}(X)$ can be represented by properly immersed surfaces $\Sigma_\pm\subset X$ with $\partial\Sigma_\pm$ forming a partition of $L\subset X$ as an oriented unframed link. By transversality we may assume all singularities of $\Sigma=\Sigma_+\cup\Sigma_-$ are transverse double points. Deleting balls around these singularities makes $\Sigma$ embedded in $X$ with some balls deleted, and after deleting one extra local ball on each component of $\Sigma$ we can put a framing on $\Sigma$ compatible with $L$. Now the class represented by the Lee canonical lasagna filling $x(\Sigma_+,\Sigma_-)$ is mapped to $e_{(\alpha_+,\alpha_-)}$, proving the surjectivity of $\varepsilon$.

We remark that the above proof for surjectivity of $\epsilon$ is a replica of that of \eqref{eq:skein_homology_iso} with minor changes. We do the same to prove the injectivity of $\epsilon$ but skip some details. It suffices to show that any $x(\Sigma_+,\Sigma_-)$ and $x(\Sigma'_+,\Sigma'_-)$ with $([\Sigma_+],[\Sigma_-])=([\Sigma'_+],[\Sigma'_-])\in H_2^{L,\times2}(X)$ represent the same element in $\mathcal S_0^{Lee}(X;L)$.  By mimicking the proof of Theorem~\ref{thm:space_of_skeins}, we see that $\Sigma=\Sigma_+\cup\Sigma_-$ and $\Sigma'=\Sigma'_+\cup\Sigma'_-$ are framed cobordant rel $L$ via a framed oriented $3$-manifold $Y_+\cup Y_-$ in the complement of the tubular neighborhood of a $1$-complex in $int(X)\times I$, where $Y_\pm$ is a framed cobordism rel $L$ between $\Sigma_\pm$ and $\Sigma'_\pm$ in the sense of Section~\ref{sbsec:skein_structure}. This in turn implies that $\Sigma$ and $\Sigma'$ are related by morphisms in $\mathcal C(X;L)$ (or their reverses) compatible with the double skein structures. Namely, we can find a zigzag of morphisms in $\mathcal C(X;L)$ $$\Sigma=\Sigma_0\leftarrow\Sigma_1\rightarrow\Sigma_2\leftarrow\cdots\rightarrow\Sigma_{2k}=\Sigma'$$ and double skein structures on each $\Sigma_i$ compatible with $\Sigma,\Sigma'$ and the morphisms. Taking the corresponding canonical Lee lasagna fillings and applying Lemma~\ref{lem:can_Lee_las_morphism}, we obtain a zigzag $$[x(\Sigma_+,\Sigma_-)]=[x(\Sigma_{0+},\Sigma_{0-})]\mapsfrom [x(\Sigma_{1+},\Sigma_{1-})]\mapsto\cdots\mapsto [x(\Sigma_{2k+},\Sigma_{2k-})]=[x(\Sigma'_+,\Sigma'_-)].$$This proves the injectivity of $\varepsilon$.
\end{proof}

The proof above allows us to formally define the basis elements described in Theorem \ref{thm:Lee_structure} as follows.

\begin{Def}\label{def:can Lee lasagna gens}
Given a double class $(\alpha_+,\alpha_-)\in H_2^{L,\times 2}(X)$ in $X$ rel $L$, the \emph{canonical Lee lasagna generator} $x_{\alpha_+,\alpha_-}$ of $\mathcal S_0^{Lee}(X;L)$ is defined to be $[x(\Sigma_+,\Sigma_-)]$ for any double skein $\Sigma=\Sigma_+\cup\Sigma_-$ representing $(\alpha_+,\alpha_-)$.
\end{Def}

\begin{proof}[Proof of additional items in Theorem~\ref{thm:Lee_structure}]
(2) is clear since $[\Sigma^\circ]=[\Sigma]=[\Sigma_+]+[\Sigma_-]=\alpha_++\alpha_-=\alpha$. (1) and (3) concern grading information inherited from the grading on Lee homology of the input links of a skein (up to a shift in the case of quantum degree), and thus our proofs of these items will make use of Proposition \ref{prop:Lee_generator_properties}(3)(4) about gradings of $KhR_{Lee}$.

Denote by $K_i$'s the input links of $\Sigma^\circ$, which also admit decompositions $K_{i+}\cup K_{i-}$ induced from the decomposition $\Sigma^\circ_+\cup\Sigma^\circ_-$. The Lee canonical generator $x_{\mathfrak O|_{K_i}}$ assigned to $K_i$ in the filling $x(\Sigma_+,\Sigma_-)$ has homological degree $-2\ell k(K_{i+},K_{i-})$. This proves (1) since $\sum\ell k(K_{i+},K_{i-})=[\Sigma_+^\circ]\cdot[\Sigma_-^\circ] = [\Sigma_+]\cdot[\Sigma_-]$.  To prove (3), note that the class $x_{\alpha_-,\alpha_+}$ is represented by $x(\Sigma_-,\Sigma_+)$ associated to the reverse double skein $\Sigma=\Sigma_-\cup\Sigma_+$, which assigns canonical generators $x_{\overline{\mathfrak O|_{K_i}}}$ to each input knot of $\Sigma^\circ$, which themselves are conjugates of $x_{\mathfrak O|_{K_i}}$ (meaning that they are equal in quantum $\Z/4$ degree $-w(K_i)-\#K_i$ and negatives of each other in degree $-w(K_i)-\#K_i-2$). Taking into account the extra renormalization factor $(-1)^{n(\cdot)}$ and the quantum shift $-\chi(\Sigma^\circ)$ as in Definition~\ref{def:can Lee lasagna filling}, we see that $x(\Sigma_+,\Sigma_-)$ and $x(\Sigma_-,\Sigma_+)$ agree in quantum $\Z/4$ degree 
\begin{align*}
&\sum(-w(K_i)-\#K_i)-\chi(\Sigma^\circ)-2(n(\Sigma^\circ_+)+n(\Sigma^\circ_-))\\
=&\,(-\alpha^2-\#\partial_-\Sigma^\circ)-\chi(\Sigma^\circ)+\chi(\Sigma^\circ)-\#L+\#\partial_-\Sigma^\circ\\
=&\,-\alpha^2-\#L,
\end{align*}
and are negatives of each other in quantum $\Z/4$ degree $-\alpha^2-\#L-2$, proving (3).

Finally we prove (4) which is more subtle. If both $q(x_{\alpha_+,\alpha_-})$ and $q(x_{\alpha_-,\alpha_+})$ are $-\infty$, the statement is trivial. By symmetry, we henceforth assume that $q(x_{\alpha_+,\alpha_-})=q>-\infty$.

Choose a Lee lasagna filling $(\Sigma,v)$ representing $x_{\alpha_+,\alpha_-}$ in filtration level $q(\Sigma,v)=q$.  As discussed above, $(\Sigma,v)$ is equivalent to a corresponding Lee lasagna filling (in the same filtration level) $(\Sigma^\circ,v^\circ)$ with labels $1$ on the extra input unknots.  Define a conjugation action $\iota$ on $KhR_{Lee}(\Sigma^\circ)$ by $1$ on the quantum $\Z/4$ degree $-\alpha^2-\#L$ part and $-1$ on the other part.

As in the proof of Lemma \ref{lem:can_Lee_las_generates}, write $(\Sigma^\circ,v^\circ)$ as a linear combination of pure Lee lasagna fillings $(\Sigma^\circ,v^\circ_{\mathfrak{o}})$, where $v^\circ_{\mathfrak{o}}=\bigotimes_i x_{\mathfrak{o}_i}$ for some orientations $\mathfrak{o}=\{\mathfrak{o}_i\}$ of the input links of $\Sigma^\circ$.  Any such pure filling $(\Sigma^\circ,v^\circ_{\mathfrak{o}})$ represents $c(\mathfrak o)x_{[\Sigma_+(\mathfrak O)],[\Sigma_-(\mathfrak O)]}\in \mathcal{S}_0^{Lee}(X;L)$ for some constant $c(\mathfrak o)$ if the orientations $\mathfrak o_i$ come from a double skein structure $\mathfrak O$ on $\Sigma$, and zero otherwise. In the first case, $[(\Sigma^\circ,\iota v^\circ_\mathfrak o)]=c(\mathfrak o)x_{[\Sigma_-(\mathfrak{O})],[\Sigma_+(\mathfrak{O})]}$ by (3). In the second case $[(\Sigma^\circ,\iota v^\circ_\mathfrak o)]=0$. Summing up, we have $[(\Sigma^\circ,\iota v^\circ)]=x_{\alpha_-,\alpha_+}$, thus $q(x_{\alpha_-,\alpha_+})\le q(\Sigma^\circ,\iota v^\circ)=q(\Sigma^\circ,v^\circ)=q(x_{\alpha_+,\alpha_-})$. Hence, by symmetry, $q(x_{\alpha_+,\alpha_-})=q(x_{\alpha_-,\alpha_+})$. By the triangle inequality this also equals $\max(q(x_{\alpha_+,\alpha_-}\pm x_{\alpha_-,\alpha_+}))$.

We are left to prove, when $\alpha_+\ne\alpha_-$ in $H_2(X,L;\Q)$, that $\min(q(x_{\alpha_+,\alpha_-}\pm x_{\alpha_-,\alpha_+}))=q-2$. Since $x_{\alpha_+,\alpha_-}\pm x_{\alpha_-,\alpha_+}$ have different $\Z/4$ grading, we know $\min(q(x_{\alpha_+,\alpha_-}\pm x_{\alpha_-,\alpha_+}))\le q-2$. To show the reverse inequality, we resort to the following linear algebra lemma.

\begin{Lem}\label{lem:cube_lin_alg}
Let $W\subset\Q^n$ be a linear subspace. Suppose $E_1\sqcup\cdots\sqcup E_m$ is a partition of the cube $E=\{\pm1\}^n\subset\Q^n$ such that each $E_i$ is contained in a coset of $W$, and $E_1=-E_2\not\subset W$. Suppose $\lambda\colon E\to\Q$ satisfies $$\sum_{\epsilon\in E_i}\lambda(\epsilon)=\begin{cases}1,&i=1\\\pm1,&i=2\\0,&i>2.\end{cases}$$ Then there exist $a_1,\cdots,a_n\in\Q$ such that $$\sum_{k=1}^n\sum_{\epsilon\in E_i}\lambda(\epsilon)\epsilon_ka_k=\begin{cases}1,&i=1\\\mp1,&i=2\\0,&i>2.\end{cases}$$
\end{Lem}
\begin{proof}
Pick any $\ell=(a_1,\cdots,a_n)\in W^\perp\subset(\Q^n)^*$ with $\ell(E_1)=\{1\}$.
\end{proof}

We now finish the proof of (4). Choose any lasagna filling $(\Sigma,v)$ representing $x_{\alpha_+,\alpha_-}\pm x_{\alpha_-,\alpha_+}$, where $\pm$ denotes whichever sign $+$ or $-$ that minimizes $q(x_{\alpha_+,\alpha_-}\pm x_{\alpha_-,\alpha_+})$. It suffices to prove that the filtration degree $q(\Sigma,v)$ of $(\Sigma,v)$ is at least $q-2$. This time, following Remark \ref{rmk:can Lee filling equivalent to any number of puncture} we focus on the equivalent filling $(\Sigma^{\circ\circ},v^{\circ\circ})$ where $\Sigma^{\circ\circ}$ includes a new local input ball (with new input unknot labeled by $1$) on every component of $\Sigma$, rather than just those components that did not intersect input balls.  The reader may verify that $q(\Sigma^{\circ\circ},v^{\circ\circ}) = q(\Sigma,v)$.  We again write $(\Sigma^{\circ\circ},v^{\circ\circ})$ as a linear combination of pure Lee lasagna fillings, but we note that (1) implies $h(v^{\circ\circ})=-2\alpha_+\cdot\alpha_-$, so that we may choose our sum to only include terms in this homological degree:
\begin{equation}\label{eq:v=sum_lambda}
(\Sigma^{\circ\circ},v^{\circ\circ})=\sum_{\mathfrak O}\lambda_\mathfrak{O} x(\Sigma^{\circ\circ}_+(\mathfrak{O}),\Sigma_-^{\circ\circ}(\mathfrak{O}))+\cdots
\end{equation}
where $\mathfrak O$ runs over orientations of $\Sigma$ (or equivalently of $\Sigma^{\circ\circ}$) with $[\Sigma_+^{\circ\circ}(\mathfrak O)]\cdot[\Sigma_-^{\circ\circ}(\mathfrak O)] = [\Sigma_+(\mathfrak O)]\cdot[\Sigma_-(\mathfrak O)] = \alpha_+\cdot\alpha_-$, and $\cdots$ denotes linear combinations of pure Lee lasagna fillings $(\Sigma^{\circ\circ},\otimes_i x_{\mathfrak{o}_i})$ that represent the zero class in $\mathcal S_0^{Lee}(X;L)$.

Let $\Sigma_1,\cdots,\Sigma_n$ be the components of $\Sigma$. Identify an orientation $\mathfrak O$ on $\Sigma$ as a vector $\mathfrak O\in E=\{\pm1\}^n\subset\Q^n$, with $\mathfrak O_i=1$ if and only if $\mathfrak O$ agrees with the orientation of $\Sigma$ on $\Sigma_i$; in this way the coefficients $\lambda_{\mathfrak O}$ in \eqref{eq:v=sum_lambda} define a function $\lambda:E\to\Q$. Partition $E$ into $E_1\sqcup\cdots\sqcup E_m$ according to the value of $[\Sigma_+(\mathfrak O)]\in H_2(X,L)$. Then each $E_i$ is contained in a coset of $W=\{(k_1,\cdots,k_n)\in\Q^n\colon\sum_ik_i[\Sigma_i]=0\in H_2(X,L;\Q)\}\subset\Q^n$. Assume $E_1,E_2$ are the parts with $[\Sigma_+(\mathfrak O)]=\alpha_+,\alpha_-$, respectively. Then $E_1,E_2$ satisfy the condition in Lemma~\ref{lem:cube_lin_alg} (they are not contained in $W$ as $\alpha_+\ne\alpha_-\in H_2(X,L;\Q)$). Also, the condition $[(\Sigma^{\circ\circ},v^{\circ\circ})]=x_{\alpha_+,\alpha_-}\pm x_{\alpha_-,\alpha_+}$ is exactly the condition on $\lambda\colon E\to\Q$ in Lemma~\ref{lem:cube_lin_alg}. Let $a_1,\cdots,a_n$ be given by the conclusion of the lemma applied to our data $(W,E_1\sqcup\cdots\sqcup E_m,\lambda)$.  Now consider the operator $L=\sum_{k=1}^na_kX_k$ on $KhR_{Lee}(\Sigma^{\circ\circ})$, where $X_k$ is a dot operation on $\Sigma_k$ (or equivalently, if $U_k$ denotes the extra input unknot on the component $\Sigma_k$ in $\Sigma^{\circ\circ}$, the operator $X_k$ multiplies the label on $U_k$ by $X$). Since $X\A=\A$ and $X\B=-\B$, it follows that $X_k x(\Sigma^{\circ\circ}_+(\mathfrak{O}),\Sigma^{\circ\circ}_-(\mathfrak{O})) = \mathfrak{O}_k x(\Sigma^{\circ\circ}_+(\mathfrak{O}),\Sigma^{\circ\circ}_-(\mathfrak{O}))$ and thus Lemma \ref{lem:cube_lin_alg} implies
\[[(\Sigma^{\circ\circ},L(v^{\circ\circ}))]=[\sum_{k=1}^n\sum_{\mathfrak O}\lambda_\mathfrak O\mathfrak O_ka_kx(\Sigma^{\circ\circ}_+(\mathfrak{O}),\Sigma^{\circ\circ}_-(\mathfrak{O}))]=x_{\alpha_+,\alpha_-}\mp x_{\alpha_-,\alpha_+}.\]
Since the operator $L$ has filtered degree $2$, we conclude that $$q=q(x_{\alpha_+,\alpha_-}\mp x_{\alpha_-,\alpha_+})\le q(\Sigma^{\circ\circ},v^{\circ\circ})+2=q(\Sigma,v)+2,$$ as desired.
\end{proof}

\subsection{Canonical Lee lasagna generators under morphisms}\label{sbsec:can_Lee_gen_mor}

The induced map on Lee homology of a link cobordism is determined by its actions on Lee canonical generators, given by \eqref{eq:can_gen}. In this section, we prove a lasagna generalization of this formula. In its most general form, the relevant cobordism map we need to examine is \eqref{eq:glue_Kh_ver}.

We recall the notations and explain the setup before stating the generalization to \eqref{eq:can_gen}. The $4$-manifold $X=X_1\cup_YX_2$ is obtained by gluing $X_1,X_2$ along some common part $Y,\overline Y$ of their boundaries, where $Y$ is a $3$-manifold possibly with boundary. Two framed oriented tangles $T_1\subset\partial X_1\backslash Y$ and $T_0\subset Y$ glue to a framed oriented link $L_1\subset\partial X_1$, and $S\subset X_2$ is a (not necessarily framed) surface with an oriented link boundary $\partial S=-T_0\cup T_2=L_2\subset\partial X_2$ where $T_2$ is a tangle in $\partial X_2\backslash\overline Y$. Then $L=T_1\cup T_2$ is a framed oriented link in $\partial X$. The gluing of $(X_2,S)$ to $(X_1,L_1)$ along $(Y,T_0)$ induces the cobordism map
\begin{equation}\label{eq:glue_las}
\mathcal S_0^{Lee}(X_2;S)\colon\mathcal S_0^{Lee}(X_1;L_1)\to\mathcal S_0^{Lee}(X;L)
\end{equation}
as in \eqref{eq:glue_Kh_ver}, which has homological degree $0$ and quantum filtration degree $-\chi(S)+\chi(T_0)-[S]^2$. We recall that by deleting local balls on $S$ (one for each component) and putting a compatible framing, the punctured surface $S^\circ$ is a skein in $X_2$ rel $L_2$ with a natural lasagna filling by assigning $1$ to all boundary unknots, and \eqref{eq:glue_las} is defined by putting the class of this lasagna filling in the second tensor summand of the more general gluing map \eqref{eq:glue_2_summands}.

A class $\alpha\in H_2^{L_1}(X_1)$ and a class $\beta\in H_2^{L_2}(X_2)$ glue to a class $\alpha*\beta\in H_2^{L}(X)$. Similarly, a double class $(\alpha_+,\alpha_-)\in H_2^{L_1,\times2}(X_1)$ and a double class $(\beta_+,\beta_-)\in H_2^{L_2,\times2}(X_2)$ glue to a double class $(\alpha_+,\alpha_-)*(\beta_+,\beta_-)=(\alpha_+*\beta_+,\alpha_-*\beta_-)\in H_2^{L,\times2}(X)$ if they are compatible, i.e. if $\partial\alpha_+,\partial\beta_+$ (and thus $\partial\alpha_-,\partial\beta_-$) agree on $T_0$, and do not glue if they are incompatible.

The punctured surface $S^\circ$ represents the class $[S]\in H_2^{L_2}(X_2)$. Thus the cobordism map \eqref{eq:glue_las} decomposes into direct sums of maps $\mathcal S_0^{Lee}(X_1;L_1;\alpha)\to\mathcal S_0^{Lee}(X;L;\alpha*[S])$.

\begin{Thm}\label{thm:can_gen_morphism}
The cobordism map \eqref{eq:glue_las} is determined by \begin{equation}\label{eq:las_can_gen_morphism}
\mathcal S_0^{Lee}(X_2;S)(x_{\alpha_+,\alpha_-})=2^{m(S,L_1)}\sum_{\mathfrak O}(-1)^{m(S_-(\mathfrak O),L_{1-})}x_{\alpha_+*[S_+(\mathfrak O')],\alpha_-*[S_-(\mathfrak O)]}.
\end{equation}
Here the sum runs over double skein structures $\mathfrak{O}$ on $S$ compatible with $(\alpha_+,\alpha_-)$, $m(S,L_1)$ is an integral constant depending only on the topological type of the input pair, which equals $n(S)$ if $S\cap L_1$ is a link, thought of as the negative boundary of $S$, and $L_{1-}$ is the sublink of $L_1$ with fundamental class $\partial\alpha_-$.
\end{Thm}

\begin{proof}
This is almost tautological from the construction.  Let $(\Sigma,\mathfrak{O}')$ be a double skein in $X_1$ rel $L_1$ representing $(\alpha_+,\alpha_-)$. Then $x_{\alpha_+,\alpha_-}$ is represented by the canonical Lee lasagna filling $x(\Sigma,\mathfrak{O}') = 2^{-n(\Sigma^\circ)}(-1)^{-n(\Sigma^\circ_-)}(\Sigma^\circ,\bigotimes_i x_{\mathfrak O'|_{K_i}})$, where $K_i$'s are input links of $\Sigma^\circ$ and $\Sigma_-^\circ=\Sigma_-^\circ(\mathfrak O')$. Gluing in $S$ produces a skein $\Sigma^\circ\cup S^\circ$ in $X$ rel $L$ with corresponding filling (assigning $x_{\mathfrak O'|_{K_i}}$ to each input link of $\Sigma^\circ$ and $1=\A+\B$ to each of the input unknots on $S^\circ$) which we denote by $v_{\mathfrak{O'}}$, so that
\[
\mathcal S_0^{Lee}(X_2;S)(x_{\alpha_+,\alpha_-})= 2^{-n(\Sigma^\circ)}(-1)^{-n(\Sigma^\circ_-)} [(\Sigma^\circ\cup S^\circ, v_{\mathfrak{O'}})].
\]
First we absorb the framing punctures where possible.  On each component of $S^\circ$ which is glued to a component (possibly several) of $\Sigma^\circ$ we choose a path from the corresponding input unknot of $S^\circ$ (with label $1$) to a component of an input link of $\Sigma^\circ$ (every component of which intersects some input ball).  Using the corresponding enclosement relations shows that $(\Sigma^\circ \cup S^\circ, v_{\mathfrak{O}'})$ is equivalent in $\mathcal{S}_0^{Lee}(X;L)$ to a lasagna filling of the form $((\Sigma^\circ \cup S)^\circ, w_{\mathfrak{O}'})$ where only components of $S$ disjoint from $\Sigma$ are punctured, and $w_{\mathfrak{O}'}$ assigns the label $1$ to these input unknots while maintaining the labels $x_{\mathfrak{O'}|_{K_i}}$ on the input links of $\Sigma^\circ$ (the framing of some of these $K_i$'s may have changed, but we keep the same notation).

Now if some component of $S$ connected two components of $\Sigma^\circ$ with incompatible orientations, then $((\Sigma^\circ \cup S)^\circ,w_{\mathfrak{O'}})$ represents zero in $\mathcal{S}_0^{Lee}(X;L)$; meanwhile there would be no orientation $\mathfrak{O}$ of $S$ compatible with $(\alpha_+,\alpha_-)$, making the right-hand side of \eqref{eq:las_can_gen_morphism} zero as well.  Otherwise any component of $S$ which glues to $\Sigma^\circ$ must have a unique orientation compatible with $(\alpha_+,\alpha_-)$ as determined by $\mathfrak{O}'$, and any component of $S$ which does not glue to $\Sigma^\circ$ has two compatible orientations corresponding to the sum $1=\A+\B$ that $w_{\mathfrak{O}'}$ assigns to its input unknot.  Thus $w_{\mathfrak{O}'}$ is actually a sum of canonical Lee generators for double skein structures (orientations) on $(\Sigma^\circ \cup S)^\circ$ which are compatible with $\mathfrak{O}'$, and if we keep track of the constants involved, noting that puncturing a surface does not affect $n(\cdot)$, we see
\[
\mathcal S_0^{Lee}(X_2;S)(x_{\alpha_+,\alpha_-}) = 2^{n(\Sigma\cup S)-n(\Sigma)}\sum_\mathfrak O(-1)^{n(\Sigma_-\cup S_-(\mathfrak O))-n(\Sigma_-)}x_{\alpha_+*[S_+(\mathfrak O)],\alpha_-*[S_-(\mathfrak O)]}.
\]
If $T_0$ is a link then $n(\Sigma\cup S)-n(\Sigma)=n(S)$ by the additivity of $n(\cdot)$. In general, $$n(\Sigma\cup S)-n(\Sigma)=\frac{-\chi(\Sigma\cup S)+\chi(\Sigma)+\#L-\#L_1}{2}=\frac{-\chi(S)+\chi(\partial S\cap L_1)+\#(\partial S\Delta L_1)-\#L_1}{2}$$ only depends on the topological type of $(S,L_1)$. It is an integer since both $n(\Sigma\cup S)$ and $n(\Sigma)$ are. Similarly for the exponent on $-1$.
\end{proof}

\subsection{Lasagna \texorpdfstring{$s$}{s}-invariants and genus bounds}\label{sbsec:las_s}

Let $(X,L)$ be a pair as before. In view of Theorem~\ref{thm:Lee_structure}, we make the following definition.
\begin{Def}\label{def:s}
The \textit{lasagna $s$-invariant} of $(X,L)$ at a double class $(\alpha_+,\alpha_-)\in H_2^{L,\times2}(X)$ is $$s(X;L;\alpha_+,\alpha_-):=q(x_{\alpha_+,\alpha_-})-2\alpha_+\cdot\alpha_-\in\Z\sqcup\{-\infty\}.$$
The \textit{lasagna $s$-invariant} of $(X,L)$ at a class $\alpha\in H_2^L(X)$ is $$s(X;L;\alpha):=s(X;L;\alpha,0)=q(x_{\alpha,0})\in\Z\sqcup\{-\infty\}.$$
\end{Def}
When $L=\emptyset$, we often drop it from the notation of lasagna $s$-invariants.

When $X$ is a $2$-handlebody, the isomorphism \eqref{eq:2-hdby_homology_mod_2} respects the canonical Lee lasagna generators in the sense that $x_{\beta_+,\beta_-}\in\mathcal S_0^{Lee}(X;(L,\partial\beta);\beta)$ is identified (up to renormalizations) with $x_{\alpha_+,\alpha_-}\in S_0^{Lee}(X;L;\alpha)$ for the unique pair $(\alpha_+,\alpha_-)\in H_2^{L,\times2}(X)$ with $\alpha_+-\alpha_-=\beta_+-\beta_-$. In particular, 
\begin{equation}\label{eq:2-hdby_s_double=single}
s(X;L;\alpha_+,\alpha_-)=s(X;(L,\partial(\alpha_+-\alpha_-));\alpha_+-\alpha_-).
\end{equation}
Thus, the double class version of lasagna $s$-invariants in Definition~\ref{def:s} can be recovered from the single class version. When $X$ is not a $2$-handlebody, however, it is not clear that \eqref{eq:2-hdby_homology_mod_2} holds in a filtered sense for Lee skein lasagna modules, so we do not have this conclusion.

We prove a few properties of lasagna $s$-invariants.
\begin{Thm}\label{thm:s_prop}
The lasagna $s$-invariants satisfy the following properties.\vspace{-5pt}
\begin{enumerate}[(1)]
\setcounter{enumi}{-1}
\item\emph{(Empty manifold)} $s(\emptyset;\emptyset;0)=0$.
\item\emph{(Recover classical $s$)} If $*$ denotes the unique element in $H_2^L(B^4)$, then $$s(B^4;L;*)=s_{\mathfrak{gl}_2}(L)+1=-s(-L)-w(L)+1,$$ where $s$ is the classical $s$-invariant defined by Rasmussen \cite{rasmussen2010khovanov} and Beliakova--Wehrli \cite{beliakova2008categorification}.
\item\emph{(Symmetries)} $s(X;L;\alpha_+,\alpha_-)=s(X;L;\alpha_-,\alpha_+)=s(X;L^r;-\alpha_+,-\alpha_-)$, where $L^r$ denotes the orientation reversal of $L$.
\item\emph{(Parity)} $s(X;L;\alpha_+,\alpha_-)\equiv(\alpha_++\alpha_-)^2+\#L\pmod2$ (by convention $-\infty$ has arbitrary parity).
\item\emph{(Connected sums)} If $(X,L=L_1\sqcup L_2)$ is a boundary sum, a connected sum, or the disjoint union of $(X_1,L_1)$ and $(X_2,L_2)$, and $(\alpha_{i,+},\alpha_{i,-})\in H_2^{L_i,\times2}(X_i)$, $i=1,2$, then $$s(X;L;\alpha_{1,+}+\alpha_{2,+},\alpha_{1,-}+\alpha_{2,-})=s(X_1;L_1;\alpha_{1,+},\alpha_{1,-})+s(X_2;L_2;\alpha_{2,+},\alpha_{2,-}).$$
\item\emph{(Reduced connected sum)} If $(X,L=L_1\#_{K_1,K_2}L_2)$ is a (reduced) boundary sum of $(X_1,L_1)$ and $(X_2,L_2)$ performed along framed oriented intervals on components $K_i\subset L_i$, $i=1,2$, and $(\alpha_{i,+},\alpha_{i,-})\in H_2^{L_i,\times2}(X_i)$ are double classes compatible with the boundary sum, then $$s(X;L_1\#_{K_1,K_2}L_2;\alpha_{1,+}*\alpha_{2,+},\alpha_{1,-}*\alpha_{2,-})=s(X_1;L_1;\alpha_{1,+},\alpha_{1,-})+s(X_2;L_2;\alpha_{2,+},\alpha_{2,-})-1.$$
\item\emph{(Genus bound)} If $(X_2,L_2)$ is glued to $(X_1,L_1)$ along some $(Y,T_0)$ to obtain $(X,L)$, and (unframed) $S\subset X_2$ is properly embedded with $\partial S=L_2$, as in the setup of \eqref{eq:glue_las}, such that every component of $S$ has a boundary on $T_0$, then for any compatible partition $S=S_+\cup S_-$ and double class $(\alpha_+,\alpha_-)\in H_2^{L_1,\times2}(X_1)$, $$s(X;L;\alpha_+*[S_+],\alpha_-*[S_-])\le s(X_1;L_1;\alpha_+,\alpha_-)-\chi(S)+\chi(T_0)-[S]^2.$$
\end{enumerate}
\end{Thm}
The special case of Theorem~\ref{thm:s_prop}(5) when $X_1=X_2=B^4$, in view of (1), recovers the connected sum formula $$s(L_1\#_{K_1,K_2}L_2)=s(L_1)+s(L_2)$$ for $s$-invariant of links in $S^3$, a fact that was proved only recently \cite[Theorem~7.1]{manolescu2023generalization}.

Since we performed gluing in its most general form, Theorem~\ref{thm:s_prop}(6) has a rather complicated look. We state some special cases as corollaries, roughly in decreasing order of generality. Note that for $2$-handlebodies, by \eqref{eq:2-hdby_s_double=single}, the double class versions of the genus bounds are equivalent to the single class versions.

\begin{Cor}\label{cor:genus_bound_along_closed}
Suppose $(X_1,L_1\sqcup L_0)$, $(X_2,-L_0\sqcup L_2)$ are two pairs that can be glued along some common boundary $(Y,L_0)$ where $Y$ is a closed $3$-manifold (in particular $L_1,L_2\not\subset Y$). Suppose (unframed) $S\subset X_2$ has $\partial S=-L_0\sqcup L_2$ and each component of $S$ has a boundary on $-L_0$, then for any compatible partition $S=S_+\cup S_-$ and double class $(\alpha_+,\alpha_-)\in H_2^{L_1\sqcup L_0,\times2}(X_1)$, $$s(X_1\cup_YX_2;L_1\sqcup L_2;\alpha_+*[S_+],\alpha_-*[S_-])\le s(X_1;L_1\sqcup L_0;\alpha_+,\alpha_-)-\chi(S)-[S]^2.$$ In particular, for any class $\alpha\in H_2^{L_1\sqcup L_0}(X_1)$, $$s(X_1\cup_YX_2;L_1\sqcup L_2;\alpha*[S])\le s(X_1;L_1\sqcup L_0;\alpha)-\chi(S)-[S]^2.$$
\end{Cor}
\begin{proof}
This is the special case of Theorem~\ref{thm:s_prop}(6) when $Y$ is closed, so $T_0=L_0$ with $\chi(T_0)=0$.
\end{proof}
\begin{Rmk}
Although Theorem~\ref{thm:s_prop}(6) is more general than Corollary~\ref{cor:genus_bound_along_closed}, it is implied by the latter. This is because we can replace $(X_2,S)$ by its union with a collar neighborhood of $(\partial X_1,L_1)$. Nevertheless, the proof of Theorem~\ref{thm:s_prop}(6) is only notationally more complicated. The same remark applies to Theorem~\ref{thm:can_gen_morphism}.
\end{Rmk}

\begin{Cor}\label{cor:genus_bound_cob}
Suppose $(W,S)\colon(Y_1,L_1)\to(Y_2,L_2)$ is an unframed cobordism between framed oriented links in closed oriented $3$-manifolds such that each component of $S$ has a boundary on $L_1$, and $X$ is a $4$-manifold that bounds $Y_1$. Then for any compatible partition $S=S_+\cup S_-$ and double class $(\alpha_+,\alpha_-)\in H_2^{L_1,\times2}(X)$, $$s(X\cup_{Y_1}W;L_2;\alpha_+*[S_+],\alpha_-*[S_-])\le s(X;L_1;\alpha_+,\alpha_-)-\chi(S)-[S]^2.$$ In particular, for any class $\alpha\in H_2^{L_1}(X)$, $$s(X\cup_YW;L_2;\alpha*[S])\le s(X;L_1;\alpha)-\chi(S)-[S]^2.$$
\end{Cor}
\begin{proof}
This is the special case of Corollary~\ref{cor:genus_bound_along_closed} when $\partial X_1=Y$.
\end{proof}

In the special case of Corollary~\ref{cor:genus_bound_cob} when $S\subset S^3\times I$ is a link cobordism between links $L_1,L_2\subset S^3$, each component of which has a boundary on $L_1$, and when $X=B^4$, in view of Proposition~\ref{prop:recover_Z} and Theorem~\ref{thm:s_prop}(1), we have $s(-L_1)\le s(-L_2)-\chi(S)$. Of course we can reverse the orientations to obtain $$s(L_1)\le s(L_2)-\chi(S),$$ which recovers the classical genus bound for link cobordisms from $s$-invariants \cite[(7)]{beliakova2008categorification}\footnote{Note (7) in \cite{beliakova2008categorification} incorrectly missed the condition about components of $S$ having boundaries on ends.}.

\begin{Cor}\label{cor:genus_bound_X_bdy}
Suppose $(X,L)$ is a pair and $S_+,S_-\subset X$ are disjoint properly embedded surfaces with $\partial S_+\cup\partial S_-=L$, then $$2g(S_+)+2g(S_-)\ge s(X;L;[S_+],[S_-])+[S_+]^2+[S_-]^2-\#L.$$ In particular, if $S\subset X$ is properly embedded with $\partial S=L$, then $$2g(S)\ge s(X;L;[S])+[S]^2-\#L.$$
\end{Cor}
\begin{proof}
Delete a local ball on each component of $S_+$ and $S_-$, regard the rest of $X$ and $S_\pm$ as a cobordism and apply Corollary~\ref{cor:genus_bound_cob}, we get $$s(X;L;[S_+],[S_-])\le s((B^4)^{\sqcup(\#S_++\#S_-)};U^{\sqcup(\#S_++\#S_-)};*)-\chi(S^\circ)-[S]^2,$$ where $s((B^4)^{\sqcup(\#S_++\#S_-)};U^{\sqcup(\#S_++\#S_-)};*)=(\#S_++\#S_-)s(B^4;U;*)=\#S_++\#S_-$ by Theorem~\ref{thm:s_prop}(1)(4). The statement follows by simplification.
\end{proof}
In the special case $X=B^4$ and $L=K\subset S^3$ is a knot, the single class version of Corollary~\ref{cor:genus_bound_X_bdy} recovers Rasmussen's slice genus bound for knots \cite[Theorem~1]{rasmussen2010khovanov}.

\begin{Cor}\label{cor:genus_bound_X_closed}
Suppose $X$ is a $4$-manifold, and $S_+,S_-\subset X$ are disjoint closed embedded surfaces, then $$2g(S_+)+2g(S_-)\ge s(X;[S_+],[S_-])+[S_+]^2+[S_-]^2.$$ In particular, if $S\subset X$ is an embedded surface, then $$2g(S)\ge s(X;[S])+[S]^2.$$
\end{Cor}
\begin{proof}
This is the special case of Corollary~\ref{cor:genus_bound_X_bdy} when $L=\emptyset$.
\end{proof}
The \textit{genus function} of a $4$-manifold $X$, $g(X;\cdot)\colon H_2(X)\to\Z_{\ge0}$, is defined by $$g(X;\alpha)=\min\{g(\Sigma)\colon\Sigma\subset X\text{ is a closed embedded surface with }[\Sigma]=\alpha\}.$$

\begin{Cor}\label{cor:genus_bound_simpliest}(Theorem~\ref{thm:intro_genus_bound})
The genus function of a $4$-manifold $X$ has a lower bound given by $$g(X;\alpha)\ge\frac{s(X;\alpha)+\alpha^2}{2}.$$
\end{Cor}
\begin{proof}
This is a reformulation of the second part of Corollary~\ref{cor:genus_bound_X_closed}.
\end{proof}

Lasagna $s$-invariants obstruct embeddings between $4$-manifolds in the following sense.

\begin{Cor}\label{cor:embedded_s}
If $i\colon X\hookrightarrow X'$ is an inclusion between two $4$-manifolds, then for any class $\alpha\in H_2(X)$, $$s(X';i_*\alpha)\le s(X;\alpha).$$
\end{Cor}

\begin{proof}
This is the special case of the single class version of Corollary~\ref{cor:genus_bound_cob} when $L_1=L_2=S=\emptyset$.
\end{proof}

\begin{Cor}\label{cor:las_s_class_0}
For any compact oriented $4$-manifold $X$, we have $s(X;0)\le0$.
\end{Cor}
\begin{proof}
This is a consequence of Corollary~\ref{cor:embedded_s} and Theorem~\ref{thm:s_prop}(0).
\end{proof}

We return to Theorem~\ref{thm:s_prop}. It is an easy consequence of the hard work that we have done in previous sections.
\begin{proof}[Proof of Theorem~\ref{thm:s_prop}]
\begin{enumerate}[(1)]
\setcounter{enumi}{-1}
\item $x_{0,0}\in\mathcal S_0^{Lee}(\emptyset)$ is represented by the Lee lasagna filling $1$ on the empty skein, which has filtration degree $0$.
\item $x_{*,0}\in\mathcal S_0^{Lee}(B^4;L)$ is represented by the Lee lasagna filling on the skein $L\times I\subset S^3\times I=B^4\backslash int(\tfrac12B^4)$ with decoration $x_{\mathfrak o_L}$ on the input link $L$. Under the identification $\mathcal S_0^{Lee}(B^4;L)\cong KhR_{Lee}(L)$ given by Proposition~\ref{prop:recover_Z}, $x_{*,0}$ is equal to the Lee canonical generator $x_{\mathfrak o_L}$, which has filtration degree $s_{\mathfrak{gl}_2}(L)+1=-s(-L)-w(L)+1$ by \eqref{eq:s_renormalize}.
\item The first equality is a consequence of Theorem~\ref{thm:Lee_structure}(4). Reversing the orientation of all skeins defines an isomorphism $\mathcal S_0^{Lee}(X;L)\cong\mathcal S_0^{Lee}(X;L^r)$ which negates the homology class grading but preserves other gradings and the quantum filtration, and exchanges the canonical Lee lasagna generators. This proves the second equality.
\item This is a consequence of Theorem~\ref{thm:Lee_structure}(3).
\item This is a consequence of Proposition~\ref{prop:cntd_sum} and the fact that the canonical Lee lasagna generators behave tensorially under the various sum operations (as can be seen on the canonical Lee lasagna filling level).
\setcounter{enumi}{5}
\item This is a consequence of Theorem~\ref{thm:can_gen_morphism}. Under the condition on the components of $S$ and the compatibility assumption, the right-hand side of \eqref{eq:las_can_gen_morphism} has exactly one term, which has $S_\pm(\mathfrak O)=S_\pm$. The statement follows from the fact that $\mathcal S_0^{Lee}(X_2;S)$ has filtered degree $-\chi(S)+\chi(T_0)-[S]^2$ (cf. the comment after \eqref{eq:glue_las}).
\setcounter{enumi}{4}
\item There is a framed saddle cobordism $(\partial X\times I,S)$ from $(\partial X,L=L_1\#_{K_1,K_2}L_2)$ to $(\partial X,L_1\sqcup L_2)$, where $\partial X$ is bounded by $X=X_1\natural X_2$. Thus Corollary~\ref{cor:genus_bound_cob} (which is a corollary of (6)) and (4) imply $$s(X;L;\alpha_{1,+}*\alpha_{2,+},\alpha_{1,-}*\alpha_{2,-})\ge s(X_1;L_1;\alpha_{1,+},\alpha_{1,-})+s(X_2;L_2;\alpha_{2,+},\alpha_{2,-})-1.$$ To prove the reverse inequality, by Theorem~\ref{thm:Lee_structure}(4), we can choose signs $\epsilon_i=\pm1$, $i=1,2$, such that $$q(x_{\alpha_{i,+},\alpha_{i,-}}+\epsilon_ix_{\alpha_{i,-},\alpha_{i,+}})=s(X_i;L_i;\alpha_{i,+},\alpha_{i,-})-2,\ i=1,2.$$ By Theorem~\ref{thm:can_gen_morphism}, the reverse $S^r$ of $S$ induces $\mathcal S_0^{Lee}(\partial X\times I;S^r)\colon\mathcal S_0^{Lee}(X;L_1\sqcup L_2)\to\mathcal S_0^{Lee}(X;L)$ that maps $(x_{\alpha_{1,+},\alpha_{1,-}}+\epsilon_1x_{\alpha_{1,-},\alpha_{1,+}})\otimes(x_{\alpha_{2,+},\alpha_{2,-}}+\epsilon_2x_{\alpha_{2,-},\alpha_{2,+}})$ to a nonzero multiple of $x_{\alpha_{1,+}*\alpha_{2,+},\alpha_{1,-}*\alpha_{2,-}}\pm x_{\alpha_{1,-}*\alpha_{2,-},\alpha_{1,+}*\alpha_{2,+}}$. Since $\mathcal S_0^{Lee}(\partial X\times I;S^r)$ has filtered degree $1$, it follows that 
\begin{align*}
&\,s(X;L;\alpha_{1,+}*\alpha_{2,+},\alpha_{1,-}*\alpha_{2,-})\\\le&\,q(x_{\alpha_{1,+}*\alpha_{2,+},\alpha_{1,-}*\alpha_{2,-}}\pm x_{\alpha_{1,-}*\alpha_{2,-},\alpha_{1,+}*\alpha_{2,+}})+2\\
\le&\,q(x_{\alpha_{1,+},\alpha_{1,-}}+\epsilon_1x_{\alpha_{1,-},\alpha_{1,+}})+q(x_{\alpha_{2,+},\alpha_{2,-}}+\epsilon_2x_{\alpha_{2,-},\alpha_{2,+}})+3\\
=&\,s(X_1;L_1;\alpha_{1,+},\alpha_{1,-})+s(X_2;L_2;\alpha_{2,+},\alpha_{2,-})-1.\qedhere
\end{align*}
\end{enumerate}
\end{proof}

\begin{Rmk}
All of the constructions in this section should work equally well over any base field of characteristic not equal to two.  If working over $\mathbb{F}_2$, one must instead use the Bar-Natan deformation \cite{bar2005khovanov} rather than the Lee deformation, for which the results still hold with some minor adjustments to the renormalization factors, although the precise statements of Theorem \ref{thm:Lee_structure}(3)(4) would require slight modification, which may also affect the validity of Theorem~\ref{thm:s_prop}(5). On the other hand, in the following sections we make use of the representation theory of the symmetric groups in characteristic zero in order to complete our computations, and it is less clear that the statements would hold over other coefficients. Moreover, note that the proof of the result in \cite{grigsby2018annular} about symmetric group actions on cables requires $2$ to be invertible.
\end{Rmk}

\section{Nonvanishing criteria for \texorpdfstring{$2$}{2}-handlebodies}\label{sec:2-hdby}
The Lee skein lasagna module of a pair $(X,L)$ is never $0$ by Theorem~\ref{thm:Lee_structure}. However, in our setup, it is appropriate to set the following convention.

\textbf{Convention.} \textit{The Lee skein lasagna module $\mathcal S_0^{Lee}(X;L)$ is \emph{vanishing} if the filtration function on it is identically $-\infty$; equivalently, it is vanishing if the lasagna $s$-invariants of $(X,L)$ are all $-\infty$. It is \emph{nonvanishing} if it is not vanishing.}

When $X$ is a $2$-handlebody, the Khovanov or Lee skein lasagna module of a pair $(X,L)$ has a nice formula \eqref{eq:2-hdby}. In this section, we make use of this $2$-handlebody formula to provide some explicit criteria for $\mathcal S_0^2(X;L)$ or $\mathcal S_0^{Lee}(X;L)$ to be nonvanishing. By our discussions in Section~\ref{sec:van} and Section~\ref{sec:Lee_lasagna}, this leads to smooth genus bounds for second homology classes of $4$-manifolds. This section is an expansion of Section~\ref{sbsec:intro_comparison} and Section~\ref{sbsec:intro_nonvan_diagram}.

Throughout the section, we assume that $X$ is the $2$-handlebody obtained by attaching $2$-handles to $B^4$ along a framed link $K=K_1\cup\cdots K_m\subset S^3$, and $L\subset\partial X$ is given by a framed oriented link in $S^3$ disjoint from $K$, as in the setup of \eqref{eq:2-hdby}.

\subsection{Rank inequality between Khovanov and Lee skein lasagna modules}\label{sbsec:rank_ineq}
In this section we prove Theorem~\ref{thm:rank_ineq}, which states that the rank of $\mathcal S_0^2(X;L)$ is bounded below by the dimension of $gr(\mathcal S_0^{Lee}(X;L))$ in every tri-degree $(h,q,\alpha)$. In particular, if $\mathcal S_0^{Lee}(X;L)$ is nonvanishing, so is $\mathcal S_0^2(X;L)$. By Theorem~\ref{thm:van}, this has the following consequence.
\begin{Cor}\label{cor:Lee_van}
If a knot trace $X_n(K)$ embeds into a $4$-manifold $X$ for some knot $K$ and framing $n\ge-TB(-K)$, then $\mathcal S_0^{Lee}(X;L)$ is vanishing for any $L\subset\partial X$. In particular, the lasagna $s$-invariants of $(X,L)$ are all $-\infty$.
\end{Cor}
\begin{proof}
This is a consequence of Theorem~\ref{thm:van} and Theorem~\ref{thm:rank_ineq} when $X=X_n(K)$. The general case follows from this and Proposition~\ref{prop:gluing}.
\end{proof}
Recall that a knot $K\subset S^3$ is $n$-slice in a $4$-manifold $X$ if there exists a framed slice disk in $X\backslash int(B^4)$ of the $n$-framed knot $K\subset S^3=-\partial B^4$.
\begin{Cor}
Suppose some class in $\mathcal S_0^{Lee}(X;L)$ has finite filtration degree. If $K\subset S^3$ is $n$-slice in $X$, then $n>TB(K)$.
\end{Cor}
\begin{proof}
The folklore trace embedding lemma states that $K$ is $n$-slice in $X$ if and only if $X_{-n}(-K)$ embeds in $X$. Therefore, the statement is the contrapositive of the first part of Corollary~\ref{cor:Lee_van} for $(-K,-n)$ in place of $(K,n)$.
\end{proof}

We turn to Theorem~\ref{thm:rank_ineq}. We will only use the $2$-handlebody condition in a mild way, namely that the $2$-handlebody formula \eqref{eq:2-hdby} expresses $\mathcal{S}_0^\bullet(X;L)$ as a filtered colimit instead of an arbitrary colimit \eqref{eq:colim}.

\begin{proof}[Proof of Theorem~\ref{thm:rank_ineq}]
We first rewrite \eqref{eq:2-hdby} into a filtered colimit. The group $S(r):=S_{|\alpha_1|+2r_1}\times\cdots\times S_{|\alpha_m|+2r_m}$ acts on $KhR_\bullet(K(r))$, where $K(r):=K(\alpha^++r,\alpha^-+r)\cup L$. Then \eqref{eq:2-hdby} for $\bullet=Lee$ can be rewritten as
\begin{equation}\label{eq:2-hdby_filtered_colim}
\mathcal{S}_0^{Lee}(X;L;\alpha)\cong\mathrm{colim}_{r\in\Z_{\ge0}^m}KhR_{Lee}(K(r))^{S(r)}\{-|\alpha|-2|r|\},
\end{equation}
as $h$-graded and $q$-filtered vector spaces, where $A^{S(r)}$ denotes the subgroup of $A$ fixed by $S(r)$, and the morphisms are given by dotted annular creation maps composed with symmetrization. A similar formula holds for $\mathcal{S}_0^2(X;L;\Q)$.

Now for a fixed triple $(h,q,\alpha)$, pick any linearly independent vectors $y_1,\cdots,y_n\in\mathcal{S}_{0,h,\le q}^{Lee}(X;L;\alpha)$ whose images in $gr_q(\mathcal S_{0,h}^{Lee}(X;L;\alpha))$ are also linearly independent. Pick representatives $v_i\in KhR_{Lee}^h($ $K(r_i))^{S(r_i)}\{-|\alpha|-2|r_i|\}$ of the $y_i$'s with $q(v_i)=q$ ($q(\cdot)$ denotes the quantum filtration function), $i=1,\cdots,n$. Since \eqref{eq:2-hdby_filtered_colim} is a filtered colimit, by increasing the $r_i$'s if necessary, we may assume $r_1=\cdots=r_n=r$ for some $r\in\Z_{\ge0}^m$.

Let $(CKhR_2(K(r);\Q),d)$ and $(CKhR_{Lee}(K(r)),d_{Lee})$ be the Khovanov-Rozansky $\mathfrak{gl}_2$ and Lee cochain complexes for $K(r)$, where $CKhR_{Lee}(K(r))=CKhR_2(K(r);\Q)$ and $d_{Lee}$ is the sum of $d$ and another term that decreases the quantum degree. We further lift $v_1,\cdots,v_n$ to the chain level to Lee cocycles $a_1,\cdots,a_n\in CKhR_{Lee}^h(K(r))\{-|\alpha|-2|r|\}$ with $a_i=c_i+\epsilon_i$ where $c_i$ is homogeneous of quantum degree $q$, and $q(\epsilon_i)<q$. Since $0=d_{Lee}a_i=dc_i+\epsilon$ where $q(\epsilon)<q$, we see $c_i$ is a Khovanov cocycle.

We claim that any nontrivial linear combination of $c_i$'s represents a nonzero element in $\mathcal{S}_{0,h,q}^2(X;L;\Q)$, which would imply $\mathrm{rank}(\mathcal S_{0,h,q}^2(X;L))\ge n$, proving the theorem. Let $c=\sum\lambda_ic_i$ where not all $\lambda_i$ are zero. To prove $c$ represents a nonzero element, it suffices to check its homology class survives under an arbitrary morphism $KhR_2(K(r);\Q)^{S(r)}\{-|\alpha|-2|r|\}\to KhR_2(K(r');\Q)^{S(r')}\{-|\alpha|-2|r'|\}$ in the filtered system. Since this map is a linear combination of induced maps by link cobordisms, it is induced by a map $\Phi\colon CKhR_2(K(r);\Q)\{-|\alpha|-2|r|\}\to CKhR_2(K(r');\Q)\{-|\alpha|-2|r'|\}$ on the chain level. Similarly $KhR_{Lee}(K(r))^{S(r)}\to KhR_{Lee}(K(r'))^{S(r')}$ is induced by some $\Phi_{Lee}\colon CKhR_{Lee}(K(r))\{-|\alpha|-2|r|\}\to CKhR_{Lee}(K(r'))\{-|\alpha|-2|r'|\}$, where we can assume $CKhR_{Lee}($ $K(r'))\cong CKhR_2(K(r');\Q)$ and $\Phi_{Lee}$ is a sum of $\Phi$ (which is homogeneous) and some another term that decreases the quantum degree. It follows that $a=\sum\lambda_ia_i$ satisfies $\Phi_{Lee}a=\Phi c+\epsilon$ for some $\epsilon$ with $q(\epsilon)<q$. Now, suppose to the contrary that $\Phi c=db$ is a Khovanov coboundary.  Then $\Phi_{Lee}a$ is homologous to $\Phi_{Lee}a-d_{Lee}b$ which has filtration degree less than $q$.  Thus the element  $y=\sum_i\lambda_iy_i\in S_0^{Lee}(X;L)$, which was represented by $[a]\in KhR_{Lee}(K(r))^{S(r)}$, is also represented by $[\Phi_{Lee}a]\in KhR_{Lee}(K(r'))^{S(r')}$ with filtration level less than $q$, implying $q(y)<q$, contradicting the choices of $y_1,\cdots,y_n$.
\end{proof}

\subsection{Comparison results for Lee skein lasagna modules}\label{sbsec:Lee_compare}
In this section, we prove two sample comparison results for Lee skein lasagna modules of $2$-handlebodies, as an easy consequence of the structure theorem of Lee skein lasagna modules (Theorem~\ref{thm:Lee_structure}) and the behavior of the canonical generators under gluing (Theorem~\ref{thm:can_gen_morphism}).

Recall that $(X,L)$ is the pair that arises from $K\cup L\subset S^3$ after attaching 2-handles along $K$. Let $(X',L')$ be another such pair from $K'\cup L'$.  Suppose there is a concordance $C\colon K\to K'$, that is, a genus $0$ component-preserving framed link cobordism in $S^3\times I$ (thus topologically $C$ is a disjoint union of annuli). Then the surgery on $C$ gives a homology cobordism $W$ between the surgery on $K$ and the surgery on $K'$. In particular, $W$ (and thus $C$) induces an isomorphism $\varphi_C\colon H_*(\partial X)\cong H_*(\partial X')$. Moreover, $X'=X\cup_{\partial X}W$. If $A\subset W$ is a concordance between $L$ and $L'$, then $(W,A)$ induces an isomorphism $H_2(X,L)\cong H_2(X',L')$ that restricts to $\phi_A\colon H_2^L(X)\cong H_2^{L'}(X')$.

\begin{Prop}\label{prop:Lee_comparison}
Let $(X,L)$ and $(X',L')$ be obtained from $K\cup L,K'\cup L'\subset S^3$ as above, and let $C\colon K\to K'$ be a concordance that induces the homology cobordism $W\colon\partial X\to\partial X'$.\vspace{-5pt}
\begin{enumerate}
\item If $A\subset W$ is a framed concordance between $L,L'$, then $\mathcal S_0^{Lee}(X;L)\cong\mathcal S_0^{Lee}(X';L')$ as tri-graded (by $\Z\times\Z/4\times H_2^L(X)$, where $H_2^L(X)\cong H_2^{L'}(X')$ via $\phi_A$), quantum filtered vector spaces. Moreover, $s(X;L;\alpha)=s(X';L';\phi_A(\alpha))$ for all $\alpha\in H_2^L(X)$.
\item If $L$ has components $L_1,\cdots,L_k$, $L'$ has components $L_1',\cdots,L_k'$, and $\varphi_C([L_i])=[L_i']$ for all $i$, then $\mathcal{S}_0^{Lee}(X;L)$ is nonvanishing if and only if $\mathcal S_0^{Lee}(X';L')$ is nonvanishing.
\end{enumerate}
\end{Prop}
\begin{proof}
\begin{enumerate}[(1)]
\item Gluing $(W,A)$ to $(X,L)$ defines a filtered (degree $0$) map $\mathcal S_0^{Lee}(W;A)\colon\mathcal S_0^{Lee}(X;L)\to\mathcal S_0^{Lee}(X';L')$. Turning $(W,A)$ upside down gives another filtered map $\mathcal S_0^{Lee}(X';L')\to\mathcal S_0^{Lee}(X;L)$. By Theorem~\ref{thm:can_gen_morphism}, the two maps interchange the corresponding canonical Lee lasagna generators, thus are inverses to each other. The statement follows.
\item By symmetry, it suffices to assume the filtration of $\mathcal S_0^{Lee}(X;L)$ is identically $-\infty$ and prove the same for $\mathcal S_0^{Lee}(X';L')$.\\
\noindent
Under the given condition, we can find framed immersed cobordisms $S_i\colon L_i\to L_i'$ in $W$ that intersect each other transversely. After deleting some local balls, we can regard $S=S_1\cup\cdots\cup S_k$ as a skein in $W$ rel $-L\sqcup L'$. For any double class $(\alpha_+,\alpha_-)\in H_2^{L,\times2}(X)$, there is a unique compatible double skein structure $S=S_+\cup S_-$. Gluing $(W,S)$ with the canonical Lee lasagna filling $x(S_+,S_-)$ to $(X,L)$ maps $x_{\alpha_+,\alpha_-}\in\mathcal S_0^{Lee}(X;L)$ to a nonzero multiple of $x_{\alpha_+*[S_+],\alpha_-*[S_-]}\in\mathcal S_0^{Lee}(X';L')$. Since $q(x_{\alpha_+,\alpha_-})=-\infty$, we see $q(x_{\alpha_+*[S_+],\alpha_-*[S_-]})=-\infty$ as well. Since all canonical Lee lasagna generators of $(X',L')$ arise this way, we conclude that the filtration function on $\mathcal S_0^{Lee}(X';L')$ is identically $-\infty$.\qedhere
\end{enumerate}
\end{proof}

\subsection{Lasagna \texorpdfstring{$s$}{s}-invariant and \texorpdfstring{$s$}{s}-invariants of cables}\label{sbsec:las_s_from_classical}
In this section we show that the lasagna $s$-invariants of $(X,L)$ are determined by the classical $s$-invariants of cables of $K\cup L$. Recall that as in the setup of \eqref{eq:2-hdby}, the choice $K\cup L\subset S^3$ gives an identification $H_2^L(X)\cong\Z^m$.

\begin{Prop}\label{prop:las_s_from_classical}
For any $\alpha\in H_2^L(X)\cong\Z^m$, 
\begin{align*}
s(X;L;\alpha)=&\,\lim_{r\to\infty^m}(s_{\mathfrak{gl}_2}(K(\alpha^++r,\alpha^-+r)\cup L)-2|r|)-|\alpha|+1\\
=&\,-\,\lim_{r\to\infty^m}(s(-(K(\alpha^++r,\alpha^-+r)\cup L))+2|r|)-w(K(\alpha^+,\alpha^-)\cup L)-|\alpha|+1.
\end{align*}
Moreover, the quantity within the limit sign in the first (resp. second) line is nonincreasing (resp. nondecreasing) in $r$.
\end{Prop}

\begin{proof}
The $\mathfrak{gl}_2$ $s$-invariant relates to the classical $s$-invariant via \eqref{eq:s_renormalize}. Since $w(K(\alpha^++r,\alpha^-+r)\cup L)$ is independent of $r$, the claims on the second row follow from the first. The existence of the annular cobordism from $K(\alpha^++r,\alpha^-+r)\cup L$ to $K(\alpha^++r+e_i,\alpha^-+r+e_i)\cup L$ implies that the quantity within the limit sign in the first row is nonincreasing in $r$. In particular, the limit exists (which could be $-\infty$). Thus, it remains to prove the first equality.

We describe explicitly the isomorphism \eqref{eq:2-hdby}. Delete a $4$-ball $B^4$ slightly smaller than the $0$-handle of $X$ and put $K(r):=K(\alpha^++r,\alpha^-+r)\cup L$ on its boundary. Cap off the sublink $K(\alpha^++r,\alpha^-+r)$ by some parallel copies of cores of the $2$-handles of various orientations corresponding to $(\alpha^++r,\alpha^-+r)$, and put a cylindrical surface $L\times I$ connecting the inner link $L\subset\partial B^4$ and the outer link $L\subset\partial X$. This gives a skein in $(X,L)$ with input link $K(r)$. Putting decorations on $K(r)$ defines $KhR_{Lee}(K(r))\to\mathcal S_0^{Lee}(X;L)$. Taking the direct sum over $r$ defines a map that descends to the isomorphism \eqref{eq:2-hdby}. It follows that the Lee canonical filling on this skein as a double skein (given by its own orientation) is given up to a scalar by the Lee canonical generator of $K(r)$ (with its own orientation) decorated on the input ball.

As in the proof of Theorem~\ref{thm:rank_ineq}, \eqref{eq:2-hdby} for $\bullet=Lee$ can be rewritten as \eqref{eq:2-hdby_filtered_colim}, where $S(r)=S_{|\alpha_1|+2r_1}\times\cdots\times S_{|\alpha_m|+2r_m}$. Label the copies of $K_i$ in $K(r)$ by $1,2,\cdots,|\alpha_i|+2r_i$. Then the orientations of $K(r)$ compatible with $L$ can be identified with tuples $(X_1,\cdots,X_m)$ where each $X_i$ is a subset of $[|\alpha_i|+2r_i]:=\{1,2,\cdots,|\alpha_i|+2r_i\}$, indicating which strands are given the same orientation as $K_i$. Those that give rise to possible orientations of $K(r)$ compatible with the homology class degree $\alpha$ correspond to such tuples with $\#X_i=\alpha_i^++r_i$. Then the canonical Lee lasagna generator $x_{\alpha,0}\in\mathcal S_0^{Lee}(X;L;\alpha)$ under \eqref{eq:2-hdby_filtered_colim} is represented, in each $r\in\Z_{\ge0}^m$ term, uniquely by the element $$x^{inv}(r):=c_r\sum_{\#X_i=\alpha_i^++r_i}x_{(X_1,\cdots,X_m)}\in KhR_{Lee}(K(r))^{S(r)}$$ with a $\{-|\alpha|-2|r|\}$ shift, where $c_r\in\Q$ is some nonzero scalar. Define $s_{\mathfrak{gl}_2}^{inv}(K(r)):=q(x^{inv}(r))$. Since \eqref{eq:2-hdby_filtered_colim} holds as filtered vector spaces, we have 
\begin{equation}\label{eq:las_s_from_s^inv}
s(X;L;\alpha)=\lim_{r\to\infty^m}(s_{\mathfrak{gl}_2}^{inv}(K(r))-|\alpha|-2|r|).
\end{equation}
It remains to show the right-hand side of \eqref{eq:las_s_from_s^inv} is unchanged if we replace $s_{\mathfrak{gl}_2}^{inv}(K(r))$ by $s_{\mathfrak{gl}_2}(K(r))+1$, the latter of which is equal to $q(x_{(X_1,\cdots,X_m)})$ where $x_{(X_1,\cdots,X_m)}$ is any term in the summation that defines $x^{inv}(r)$.

Since $x^{inv}(r)$ is a sum of $x_{(X_1,\cdots,X_m)}$ and some $S(r)$-translates of it (which all have the same filtration degree), we have $q(x^{inv}(r))\le q(x_{(X_1,\cdots,X_m)})$, thus $$\lim_{r\to\infty^m}(s_{\mathfrak{gl}_2}^{inv}(K(r))-|\alpha|-2|r|)\le\lim_{r\to\infty^m}(s_{\mathfrak{gl}_2}
(K(r))+1-|\alpha|-2|r|).$$ Now suppose the right-hand side is not $-\infty$, so the value stabilizes to some limit $M$ for $r$ sufficiently large. We claim that for sufficiently large $r$, we have $q(x^{inv}(r))=q(x_{(X_1,\cdots,X_m)})=M+|\alpha|+2|r|$, from which the statement follows.

To this end, we follow the argument in \cite[Section~4]{ren2023lee}. By looking at the $S(r)$-action on the canonical generators, we see $$KhR_{Lee}(K(r))\cong\bigotimes_{i=1}^m\Q\,\mathcal{P}([|\alpha_i|+2r_i])$$ as an $S(r)$-representation, where $\mathcal P$ denotes the power set operation on a set. Each tensor summand can be further decomposed as
\begin{align*}
\Q\,\mathcal P([|\alpha_i|+2r_i])=&\,\bigoplus_{j=0}^{|\alpha_i|+2r_i}\Q\{X\subset[|\alpha_i|+2r_i]\colon\#X=j\}\\
=&\,\bigoplus_{j=0}^{|\alpha_i|+2r_i}\left(\bigoplus_{k=0}^{\min(j,|\alpha_i|+2r_i-j)}(|\alpha_i|+2r_i-k,k)\right)
\end{align*}
into irreducible $S_{|\alpha_i|+2r_i}$-representations. Here $(a,b)$ denote the irreducible $S_{a+b}$-representation over $\Q$ defined by the two-row Young diagram $(a,b)$. Since the $S(r)$-action on $KhR_{Lee}(K(r))$ respects the filtration structure, all nonzero elements in each irreducible piece $V$ in the decomposition of $S(r)$ have the same filtration degree, which is denoted $q(V)$.

Without loss of generality, suppose $\alpha_i\ge0$ for all $i$ (otherwise reverse the orientations of some $K_i$'s). Then $\alpha_i^+=\alpha_i$, $\alpha_i^-=0$. The canonical generator $x_{(X_1,\cdots,X_m)}\in KhR_{Lee}(K(r))$ lies in the $S(r)$-subrepresentation (which is optimally small)
\begin{equation}\label{eq:x_in_sub_rep}
\bigotimes_{i=1}^m\Q\{X\subset[\alpha_i+2r_i]\colon\#X=\alpha_i+r\}=\bigoplus_{k_1=0}^{r_1}\cdots\bigoplus_{k_m=0}^{r_m}\bigotimes_{i=1}^m(\alpha_i+2r_i-k_i,k_i),
\end{equation}
while the element $x^{inv}(r)$ lies in the subrepresentation $$\bigotimes_{i=1}^m(\alpha_i+2r_i,0).$$

Thus $q(x_{(X_1,\cdots,X_m)})=\max_k(q(\otimes_{i=1}^m(\alpha_i+2r_i-k_i,k_i)))$ (this uses the fact that these representations are pairwise nonisomorphic) and $q(x^{inv}(r))=q(\otimes_{i=1}^m(\alpha_i+2r_i,0))$.

Suppose to the contrary that $q(x^{inv}(r))<q(x_{(X_1,\cdots,X_m)})$. Then there is a summand $\otimes_{i=1}^m(\alpha_i+2r_i-k_i,k_i)$ in \eqref{eq:x_in_sub_rep} with some $k_t\ne0$ that has $q(\otimes_{i=1}^m(\alpha_i+2r_i-k_i,k_i))=q(x_{(X_1,\cdots,X_m)})=M+|\alpha|+2|r|$. The argument in \cite[Section~4.3]{ren2023lee} shows that $q(\otimes_{i=1}^m(\alpha_i+2r_i-k_i-\delta_{it},k_i-\delta_{it}))=q(\otimes_{i=1}^m(\alpha_i+2r_i-k_i,k_i))$ where $\otimes_{i=1}^m(\alpha_i+2r_i-k_i-\delta_{it},k_i-\delta_{it})$ is the corresponding $S(r-e_t)$-irreducible summand of $KhR_{Lee}(K(r-e_t))$ ($e_t$ is the $t$\textsuperscript{th} coordinate vector). However, this implies the corresponding canonical generator for $KhR_{Lee}(K(r-e_t))$ has
\begin{align*}
q(x_{(X_1,\cdots,X_t\backslash\{*\},\cdots,X_m)})\ge&\,q(\otimes_{i=1}^m(\alpha_i+2r_i-k_i-\delta_{it},k_i-\delta_{it}))\\
=&\,q(x_{(X_1,\cdots,X_m)})=M+|\alpha|+2|r|>M+|\alpha|+2|r-e_t|.
\end{align*}
This contradicts $r$ being sufficiently large (so that $\lim(s_{\mathfrak{gl}_2}(K(\cdot))+1-|\alpha|-2|\cdot|)=M$ already stabilizes at $r-e_t$). The proof is complete.
\end{proof}

\subsection{Diagrammatic nonvanishing result}\label{sbsec:diagram}
For simplicity, in this section we assume $L=\emptyset$. We give a diagrammatic criterion to guarantee the nonvanishing of $\mathcal S_0^2(X)$ and $\mathcal S_0^{Lee}(X)$.

As before, assume $X$ is a $2$-handlebody obtained by attaching $2$-handles to $B^4$ along a framed oriented link $K\subset S^3$ with components $K_1,\cdots,K_m$ and framings $p_1,\cdots,p_m$. Let $p=(p_1,\cdots,p_m)$ and $P$ be the diagonal matrix with entries $p_1,\cdots,p_m$. Choose a link diagram $D$ of $K$ as an unframed link. Let $N$ denote the crossing matrix of $D$, where $N_{ii}$ is the number of crossings of $K_i$ in diagram $D$, and $N_{ij}$ is a half the number of crossings between $K_i$ and $K_j$ in $D$ for $i\ne j$. Let $w=(w_1,\cdots,w_m)$ be the writhe vector of the diagram $D$, where $w_i$ is the writhe of the $i$-th component of $D$, and let $W$ be the diagonal matrix with entries $w_1,\cdots,w_m$. Finally, for $f,n\in\Z^m$, let $K^f$ denote the knot $K$ with framing $f$, and $K^f(n)$ denote the $n$-cable of $K$ with framing $f$, obtained by replacing $K_i$ by an $f_i$-framed $|n_i|$-cable, with orientation reversed if $n_i<0$. Thus, any $K^w(n)$ has a natural diagram induced from $D$ by taking blackboard-framed cables. (Here and below, unless explicitly stated otherwise, the blackboard framing of a diagram of a framed link is equal to the framing of the link.)

\begin{Thm}\label{thm:nonvan_diagram}
In the notations above, suppose a class $\alpha\in H_2(X)\cong\Z^m$ satisfies\vspace{-5pt}
\begin{enumerate}[(i)]
\item $\alpha$ maximizes $h(\alpha'):=\alpha'^T(P-W+N)\alpha'$ among all vectors $\alpha'\in\alpha+2\Z^m$ (in particular, $P-W+N$ is nonpositive);
\item $K^w(\alpha')$ has no negative crossings in the diagram induced from $D$ for any $\alpha'\in\alpha+2\Z^m$ with $h(\alpha')=h(\alpha)$.\vspace{-5pt}
\end{enumerate}
Then\vspace{-5pt}
\begin{enumerate}[(1)]
\item For $\bullet\in\{2,Lee\}$, $\mathcal S_{0,>0}^\bullet(X;\alpha)=0$;
\item The dimension of $\mathcal S_{0,0}^{Lee}(X;\alpha)$ is equal to the number of maximal points of $h$ on $\alpha+2\Z^m$. The rank of $\mathcal S_{0,0}^2(X;\alpha)$ is bounded \emph{above} by this number;
\item $s(X;\alpha)=-s(-K^p(\alpha))-w(K^p(\alpha))-|\alpha|+1$. In particular, $X$ has nonvanishing Khovanov and Lee skein lasagna modules.
\end{enumerate}
\end{Thm}\smallskip

\begin{Rmk}
It is possible to allow a boundary link $L\subset\partial X$, but we refrained from doing so to avoid making the statement too complicated.
\end{Rmk}

\begin{Cor}\label{cor:neg_def_nonvan}
In the notations of Theorem~\ref{thm:nonvan_diagram}, if $P-W+N$ is negative definite, then $s(X;0)=0$ and $\mathcal S_{0,0,q}^2(X;0)\otimes\Q=gr_q(\mathcal S_{0,0}^{Lee}(X;0))=\begin{cases}\Q,&q=0\\0,&q\ne0.\end{cases}$
\end{Cor}
\begin{proof}
Conditions (i)(ii) in Theorem~\ref{thm:nonvan_diagram} are satisfied for $\alpha=0$, which is the unique maximal point of $h$. Thus (3) implies $s(X;0)=-s(\emptyset)+1=0$. By definition of $s$, this implies $\dim gr_0(\mathcal S_{0,0}^{Lee}(X;0))\ge1$. Now (2) and Theorem~\ref{thm:rank_ineq} imply the claims on $\mathcal S_{0,0}^\bullet(X;0)$.
\end{proof}

\begin{proof}[Proof of Theorem~\ref{thm:nonvan_diagram_knot}]
We apply Theorem~\ref{thm:nonvan_diagram} and Corollary \ref{cor:neg_def_nonvan} to the situation in Theorem~\ref{thm:nonvan_diagram_knot}, that is to say, to computing the lasagna invariants of knot traces $X=X_n(K)$ in homology class degrees $\alpha=0,1$.  In this case, in the notation of Theorem~\ref{thm:nonvan_diagram} we have $m=1$, $P-W+N=n+2n_-(D)$. Thus the computation for $\alpha=0$ follows from Corollary~\ref{cor:neg_def_nonvan}.

For the computation for $\alpha=1$ when $K$ is positive, take $D$ to be a positive diagram of $K$. Then for $n\le0$, (i)(ii) in Theorem~\ref{thm:nonvan_diagram} are satisfied for both $\alpha=0,1$. Thus $s(X;0)=-s(\emptyset)+1=0$ and $s(X;1)=-s(-K)-n-1+1=s(K)-n$ by Theorem~\ref{thm:nonvan_diagram}(3). For $n<0$, $\alpha=\pm1$ are the unique maximal points of $h$ on $1+2\Z$, thus Theorem~\ref{thm:nonvan_diagram}(2) implies that the rank of $\mathcal S_{0,0}^2(X;1)$ is at most $2$. Now the claims on $\mathcal S_{0,0}^\bullet(X;1)$ follow from Theorem~\ref{thm:Lee_structure}(4) and Theorem~\ref{thm:rank_ineq}.
\end{proof}

The rest of the section is devoted to proving Theorem~\ref{thm:nonvan_diagram}.

For a link $L$, let $P(L)(t,q)$ be the Poincar\'e polynomial of $KhR_2(L)$. Define $$\widetilde{KhR}_2(L):=KhR_2(L)[\tfrac12w(L)]\{-\tfrac12w(L)\}$$ to be the renormalized homology of $L$, which is independent of the orientation of $L$ (cf. \eqref{eq:KhR_ori_change}). Let $\tilde P(L)(t,q)$ be the Poincar\'e polynomial of $\widetilde{KhR}_2(L)$ (which is a Laurent polynomial in $t^{1/2},q^{1/2}$).

For two (Laurent) polynomials $P,Q$ in variables $t^{1/2},q^{1/2}$, we write $P\le Q$ if every coefficient of $P$ is less than or equal to that of $Q$. For two vectors $a,b\in\Z^m$, write $a\le b$ if $a_i\le b_i$ for all $i$.
 
\begin{Lem}[Comparison lemma]\label{lem:full_comparison}
For $n\in\Z_{\ge0}^m$ and $p\in\Z^m$ with $p\le w$, we have 
\begin{equation}\label{eq:compare}
\tilde P(K^p(n))\le\sum_{\substack{0\le n'\le n\\2|n-n'}}R_{n,n',w-p}\tilde P(K^w(n'))
\end{equation}
for some polynomials $R_{n,n',w-p}\in\Z[t^{\pm1/2},q^{\pm1/2}]$ independent of $K$ with $$R_{n,n',w-p}(t,q)=t^{-\sum_i\frac{(w_i-p_i)n'^2_i}{2}}q^{\sum_i\frac{(w_i-p_i)(n'^2_i+2n'_i)}{2}}\prod_{j=1}^m\left(\sum_{r_j=0}^{\frac{n_j-n'_j}{2}}\dim\left(\frac{n_j+n'_j}{2}+r_j,\frac{n_j-n'_j}{2}-r_j\right)q^{2r_j}\right)+\cdots$$
where $\cdots$ are terms with lower $t$-degrees, and $(a,b)$ denotes the irreducible $S_{a+b}$-representation given by the Young diagram $(a,b)$.

Moreover, if \eqref{eq:compare} is sharp at the highest $t$-degree of the left-hand side, then for any $i=1,\cdots,m$ with $n_i\ge2$, the morphism on $KhR_2$ induced by the dotted annular creation map $K^p(n-2e_i)\to K^p(n)$\footnote{Really, the dotted annular creation map $K^p(n-2e_i)\to K^p(n-e_i,e_i)$ on $KhR_2$ composed with the isomorphism that changes the orientation of the reversed $K_i$ strand.} is injective in the highest nontrivial homological degree of $KhR_2(K^p(n))$.
\end{Lem}

For the reader's convenience, we note that for $r,n\in\Z_{\ge0}$, $r\le n/2$, $$\dim(n-r,r)=\binom{n}{r}-\binom{n}{r-1}.$$

Lemma~\ref{lem:full_comparison} is a consequence of an inductive argument pioneered by Sto\v si\'c \cite{stovsic2009khovanov}. Although it might have been implicit in the proof of Theorem~1 there, or in more detail in \cite{tagami2013maximal}, for completeness we sketch a proof in Appendix~\ref{sec:append}. For our purpose, we do not really need the quantum shifting information in the lemma.

\begin{Lem}\label{lem:max_t}
For $n\in\Z_{\ge0}^m$, the maximal $t$-degree of $\tilde P(K^w(n))$ is at most $\tfrac12n^TNn$.
\end{Lem}
\begin{proof}
For any framed link diagram $D$ of $L$, the maximal $t$-degree of $P(L)$ is at most the number of negative crossings of $D$, and thus the maximal $t$-degree of $\tilde P(L)$ is at most one half the total number of crossings of $D$. For $D$ the standard diagram of $K^w(n)$, this number is $\tfrac12n^TNn$.
\end{proof}
\begin{proof}[Proof of Theorem~\ref{thm:nonvan_diagram}]
Note that adding positive kinks to $D$ does not affect the conditions and statements, but increases $w$. Thus we can assume that $w\ge p$ and $D$ has no $0$-framed unknot component. We also assume $\alpha\in\Z_{\ge0}^m$, as the general case is only notationally more complicated (and can indeed be deduced from the special case by reversing the orientations of some $K_i$'s). Thus $\alpha^+=\alpha$, $\alpha^-=0$. Write $N=N_++N_-$ to separate contributions coming from positive/negative crossings.

Recall that the $2$-handlebody formula for $X$ and $\alpha\in\Z_{\ge0}^m$ \eqref{eq:2-hdby} can be rewritten as
\begin{equation}\label{eq:2-hdby_X}
\mathcal S_0^\bullet(X;\alpha)=\mathrm{colim}_{r\in\Z_{\ge0}^m}KhR_\bullet(K^p(\alpha+r,r))^{S_{\alpha+2r}}\{-|\alpha|-2|r|\},
\end{equation}
where $S_n=S_{n_1}\times\cdots\times S_{n_m}$ for $n\in\Z_{\ge0}^m$. Here if $\bullet=2$ we assume we work over $\Q$.

\textbf{Step 1:} Isolating the highest $t$-degrees in $\tilde{P}$.

In the first step, we prove that:\vspace{-5pt}
\begin{enumerate}[(a)]
\item For any $n\in\Z_{\ge0}^m$, $n\equiv\alpha\pmod2$, 
\begin{equation}\label{eq:max_t_equal}
\tilde P(K^p(n)))=\sum_{\substack{0\le n'\le n\\2|n-n'\\h(n')=h(\alpha)}}R_{n,n',w-p}\tilde P(K^w(n'))+\cdots
\end{equation}
where the highest nontrivial $t$-degree in each summand is $\tfrac12h(\alpha)$, and $\cdots$ denotes terms with lower $t$-degrees.
\item If $\tilde P_{Lee}(K^p(n))(t,q)$ denotes the corresponding renormalized Poincar\'e polynomial for $gr(KhR_{Lee}(K^p(n)))$, then the highest $t$-degree parts of $\tilde P(K^p(n))$ and $\tilde P_{Lee}(K^p(n))$ are equal.
\end{enumerate}

In the comparison lemma \ref{lem:full_comparison}, the highest $t$-degree of the $n'$ term on the right-hand side is at most $$\tfrac12n'^T(P-W)n'+\tfrac12n'^TNn'=\tfrac12n'^T(P-W+N)n'$$ by Lemma~\ref{lem:max_t}, which is strictly less than $h(\alpha)$ if $h(n')\ne h(\alpha)$ by assumption (i). On the other hand, if $h(n')=h(\alpha)$, then by assumption (ii) $K^w(n')$ has a positive diagram induced by $D$ with some number $s(n')$ of connected components  (connected components being understood in the sense of 4-valent graphs; thus, $s(n')$ only depends on the set of indices of the nonzero coordinates of $n'$, as $D$ has no $0$-framed unknot component).  Then $KhR_2^0(K^w(n'))$ (as well as $KhR_{Lee}^0(K^w(n'))$) has rank $2^{s(n')}$ \cite[Proposition~6.1]{khovanov2003patterns}. After renormalizing, we see that $R_{n,n',w-p}\tilde P(K^w(n'))$ has maximal $t$-degree $\tfrac12h(\alpha)$, and in this $t$-degree the polynomial has rank (i.e. the evaluation at $q=1$) $$2^{s(n')}\prod_{j=1}^m\left(\sum_{r_j=0}^{\frac{n_j-n'_j}{2}}\dim\left(\frac{n_j+n'_j}{2}+r_j,\frac{n_j-n'_j}{2}-r_j\right)\right)=2^{s(n')}\binom{n}{\frac{n-n'}{2}}$$ where for $a,b\in\Z_{\ge0}^m$ we write $\binom{a}{b}=\prod_i\binom{a_i}{b_i}$. Now it follows from Lemma~\ref{lem:full_comparison} that the highest $t$-degree of $\tilde P(K^p(n))$ is at most $\tfrac12h(\alpha)$, and in this $t$-degree it has rank at most $$\sum_{\substack{0\le n'\le n\\2|n-n'\\h(n')=h(\alpha)}}2^{s(n')}\binom{n}{\frac{n-n'}{2}}.$$

On the other hand, the rank of $\tilde P_{Lee}(K^p(n))$ at $t=\tfrac12h(\alpha)$ is equal to the number of orientations of $K^p(n)$ with writhe $h(\alpha)$. Dividing the count into algebraic cable types of the resulting oriented link (here, two cables $K^p(a,b)$ and $K^p(c,d)$ are said to have the same algebraic cable type if $a-b=c-d$), the rank is again
\begin{align*}
\sum_{\substack{-n\le n'\le n\\2|n-n'\\n'^T(P-W+N_+-N_-)n'\\=\alpha^T(P-W+N)\alpha}}\binom{n}{\frac{n-n'}{2}}=\sum_{\substack{-n\le n'\le n\\2|n-n'\\h(n')=h(\alpha)}}\binom{n}{\frac{n-n'}{2}}=\sum_{\substack{0\le n'\le n\\2|n-n'\\h(n')=h(\alpha)}}2^{s(n')}\binom{n}{\frac{n-n'}{2}},
\end{align*}
where the first equality follows from conditions (i) and (ii).  The second equality is true because in view of (ii), for $n'\ge0$, changing signs of coordinates of $n'$ maintains $h(n')$ if and only if, for each connected component of the diagram of $K^w(n')$, we either change all of the corresponding coordinates of $n'$ or none of them.  Since $\tilde P(K^p(n))$ is bounded below by $\tilde P_{Lee}(K^p(n))$ as there is a spectral sequence from $KhR_2$ to $KhR_{Lee}$, we must have the equality \eqref{eq:max_t_equal} as well as its version for $\tilde P_{Lee}(K^p(n))$ as claimed.

\textbf{Step 2:} Unrenormalization and proofs of (1) and (2).

For $n=\alpha+2r$, condition (ii) ensures that the writhe of $K^w(n')$ is the same as $n'^TNn'$, allowing us to unrenormalize \eqref{eq:max_t_equal} to give $$P(K^p(\alpha+r,r))=\sum_{\substack{0\le n'\le n\\2|n-n'\\h(n')=h(\alpha)}}(tq^{-1})^{\tfrac12n'^T(W-P)n'}R_{n,n',w-p}P(K^w(n'))+\cdots,$$ where the maximal $t$-degree of each term on the right-hand side is $0$. The same holds for the Lee version. In view of \eqref{eq:2-hdby_X}, this proves (1), namely that $\mathcal S_{0,>0}^{\bullet}(X;\alpha)=0$.

For (2), by Theorem~\ref{thm:Lee_structure}(1)(2), $\mathcal S_{0,0}^{Lee}(X;\alpha)$ has a basis consisting of $x_{\alpha_+,\alpha_-}$ with $\alpha_\pm\in H_2(X)$, $\alpha_++\alpha_-=\alpha$, $\alpha_+\cdot\alpha_-=0$. Such generators are in one-to-one correspondence with $\alpha'\in\alpha+2\Z^m$ with $\alpha'^2=\alpha^2$, where the correspondence is given by $\alpha':=\alpha_+-\alpha_-$. The bilinear form on $H_2(X)$ can be rewritten as $\beta^2=\beta^T(P-W+N-2N_-)\beta$. Thus by conditions (i) and (ii), we see that such $\alpha'$ are exactly those with $h(\alpha')=h(\alpha)$, proving the statement for $\mathcal S_{0,0}^{Lee}$.

Write \eqref{eq:2-hdby_X} in homological degree $0$ for $\bullet=Lee$ as $\mathcal S_{0,0}^{Lee}(X;\alpha)=\mathrm{colim}_rV_r$ for short. Then by claim (b) in Step 1 and the naturality of the Lee spectral sequence, we know $\mathcal S_{0,0}^2(X;\alpha;\Q)=\mathrm{colim}_rgr(V_r)$, which equals the subspace of $gr(\mathrm{colim}_rV_r)$ of finite degree elements together with $0$. In particular, the rank of $\mathcal S_{0,0}^2(X;\alpha)$ is bounded above by the dimension of $\mathcal S_{0,0}^{Lee}(X;\alpha)$.

\textbf{Step 3:} Decomposition of $KhR_{Lee}$ and proof of (3).

As we have already seen once in Step 1 (with $n=\alpha+2r$, and renormalization), $KhR_{Lee}^0(K^p(\alpha+r,r))$ is generated by canonical generators coming from orientations of $K^p(\alpha+r,r)$ with writhe $h(\alpha)$, and we can divide these orientations according to their algebraic cable type.

More precisely, for $-n\le n'\le n$ with $2|n-n'$, $h(n')=h(\alpha)$, let $$KhR_{Lee,n'}^0(K^p(\alpha+r,r))\subset KhR_{Lee}^0(K^p(\alpha+r,r))$$ be the subspace spanned by generators whose orientations make $K^p(n)$ algebraic $n'$-cables of $K$. This subspace is an $S_n$-subrepresentation.

We will be interested in $n'=\alpha$. As in \cite[Proposition~4.2]{ren2023lee}, we know \begin{equation}\label{eq:KhR_Lee_decompose}
KhR_{Lee,\alpha}^0(K^p(\alpha+r,r))\cong\bigoplus_{0\le s\le r}(n-s,s)
\end{equation}
as $S_n$-representations, where for $a,b\in\Z_{\ge0}^m$, $a\ge b$, $(a,b)$ denotes the $S_{a+b}$-representation given by the tensor product of the $S_{a_i+b_i}$-representations $(a_i,b_i)$. Similarly, 
\begin{equation}\label{eq:KhR_Lee_decompose_+2}
KhR_{Lee,\alpha}^0(K^p(\alpha+r+e_i,r+e_i))\cong\bigoplus_{0\le s\le r+e_i}(n+2e_i-s,s)
\end{equation}
as $S_{n+2e_i}$-representations.

By the argument in \cite[Section~4.3]{ren2023lee}, for every $0\le s\le r$, the subrepresentation $(n-s,s)$ in \eqref{eq:KhR_Lee_decompose} and the subrepresentation $(n-s+e_i,s+e_i)$ in \eqref{eq:KhR_Lee_decompose_+2} have the same filtration degree. On the other hand, the dotted annular creation cobordism map restricts to a filtered degree $2$ map $$\phi_i\colon KhR_{Lee,\alpha}^0(K^p(\alpha+r,r))\to KhR_{Lee,\alpha}^0(K^p(\alpha+r+e_i,r+e_i)).$$
By Step 1, the associated graded map of $\phi_i$ is exactly the corresponding dotted annular creation map on $KhR_2$, which is injective by the conclusions of Step 1 and of Lemma~\ref{lem:full_comparison}. Thus $\phi_i$ increases the filtration degree of every nonzero element by exactly $2$. Inductively from $r=0$ (where $KhR_{Lee,\alpha}^0(K^p(\alpha))=(\alpha,0)$ in filtration degree $s_{\mathfrak{gl}_2}(K^p(\alpha))+1$ by definition), we see that the filtration degree of $(n,0)$ (the $S_n$-invariant subspace) in \eqref{eq:KhR_Lee_decompose} is $s_{\mathfrak{gl}_2}(K^p(\alpha))+1+2|r|$. Consequently by \eqref{eq:2-hdby} (cf. \eqref{eq:las_s_from_s^inv}) we have $$s(X;\alpha)=\lim_{r\to\infty^m}(s_{\mathfrak{gl}_2}(K^p(\alpha))+1+2|r|-|\alpha|-2|r|)=-s(-K^p(\alpha))-w(K^p(\alpha))-|\alpha|+1>-\infty.$$
In particular, $\mathcal S_{0,0,s(X;\alpha)}^2(X;\alpha)\ne0$ by Theorem~\ref{thm:rank_ineq}.
\end{proof}

\begin{Rmk}\label{rmk:purported_S_0^2}
\begin{enumerate}
\item We expect that in the notation of Theorem~\ref{thm:nonvan_diagram}, assuming no component of $D$ is a $0$-framed unknot, then
\begin{equation}\label{eq:purported_S_0^2}
\quad\quad\quad\mathcal S_{0,0}^2(X;\alpha)=\bigoplus_{\substack{\alpha'\in\Z_{\ge0}^m\\2|\alpha-\alpha'\\h(\alpha')=h(\alpha)}}KhR_2^0(K^w(\alpha'))\left\{\sum_i\frac{(w_i-p_i)(\alpha_i'^2+2\alpha_i')}{2}-\frac{w(K^w(\alpha'))}{2}-|\alpha'|\right\}.
\end{equation}
In fact, an expanded argument of the proof above shows that the graded rank of the left-hand side is bounded above by that of the right-hand side. However, we have not been able to work out a proof of \eqref{eq:purported_S_0^2}. We remark that over $\Q$, \eqref{eq:purported_S_0^2} is equivalent to adding to the last claim of Lemma~\ref{lem:full_comparison} the conclusion that the symmetrized dotted annular creation map gives an injection $$KhR_2^{h_{max}}(K^p(n-2e_i);\Q)^{S_{n-2e_i}}\hookrightarrow KhR_2^{h_{max}}(K^p(n);\Q)^{S_n}.$$ where $h_{max}$ is the maximal nontrivial homological degree of the right-hand side.
\item In particular, in the setup of Theorem~\ref{thm:nonvan_diagram_knot}(1) or (2) when $n<0$, one can easily show the validity of \eqref{eq:purported_S_0^2}. Therefore, the conclusions in Theorem~\ref{thm:nonvan_diagram_knot} for $\mathcal S_{0,0}^2$ can be stated over $\Z$ instead.
\end{enumerate}
\end{Rmk}

\section{Nonvanishing examples and applications}\label{sec:examples}
In this section we present some explicit calculations and applications for Khovanov and Lee skein lasagna modules and lasagna $s$-invariants of some $2$-handlebodies (with empty boundary link). Roughly the first half of this section utilizes the tools developed in Section~\ref{sec:2-hdby}, while the second half utilizes results in a paper of the first author \cite{ren2023lee}.

\subsection{Knot traces of some small knots}\label{sbsec:example_traces}\phantom{}
\begin{table}[H]
\centering
\begin{tabular}{c|c|c|c|c|c|c|c|c|c|c|c}
\backslashbox{Knot}{Framing}&$-4$&$-3$&$-2$&$-1$&$0$&$1$&$2$&$3$&$4$&$5$&$6$\\\hline
$U$     &N&N&N&N&N&V&V&V&V&V&V\\\hline
$3_1$   &N&?&?&V&V&V&V&V&V&V&V\\\hline
$-3_1$  &N&N&N&N&N&?&?&?&?&?V&V\\\hline
$4_1$   &N&N&N&N&N&?V&?V&V&V&V&V
\end{tabular}
\caption{Vanishing/Nonvanishing of Khovanov skein lasagna modules of some knot traces. An entry N/V means the corresponding knot trace has nonvanishing/vanishing Khovanov skein lasagna module. An entry ?V means it has vanishing Lee skein lasagna module (meaning the filtration function is identically $-\infty$).}
\label{tab:knot_trace_n/v}
\end{table}

\begin{Ex}[Unknot]\label{ex:unknot_trace_nonvan}
By Theorem~\ref{thm:nonvan_diagram_knot}, for $n\le0$, the $D^2$-bundle $D(n)$ over $S^2$ with Euler number $n$ has $s(D(n);0)=0$, $s(D(n);1)=-n$. In particular, $\mathcal S_{0,0}^2(D(n);\alpha)\ne0$ for $\alpha=0,1$.

For $n<0$, more explicitly we have $\mathcal S_{0,0}^2(D(n);0)=\Z$ concentrating in degree $0$ and $\mathcal S_{0,0}^2(D(n);1)=\Z^2$ concentrating in degrees $n,n-2$ (cf. Remark~\ref{rmk:purported_S_0^2}(2)).

For $n=0$, a direct computation using the $2$-handlebody formula yields (see Manolescu--Neithalath \cite[Theorem~1.2]{manolescu2022skein}) $\mathcal S_0^2(D(0))\cong\Z[A_0^{\pm},A_1]$ as tri-graded abelian groups, where $A_0$ has tri-degree $(0,0,1)$ and $A_1$ has tri-degree $(0,-2,1)$. A similar computation gives $s(D(0);\alpha)=0$ for all $\alpha$. Alternatively, this follows from a direct application of Proposition~\ref{prop:las_s_from_classical} or Theorem~\ref{thm:nonvan_diagram}.

On the other hand, for $n>0$, Example~\ref{ex:unknot_van} shows $\mathcal S_0^2(D(n))=0$ and consequently all lasagna $s$-invariants of $D(n)$ are $-\infty$.
\end{Ex}

For other examples, we only mention the vanishingness/nonvanishingness of the Khovanov skein lasagna modules of the knot traces. For small knots, Corollary~\ref{cor:emd_-kCP2} (which will be proved in Section~\ref{sbsec:-CP^2}) usually provides more effective nonvanishing bounds than Theorem~\ref{thm:nonvan_diagram_knot}. We emphasize however that Corollary~\ref{cor:emd_-kCP2} does not provide an explicit $s$-invariant in most cases.

\begin{Ex}[Trefoils]
The left-handed trefoil, or $3_1$, has $3$ negative crossings. Therefore~\ref{thm:nonvan_diagram_knot}(1) shows $X_n(3_1)$ has nonvanishing Khovanov skein lasagna module for $n<-6$.

We can do better as follows. The knot $3_1$ can be obtained from the unknot by a crossing change that is realized by a $-1$ twist on two parallelly oriented strands. By Kirby calculus, the twist is realized by a $(+1)$-framed $2$-handle attached to $B^4$, with the effect that the $0$-framed unknot on $\partial B^4$ becomes the $(-4)$-framed unknot on $\partial(B^4\cup2\text{-handle})$. This implies that $3_1$ is $(-4)$-slice in $\CP^2$, thus is $n$-slice in some $k\CP^2$ for every $n\le-4$. It follows that $X_n(3_1)$ embeds in some $k\overline{\CP^2}$ for $n\le-4$. By Corollary~\ref{cor:emd_-kCP2}, $X_n(3_1)$ has nonvanishing Khovanov skein lasagna module for $n\le-4$.

On the other hand, $TB(-3_1)=1$, so $X_n(3_1)$ has vanishing Khovanov skein lasagna module for $n\ge-1$ by Theorem~\ref{thm:van}. It would be interesting to know whether $X_n(3_1)$ has nonvanishing Khovanov skein lasagna module for $n=-3,-2$.

The right-handed trefoil $-3_1$ is positive. By either Theorem~\ref{thm:nonvan_diagram_knot} or Corollary~\ref{cor:emd_-kCP2} we know $X_n(-3_1)$ has nonvanishing Khovanov skein lasagna module for $n\le0$. On the other hand, since $TB(3_1)=-6$, $X_n(-3_1)$ has vanishing Khovanov skein lasagna module for $n\ge6$. In this case we have a larger gap of unknown. We remark that Proposition~\ref{prop:las_s_compare_-CP2} will imply that $X_5(-3_1)$ has vanishing Lee skein lasagna module, as there is a framed concordance in $\overline{\CP^2}$ from the $1$-framed unknot to the $5$-framed $-3_1$.
\end{Ex}

\begin{Ex}[Figure $8$]
The standard diagram of the figure $8$ knot, or $4_1$, has $2$ negative crossings. Theorem~\ref{thm:nonvan_diagram_knot}(1) shows $X_n(4_1)$ has nonvanishing Khovanov skein lasagna module for $n<-4$, while Corollary~\ref{cor:emd_-kCP2} shows so for $n\le0$ as $4_1$ is $0$-slice in $\CP^2$. On the other hand, since $-4_1=4_1$ and $TB(4_1)=-3$, we know $X_n(4_1)$ has vanishing Khovanov skein lasagna module for $n\ge3$ by Theorem~\ref{thm:van}.

It would be interesting to know whether $X_n(4_1)$ has nonvanishing Khovanov skein lasagna module for $n=1,2$. Again, Proposition~\ref{prop:las_s_compare_-CP2} will imply they both have vanishing Lee skein lasagna modules.
\end{Ex}

\subsection{Nonpositive definite plumbings and generalizations}\label{sbsec:example_plumbing}
\begin{Ex}[Nonpositive plumbing tree]\label{ex:plumbing_tree}
Let $X$ be a nonpositive definite plumbing tree of spheres, i.e. a plumbing of $D^2$-bundles over $2$-spheres along a tree, whose intersection form is nonpositive definite. Then we can orient the unknots in the underlying Kirby diagram so that all crossings are positive. In particular, in the notation of Theorem~\ref{thm:nonvan_diagram}, $P-W+N=P-W+N_+-N_-$ is the intersection form of $X$, which is nonpositive. Therefore Theorem~\ref{thm:nonvan_diagram}(3) implies $s(X;0)=0$ and $X$ has nonvanishing Khovanov skein lasagna module.

The negative smooth $E8$ manifold, obtained by plumbing eight $(-2)$-disk bundles over $S^2$ according to the negative-definite $E8$ matrix, belongs to this family.
\end{Ex}\smallskip

\begin{Ex}[Traces on alternating chainmail links]\label{ex:chainmail}
Let $G$ be a chainmail graph without loops and with positive edge weights and nonpositive vertex weights, and $K$ be its associated chainmail link, in the sense of \cite[Section~2]{agol2023chainmail}. Then $K$ has a standard positive diagram (as an unframed link), and a nonpositive framing matrix which is equal to $P-W+N$ for its standard diagram. This means the trace $X$ on $K$ has nonvanishing Khovanov skein lasagna module. This generalizes the previous example.
\end{Ex}\smallskip

\begin{Ex}[Branched double cover of definite surfaces of alternating links]
Suppose $L\subset S^3$ is a nonsplit alternating link, and $S\subset B^4$ is an unoriented surface bounding $L$ with negative definite Gordon–Litherland pairing \cite{gordon1978signature}. Greene \cite{greene2017alternating} showed that every such $S$ is isotopic rel boundary to a checkerboard surface of some alternating diagram of $L$. Conversely, one of the checkerboard surfaces of any alternating diagram is negative definite, a fact that is much easier.

The branched double cover of $S$ is a $4$-manifold $\Sigma(S)$ bounding the branched double cover $\Sigma(L)$ of $L$. It is also considered by Ozsv\'ath--Szab\'o \cite{ozsvath2005heegaard} (denoted $X_L$ in Section~3 there), and from the description there (see also \cite[Section~3.1]{greene2013spanning}) we realize that $\Sigma(S)$ is a special case of Example~\ref{ex:chainmail}, hence it has nonvanishing Khovanov skein lasagna module. In particular, $\Sigma(L)$ bounds a $2$-handlebody with nonvanishing Khovanov skein lasagna module. Since $\overline{\Sigma(L)}=\Sigma(-L)$ we see $\Sigma(L)$ bounds such $2$-handlebodies on both sides.
\end{Ex}\smallskip

\begin{Ex}
One can further generalize Example~\ref{ex:chainmail} by connect-summing some components of $K$ with $0$-framed positive knots. The same argument shows the resulting trace has nonvanishing Khovanov skein lasagna module.
\end{Ex}

\subsection{More exotica from connected sums}\label{sbsec:more_exotica}
In this section we prove Corollary~\ref{cor:more_exotic} on exotic pairs arising from connected sums of various copies of our two exotic manifolds $X_1$ and $X_2$ from Theorem \ref{thm:exotic} as well as copies of $S^1\times S^3$, $S^2\times D^2$, $\overline{\CP^2}$, and the negative $E8$ manifold. One can show the relevant pairs of manifolds have nonisomorphic Khovanov skein lasagna modules, but we give a simpler argument by showing that they have different lasagna $s$-invariants. Note that by construction, lasagna $s$-invariants of $X$ are invariants of $int(X)$; see Lemma~\ref{lem:interior_invariant}.

\begin{proof}[Proof of Corollary~\ref{cor:more_exotic}]
Write $W_1:=X_1^{\#a}\#X_2^{\#b}$ and $W_2:=X_1^{\#a'}\#X_2^{\#b'}$, where $a+b=a'+b'=n$ and $(a,b)\ne(a',b')$.

Recall from Section~\ref{sbsec:exotic_proof} that $s(X_1;1)=3$, $s(X_2;1)=1$, where $1$ is a generator of $H_2(X_i)$. Since we did not specify the sign of $1$, we equally have $s(X_1;-1)=3$, $s(X_2;-1)=1$. By the connected sum formula (Theorem~\ref{thm:s_prop}(4)), $$s(W_1;(\epsilon_1,\cdots,\epsilon_n))=3a+b$$ for any $\epsilon_i=\pm1$. If there were a diffeomorphism $W_1\cong W_2$, then by considering the intersection form we see it sends each class $(\epsilon_1,\cdots,\epsilon_n)$ to some class $(\epsilon'_1,\cdots,\epsilon'_n)$, $\epsilon_i'=\pm1$. However, $s(W_2;(\epsilon'_1,\cdots,\epsilon'_n))=3a'+b'\ne3a+b$, a contradiction.

By (a Lee version of) \cite[Corollary~4.2] {manolescu2023skein} we know $s(S^1\times S^3;0)=0$. Since $H_2(S^1\times S^3)=0$ the above argument remains valid under connected sums with $S^1\times S^3$. Since $s(\overline{\CP^2};1)=1$ by Example~\ref{ex:unknot_trace_nonvan} for $n=-1$ and Proposition~\ref{prop:34_hd}, the argument remains valid under further connected sums with $\overline{\CP^2}$.

By the $n=0$ case of Example~\ref{ex:unknot_trace_nonvan}, we know $s(S^2\times D^2;\alpha)=0$ for all $\alpha$. Any diffeomorphism $k(S^1\times S^3)\#m\overline{\CP^2}\#W_1\natural(S^2\times D^2)^{\natural c}\cong k(S^1\times S^3)\#m\overline{\CP^2}\#W_2\natural(S^2\times D^2)^{\natural c}$ sends each class $(\epsilon_1,\cdots,\epsilon_{m+n},0^c)$ to some class $(\epsilon_1',\cdots,\epsilon_{m+n}',\alpha_1,\cdots,\alpha_c)$. Again, the $s$-invariant of the latter class is not equal to the $s$-invariant of the former class, a contradiction.

By Example~\ref{ex:plumbing_tree} we know the negative $E8$ manifold has $s(E8;0)=0$. Any diffeomorphism between the above pair with $d$ further copies of $E8$ added sends each $(\epsilon_1,\cdots,\epsilon_{m+n},0^c,0^d)$ to $(\epsilon_1',\cdots,\epsilon_{m+n}',\alpha_1,\cdots,\alpha_c,0^d)$, whose $s$-invariants still differ, leading to a contradiction.

Finally, by Lemma~\ref{lem:interior_invariant}, the interiors of any of the above pairs are not diffeomorphic.
\end{proof}

\subsection{Nonpositive shake genus of positive knots}\label{sbsec:shake_genus}
We prove Theorem~\ref{thm:shake_genus}(1) computing nonpositive shake genus of knots concordant to some positive knot.
\begin{proof}[Proof of Theorem~\ref{thm:shake_genus}(1)]
By Theorem~\ref{thm:nonvan_diagram_knot}(2), if $K'$ is a positive knot and $n\le0$, then $s(X_n(K');1)=s(K')-n$. By Proposition~\ref{prop:intro_Lee_comparison}, if $K$ is a knot concordant to $K'$, then $s(X_n(K);1)=s(X_n(K');1)$. Thus \eqref{eq:shake_genus_bound} implies $g_{sh}^n(K)\ge s(K')/2$. On the other hand, Rasmussen's classical bound gives $s(K')/2\ge g_4(K')$. Since $g_4(K')=g_4(K)\ge g_{sh}^n(K)$, all these quantities are equal.
\end{proof}

\subsection{Another comparison result for lasagna \texorpdfstring{$s$}{s}-invariants}\label{sbsec:comparison_refined}
In this section we make use of the adjunction inequality for $s$-invariants \cite[Corollary~1.4]{ren2023lee} to deduce a comparison result for lasagna $s$-invariants in the spirit of Proposition~\ref{prop:Lee_comparison}. Then we prove Theorem~\ref{thm:shake_genus}(2). For simplicity, assume the boundary link $L$ is empty. Instead of working on the $4$-manifold level as in Section~\ref{sbsec:Lee_compare} (which is more invariant), we work on the cabled links level directly.

As in the setup of Section~\ref{sbsec:Lee_compare}, assume $X$ is a $2$-handlebody obtained by attaching $2$-handles to a framed link $K=K_1\cup\cdots K_m\subset S^3=\partial B^4$, and $X'$ is one obtained by attaching handles to $K'=K_1'\cup\cdots\cup K_m'$. 
\begin{Prop}\label{prop:las_s_compare_-CP2}
Suppose $C\colon K\to K'$ is a framed concordance in $k\overline{\CP^2}$ that restricts to concordances $C_i\colon K_i\to K_i'$. Then $$s(X';\alpha)\le s(X;\alpha)-|[C(\alpha)]|$$ for any $\alpha\in H_2(X)\cong H_2(X')\cong\Z^m$, where $[C(\alpha)]=\sum_i\alpha_i[C_i]\in H_2(k\overline{\CP^2})\cong\Z^k$.
\end{Prop}
\begin{proof}
The adjunction inequality for the $s$-invariant \cite[Corollary~1.4]{ren2023lee} in the $\mathfrak{gl}_2$ renormalization states that if $\Sigma\colon L\to L'$ is a framed cobordism in $k\overline{\CP^2}$, such that each component of $\Sigma$ has a boundary on $L$, then
\begin{equation}\label{eq:s_adj_gl2}
s_{\mathfrak{gl}_2}(L')\le s_{\mathfrak{gl}_2}(L)-\chi(\Sigma)-|[\Sigma]|.
\end{equation}

Now, the $(\alpha^++r,\alpha^-+r)$-cable of $C$ gives a concordance between $K(\alpha^++r,\alpha^-+r)$ and $K'(\alpha^++r,\alpha^-+r)$ with homology class $[C(\alpha)]$, which by \eqref{eq:s_adj_gl2} implies $$s_{\mathfrak{gl}_2}(K'(\alpha^++r,\alpha^-+r))\le s_{\mathfrak{gl}_2}(K(\alpha^++r,\alpha^-+r))-|[C(\alpha)]|.$$ The statement now follows from Proposition~\ref{prop:las_s_from_classical}.
\end{proof}

Now we prove Theorem~\ref{thm:shake_genus}(2) on shake genus of knots $K$ concordant to positive knots $K'$ via a concordance $\Sigma$ in some twice punctured $k\overline{\CP^2}$.

\begin{proof}[Proof of Theorem~\ref{thm:shake_genus}(2)]
Putting a framing on $\Sigma$ makes it a framed cobordism from $(K,n)$ to $(K',n-[\Sigma]^2)$. It follows from Proposition~\ref{prop:las_s_compare_-CP2} that $s(X_n(K);1)\ge s(X_{n-[\Sigma]^2}(K');1)+|[\Sigma]|$. If $n\le[\Sigma]^2$, by Theorem~\ref{thm:nonvan_diagram_knot} we know $s(X_{n-[\Sigma]^2}(K');1)=s(K')-n+[\Sigma]^2$, thus the statement follows.
\end{proof}

\subsection{Yasui's family of knot trace pairs}\label{sbsec:yasui}
Yasui \cite[Figure~10]{yasui2015corks} defined two families of satellite patterns $P_{n,m}$, $Q_{n,m}$, $n,m\in\Z$ (see Figure \ref{fig:Yasui_patterns}). The patterns $P_{n,0}$ are also known as the (twisted) Mazur pattern \cite[Fig.~1]{mazur1961note}, and have been the subject of study in many papers. Yasui showed that for any $n,m$ and any knot $K\subset S^3$, the two manifolds $X_n(P_{n,m}(K))$ and $X_n(Q_{n,m}(K))$, namely the $n$-traces on the two corresponding satellite knots, are homeomorphic. Moreover, he showed that if there exists a Legendrian representative $\mathcal K$ of $K$ with
\begin{equation}\label{eq:Yasui_condition}
tb(\mathcal K)+|rot(\mathcal K)|-1=2g_4(K)-2,\ n\le tb(\mathcal K),\ m\ge0,
\end{equation}
then $X_n(P_{n,m}(K))$ and $X_n(Q_{n,m}(K))$ are not diffeomorphic, thus form an exotic pair \cite[Theorem~4.1]{yasui2015corks}. In fact, under these conditions, the exotica were detected by a discrepancy in shake slice genera: $g_{sh}^n(P_{n,m}(K))=g_4(K)+1>g_{sh}^n(Q_{n,m}(K))$. Thus, the interiors of such pairs are also exotic. The case $K=U$, $n=-1$, $m=0$ yields the exotic pair $X_{-1}(-5_2)$ and $X_{-1}(P(3,-3,-8))$ in Theorem~\ref{thm:exotic}.

\begin{figure}
\[
\ILtikzpic[xscale=.5,yscale=.5]{
\draw[thick] (0.5,1) -- (1,1) to[out=0,in=-90,looseness=.5] (3,1.5);
\drawover[thick]{
    (2,1.5) to[out=-90,in=180,looseness=.5] (3,1) -- (7.5,1);
    }
\draw[thick] (2,1.5) to[out=90,in=180,looseness=.5] (3,2) to[out=0,in=-90] (4,3) -- (4,4) to[out=90,in=180] (5,5) to[out=0,in=180] (6,3) -- (8,3);
\drawover[thick]{
    (3,1.5) to[out=90,in=0,looseness=.5] (2,2) -- (0,2);
    }
\drawover[thick]{
    (8,2) -- (6,2) to[out=180,in=-90] (5,3) -- (5,4) to[out=90,in=0] (4,5) to[out=180,in=0] (0,3);
    }
\draw[fill=white] (3.8,4.2) rectangle (5.2,2.8);
\node at (4.5,3.5){$-m$};
\draw[fill=white] (6.3,3.2) rectangle (7.2,0.8);
\node at (6.75,2){$n$};
\draw[thick] (7.5,1) to[out=0,in=0] (6,-2) -- (2,-2) to[out=180,in=180] (0.5,1);
\draw[thick] (8,2) to[out=0,in=0] (6,-3) -- (2,-3) to[out=180,in=180] (0,2);
\draw[thick] (8,3) to[out=0,in=0] (6,-4) -- (2,-4) to[out=180,in=180] (0,3);
\draw[dashed] (3,-.5) to[out=45,in=135] (5,-.5);
\draw[dashed] (2,-.25) to[out=-45,in=-135] (6,-.25);
\draw[dashed] (-2.5,-.25) to[out=90,in=90,looseness=1.7] (10.5,-.25) to[out=-90,in=-90,looseness=1.4] (-2.5,-.25);
}
\quad
\ILtikzpic[xscale=.45,yscale=.35]{
\draw[thick] (0,5) to[out=0,in=-180] (2.5,8);
\draw[thick] (0,4) to[out=0,in=-180] (2.5,7);
\draw[thick] (2.5,8.5) to[out=0,in=90,looseness=.5] (3,6);
\drawover[thick]{
    (2.5,8.5) to[out=180,in=90,looseness=.5] (2,6);
    }
\draw[thick] (6.5,6) -- (6,6) to[out=180,in=90] (4,4) -- (4,3) to[out=-90,in=180] (6,1) -- (8.5,1) to[out=0,in=180,looseness=.5] (11,4) -- (13,4);
\draw[thick] (6.5,5) -- (5.5,5) to[out=180,in=90] (5,4) -- (5,3) to[out=-90,in=180] (5.5,2) -- (8.5,2) to[out=0,in=180,looseness=.5] (11,5) -- (13,5);
\drawover[thick]{
    (2.5,8) -- (8.5,8) to[out=0,in=180,looseness=.5] (11,3) -- (12.5,3);
    }
\drawover[thick]{
    (2.5,7) -- (8,7) to[out=0,in=90,looseness=.5] (10,3) -- (10,2) to[out=-90,in=0] (6.5,0) to[out=180,in=0] (0.5,3);
    }
\drawover[thick]{
    (6,6) -- (7.5,6) to[out=0,in=0] (7.5,3) -- (3,3) to[out=180,in=-90] (2,4) -- (2,6);
    }
\drawover[thick]{
    (6,5) -- (7,5) to[out=0,in=0] (7,4) -- (4,4) to[out=180,in=-90] (3,5) -- (3,6);
    }
\draw[fill=white] (5.5,8.2) rectangle (7,4.8);
\node at (6.25,6.5){$-m$};
\draw[fill=white] (5.5,2.2) rectangle (7,0.8);
\node at (6.25,1.5){$-2$};
\draw[fill=white] (11,5.2) rectangle (12,2.8);
\node at (11.5,4){$n$};
\draw[dashed,pattern=north west lines] (2,5) rectangle (3,6);
\draw[thick] (12.5,3) to[out=0,in=0] (10,-2) -- (3,-2) to[out=180,in=180] (0.5,3);
\draw[thick] (13,4) to[out=0,in=0] (10,-3) -- (3,-3) to[out=180,in=180] (0,4);
\draw[thick] (13,5) to[out=0,in=0] (10,-4) -- (3,-4) to[out=180,in=180] (0,5);
\draw[dashed] (5.5,-1) to[out=45,in=135] (7.5,-1);
\draw[dashed] (4.5,-.75) to[out=-45,in=-135] (8.5,-.75);
\draw[dashed] (-2.5,1) to[out=90,in=90,looseness=1.8] (15.5,1) to[out=-90,in=-90,looseness=1.2] (-2.5,1);
}
\]
\caption{Satellite patterns $P_{n,m}$ (left) and $Q_{n,m}$ (right). The boxes denote full twists. The shaded band gives rise to a concordance between $Q_{n,m}$ and the identity pattern.}
\label{fig:Yasui_patterns}
\end{figure}

Using lasagna $s$-invariants, we present a different family (as stated in Theorem~\ref{thm:exotic_family}) of Yasui's knot trace pairs that are exotic. We recall that our hypothesis on $(K,n,m)$ is that there exists a slice disk $\Sigma$ of $K$ in some $k\CP^2$ such that 
\begin{equation}\label{eq:s_condition}
s(K)=|[\Sigma]|-[\Sigma]^2,\ n<-[\Sigma]^2,\ m\ge0.
\end{equation}

Before giving the proof, we compare and discuss the conditions imposed on $K$ by Yasui \eqref{eq:Yasui_condition} and by us \eqref{eq:s_condition}.

The adjunction inequality for the $s$-invariant \cite[Corollary~1.5]{ren2023lee} applied to the slice disk $\Sigma$ implies that $s(K)\ge|[\Sigma]|-[\Sigma]^2$. Therefore the first condition in \eqref{eq:s_condition} is that this adjunction inequality attains equality for $\Sigma$. Similarly, the first condition in \eqref{eq:Yasui_condition} is that the slice-Bennequin inequality attains equality for $\mathcal K$. These two conditions are satisfied on somewhat different families of knots. For example, as remarked by Yasui, his condition on $K$ is satisfied for all positive torus knots. On the other hand, our condition on $K$ is satisfied for all negative torus knots (cf. Lemma~\ref{lem:s_condition_neg_torus}). Our condition also has the advantage of being closed under concordances and connected sums of knots.

\begin{proof}[Proof of Theorem~\ref{thm:exotic_family}]
As observed by Yasui, adding a band as indicated in Figure \ref{fig:Yasui_patterns} and capping off the unknot component give a concordance between $Q_{n,m}$ and the identity satellite pattern. Therefore, $Q_{n,m}(K)$ is concordant to $K$, thus by Proposition~\ref{prop:intro_Lee_comparison} and Proposition~\ref{prop:las_s_from_classical} we have 
\begin{align}\label{eq:s_Q}
&\,s(X_n(Q_{n,m}(K));1)=s(X_n(K);1)\nonumber\\
=&-\lim_{r\to\infty}(s(-K(r+1,r))+2r)-n-1+1\le-s(-K(1,0))-n=s(K)-n.
\end{align}
Deleting a local ball in $B^4$ near $\Sigma$, turning the cobordism upside down and reversing the ambient orientation, we obtain a concordance $-\Sigma^\circ\colon K\to U$ in $k\overline{\CP^2}$. Applying the satellite operation $P_{n,m}$ gives a framed concordance $C_1=P_{n,m}(-\Sigma^\circ)\colon(P_{n,m}(K),n)\to(P_{n+[\Sigma]^2,m}(U),n+[\Sigma]^2)$ in $k\overline{\CP^2}$ with $[C_1]=[-\Sigma^\circ]$ since $P_{n,m}$ has winding number $1$. Since $n<-[\Sigma]^2$, there is a framed concordance $C_2\colon(P_{n+[\Sigma]^2,m}(U),n+[\Sigma]^2)\to(P_{-1,m}(U),-1)$ in $k_1\overline{\CP^2}$ obtained by adding $k_1$ positive twists along the meridian of the satellite solid torus, which has $[C_2]=(1,\cdots,1)\in\Z^{k_1}=H_2(k_1\overline{\CP^2})$, where $k_1=-1-n-[\Sigma]^2$. Similarly, there is a null-homologous framed concordance $C_3\colon(P_{-1,m}(U),-1)\to(P_{-1,0}(U),-1)=(-5_2,-1)$ in $k_2\overline{\CP^2}$ where $k_2=m$ (note that the two strands being twisted $-m$ times in $P_{n,m}$ are oppositely oriented). Now Proposition~\ref{prop:las_s_compare_-CP2} applied to the concordance $C_1\cup C_2\cup C_3\colon(P_{n,m}(K),n)\to(-5_2,-1)$ yields
\begin{align}\label{eq:s_P}
&s(X_n(P_{n,m}(K));1)\ge s(X_{-1}(-5_2);1)+|[C_1]|+|[C_2]|+|[C_3]|\nonumber\\
=&\,3+|[\Sigma]|+k_1+0=|[\Sigma]|-[\Sigma]^2-n+2=s(K)-n+2.
\end{align}
Together, \eqref{eq:s_Q} and \eqref{eq:s_P} imply that $s(X_n(P_{n,m}(K));1)>s(X_n(Q_{n,m}(K));1)$, proving the theorem.
\end{proof}

\begin{Rmk}\label{rmk:exotic_mazur}
Let $C\cup K$ be a two-component link in $S^3$, where $C$ is an unknot and $\ell k(C,K)=1$. Regarding $C$ as a dotted $1$-handle and attaching an $n$-framed $2$-handle along $P_{n,m}(K)$ (resp. $Q_{n,m}(K)$) yields a Mazur manifold $Z_P$ (resp. $Z_Q$), namely a contractible $4$-manifold with one handle of each index $0,1,2$. The manifolds $Z_P$ and $Z_Q$ are homeomorphic, and Hayden--Mark--Piccirillo \cite[Theorem~2.7]{hayden2021exotic} showed that if $C\cup K$ satisfies the (very mild) assumptions that
\begin{itemize}
\item $K$ is not the meridian of $C$,
\item every self homeomorphism of $\partial Z_P$ preserves the JSJ torus (with orientation) given by the image of the satellite torus of $P_{n,m}(K)$ in $\partial Z_P$,
\end{itemize}
then $Z_P$ being diffeomorphic to $Z_Q$ implies that $X_n(P_{n,m}(K))$ is diffeomorphic to $X_n(Q_{n,m}(K))$. Therefore, by choosing suitable $C\cup K$, Theorem~\ref{thm:exotic_family} also yields various families of exotic Mazur manifolds.
\end{Rmk}

As an example, we present a family of knots $K$ for which Theorem~\ref{thm:exotic_family} applies.
\begin{Lem}\label{lem:s_condition_neg_torus}
For every negative torus knot $-T(p,q)$, there exists a slice disk $\Sigma$ in some $k\CP^2$ with $|[\Sigma]|-[\Sigma]^2=-(p-1)(q-1)=s(-T(p,q))$. 
\end{Lem}
\begin{proof}
We induct on $p+q$. For $p=q=1$, $-T(p,q)=U$ is the unknot and we take $\Sigma$ to be the trivial slice disk in $B^4$.

For the induction step, by symmetry, we may assume $p>q\ge1$. There is a positive twist along $q$ strands in $-T(p,q)$ that converts it into $-T(p-q,q)$. This gives a concordance $C\colon-T(p-q,q)\to-T(p,q)$ in $\CP^2$ with $[C]=q[\CP^1]$. By induction hypothesis, there is a slice disk $\Sigma'$ of $-T(p-q,q)$ in $k'\CP^2$ with $|[\Sigma']|-[\Sigma']^2=-(p-q-1)(q-1)$. Then $\Sigma=\Sigma'\cup C$ is a slice disk of $-T(p,q)$ in $(k'+1)\CP^2$ with $|[\Sigma]|-[\Sigma]^2=-(p-1)(q-1)$.
\end{proof}

\subsection{Conway knot cables can obstruct its sliceness}\label{sbsec:conway}
Recently, Piccirillo \cite{piccirillo2020conway} famously proved that the Conway knot, denoted $Conway$, is not slice. Building on her proof, we show the following.

\begin{Prop}\label{prop:Conway}
Let $Conway(n^+,n^-)$ denote the $0$-framed $(n^++n^-)$-cable of the Conway knot, with the orientation on $n^-$ of the strands reversed. Then for some $n>0$ we have $$s(Conway(n+1,n))>-2n=s(U^{\sqcup(2n+1)}).$$ Here $U^{\sqcup m}$ is the $m$-component unlink.
\end{Prop}
If $Conway$ were slice, then any cable of it is concordant to the corresponding cable of the unknot, implying that they have the same $s$-invariant. Therefore, Proposition~\ref{prop:Conway} shows that, if people had enough computer power, we could have obstructed the sliceness of the Conway knot by calculating the $s$-invariants of its cables. In practice, however, the calculation of $s(Conway(2,1))$ seems already out of reach.

\begin{proof}
Piccirillo \cite{piccirillo2020conway} found a knot $K'$ with $X_0(K')=X_0(Conway)$ and $s(K')=2$. Let $X$ be the mirror image of this common trace, i.e. $X=X_0(-Conway)=X_0(-K')$. By Proposition~\ref{prop:las_s_from_classical}, we have $$s(X;1)=-\lim_{n\to\infty}(s(K'(n+1,n))+2n)\le-s(K'(1,0))=-s(K')=-2.$$ Again, by Proposition~\ref{prop:las_s_from_classical} we have $\lim_{n\to\infty}(s(Conway(n+1,n))+2n)=-s(X;1)\ge2$ and hence $s(Conway(n+1,n))+2n\ge2>0$ for all sufficiently large $n$.
\end{proof}

\subsection{Knot traces of torus knots}\label{sbsec:torus_traces}
The first author \cite{ren2023lee} calculated the $s$-invariants of all torus links with various orientations. In this section, we use this to determine the complete structure of the Lee skein lasagna modules of large negative traces on negative torus links.

For $d=r+s>0$ and coprime $p,q\ne0$, let $T(dp,dq)_{r,s}$ denote the torus link $T(dp,dq)$ equipped with an orientation where $r$ of the strands are oriented against the other $s$ strands, we have
\begin{Lem}[{\cite[Theorem~1.1,Corollary~1.3]{ren2023lee}}]\label{lem:torus_links_s}
If $p,q>0$, then $$s(T(dp,dq)_{r,s})=(p|r-s|-1)(q|r-s|-1)-2\min(r,s).$$ If $pq<0$, then $$s(T(dp,dq)_{r,s})=C(d,p,q,|r-s|)$$ for some constant $C$ depending on $d,p,q,|r-s|$.
\end{Lem}

We study the $s$-invariant of $T(p,q)^k(r,s)$, the $k$-framed $(r+s)$-cable of $T(p,q)$, where $s$ of the strands have their orientation reversed.

\begin{Lem}\label{lem:s_Tpqrs}
For any $p,q>0$, $k\ge pq$, $r,s\ge0$, we have $$s(T(p,q)^k(r,s))=k|r-s|^2-(k-pq+p+q)|r-s|-2\min(r,s)+1.$$
\end{Lem}
\begin{proof}
When $d:=r+s=0$, the statement is true as $s(\emptyset)=1$. Suppose $d>0$. Twisting $k-pq$ times along the $d$ parallel strands in $T(dp,dq)_{r,s}=T(p,q)^{pq}(r,s)$ yields $T(p,q)^k(r,s)$; this gives rise to a cobordism in $(k-pq)\overline{\CP^2}$ from $T(dp,dq)_{r,s}$ to $T(p,q)^k(r,s)$. The adjunction inequality for $s$-invariant \cite[Corollary~1.4]{ren2023lee} applied to this cobordism implies that
\begin{align*}
s(T(p,q)^k(r,s))\le&\,s(T(dp,dq)_{r,s})+(k-pq)(|r-s|^2-|r-s|)\\
=&\,k|r-s|^2-(k-pq+p+q)|r-s|-2\min(r,s)+1.
\end{align*}
On the other hand, since one can cap off $\min(r,s)$ pairs of oppositely oriented components in $T(p,q)^k(r,s)$ by annuli to obtain $T(p,q)^k(|r-s|,0)$, we have 
\begin{align*}
s(T(p,q)^k(r,s))\ge&\,s(T(p,q)^k(|r-s|,0))-2\min(r,s)\\
=&\,k|r-s|^2-(k-pq+p+q)|r-s|-2\min(r,s)+1
\end{align*}
where the last equality holds as one can find a positive diagram of $T(p,q)^k(|r-s|,0)$, say obtained by adding $k-pq$ full twists to a standard positive diagram of $T(p|r-s|,q|r-s|)$, and then use the fact that the $s$-invariant of a link with a positive diagram $D$ is equal to the writhe of $D$ minus twice the number of Seifert circles of $D$ plus $1$.
\end{proof}

It follows from Lemma~\ref{lem:s_Tpqrs} and Proposition~\ref{prop:las_s_from_classical} that for any $p,q>0$, $k\ge pq$ and $\alpha\in H_2(X_{-k}(-T(p,q)))\cong\Z$, we have
\begin{align}\label{eq:torus_knot_traces}
s(X_{-k}(-T(p,q));\alpha)&\,=-\lim_{r\to\infty}(s(T(p,q)^k(\alpha^++r,\alpha^-+r))+2r)+w(T(p,q)^k(\alpha^++r,\alpha^-+r))-|\alpha|+1\nonumber\\&\,=-\lim_{r\to\infty}(k|\alpha|^2-(k-pq+p+q)|\alpha|-2r+1+2r)+k|\alpha|^2-|\alpha|+1\nonumber\\&\,=(k-pq+p+q-1)|\alpha|.
\end{align}

We remark that \eqref{eq:torus_knot_traces} only yields vacuous lower bounds for the smooth genus function of $X_{-k}(-T(p,q))$ via Corollary~\ref{cor:genus_bound_simpliest}.

Theorem~\ref{thm:Lee_structure} states that the Lee skein lasagna module of $X_{-k}(-T(p,q))$ is generated by Lee canonical generators $x_{\alpha^+,\alpha^-}$, $\alpha^+,\alpha^-\in H_2(X_{-k}(-T(p,q)))$. The generator $x_{\alpha^+,\alpha^-}$ has homological degree $-2\alpha^+\cdot\alpha^-$ ($\cdot$ is the intersection form on $H_2$ instead of the product on $\Z$) and homology class degree $\alpha^++\alpha^-$. Moreover, \eqref{eq:torus_knot_traces} and \eqref{eq:2-hdby_s_double=single} imply it has quantum filtration degree $$s(X_{-k}(-T(p,q));\alpha^+-\alpha^-)+2\alpha^+\cdot\alpha^-=(k-pq+p+q-1)|\alpha^+-\alpha^-|+2\alpha^+\cdot\alpha^-.$$ These data (together with Theorem~\ref{thm:Lee_structure}(4)) completely determine the structure of $\mathcal S_0^{Lee}(X_{-k}(-T(p,q)))$. More explicitly, we have the following
\begin{Prop}\label{prop:torus_knot_trace_Lee}
If $p,q>0$ are coprime and $k\ge pq$, then $s(X_{-k}(-T(p,q));\alpha)=(k-pq+p+q-1)|\alpha|$. Therefore, the associated graded vector space of $\mathcal S_0^{Lee}(X_{-k}(-T(p,q)))$ is given by $$gr_*(\mathcal S_{0,(\beta^2-\alpha^2)/2}^{Lee}(X_{-k}(-T(p,q));\alpha))=\Q$$ for $$\beta\equiv\alpha\pmod2,\quad*=\begin{cases}\frac{\alpha^2-\beta^2}2+(k-pq+p+q-1)|\beta|-1\pm1,&\beta\ne0\\\frac{\alpha^2}2,&\beta=0,\end{cases}$$ and $0$ in other tri-degrees.\qed
\end{Prop}

In the standard diagram, the number of negative crossings of $-T(p,q)$ is $pq-p$ or $pq-q$. Thus the nonvanishing bound on $k$ given by Proposition~\ref{prop:torus_knot_trace_Lee} is better than that by Theorem~\ref{thm:nonvan_diagram_knot}. However, Corollary~\ref{cor:emd_-kCP2} gives an even better bound. For example, $-T(n,n+1)$ is $n^2$-slice in $\CP^2$, thus $X_{-k}(-T(n,n+1))$ has nonvanishing Lee and Khovanov skein lasagna module for $k\ge n^2$.\smallskip

Finally we note that the second part of Lemma~\ref{lem:torus_links_s} implies that for $p,q>0$, $s(X_{pq}(T(p,q));\alpha)=-\infty$ for all $\alpha$. However, since $TB(-T(p,q))=-pq$ \cite{etnyre2001knots}, we already know from Corollary~\ref{cor:Lee_van} that $s(X_k(T(p,q));\alpha)=-\infty$ for all $k\ge pq$ and all $\alpha$.

\subsection{Expectations and calculations for \texorpdfstring{$\overline{\CP^2}$}{-CP2}}\label{sbsec:-CP^2}
In this section, we present a conjectural formula for a full computation of the Khovanov skein lasagna module for $\overline{\CP^2}$ over $\Q$ before proving Proposition \ref{prop:-CP^2} which addresses the cases that we can currently compute.  We also prove Corollary \ref{cor:emd_-kCP2} concerning the corresponding invariants for 4-manifolds which embed into some $k\overline{\CP^2}$.

Since $\overline{\CP^2}$ is the $(-1)$-trace on the unknot with an additional $4$-handle attached, the $2$-handlebody formula \eqref{eq:2-hdby} and Proposition~\ref{prop:34_hd} expresses the Khovanov skein lasagna module of $\overline{\CP}^2$ as a colimit 
\begin{equation}\label{eq:-CP2_colimit}
\mathcal S_0^2(\overline{\CP}^2;\alpha;\Q)=\mathrm{colim}_{r\to\infty}KhR_2(-T(|\alpha|+2r,|\alpha|+2r)_{\alpha_++r,\alpha_-+r};\Q)^{S_{|\alpha|+2r}}\{-|\alpha|-2r\}
\end{equation}
along the symmetrized dotted cobordism maps. Conjecture~6.1 of \cite{ren2023lee} would imply  that these morphisms are injective.

By the argument in \cite[Section~4.3]{ren2023lee}, the structure of $KhR_2(T(n,n);\Q)$ as an $S_n$-representation is determined by the structure of $KhR_2(T(n',n');\Q)^{S_{n'}}$ for $n'\le n$. Using this fact, one can check Conjecture~6.1 (equivalently Conjecture~6.1$'$) of \cite{ren2023lee} is equivalent to the following complete determination of $\mathcal S_0^2(\overline{\CP^2};\Q)$ (Note that by \eqref{eq:2-hdby_homology_mod_2}, all $\mathcal S_0^2(\overline{\CP^2};\alpha;\Q)$ are determined from those with $\alpha=0,1$).

\begin{Conj}\label{conj:-CP2}
Let $K_n(t,q)$ denote the Poincar\'e polynomial of $Kh(T(n,n-1);\Q)$ for $n>0$. Then $\mathcal S_0^2(\overline{\CP^2};0;\Q)$ has Poincar\'e polynomial 
\begin{equation}\label{eq:-CP2;0}
1+\sum_{n=1}^\infty t^{-2n^2}q^{6n^2-4n+1}K_{2n}(t,q^{-1})
\end{equation}
and $\mathcal S_0^2(\overline{\CP^2};1;\Q)$ has Poincar\'e polynomial 
\begin{equation}\label{eq:-CP2;1}
\sum_{n=1}^\infty t^{-2n^2+2n}q^{6n^2-10n+4}K_{2n-1}(t,q^{-1}).
\end{equation}
\end{Conj}

In fact, $K_n(t,q)$ also admits a conjectural recursive formula due to Shumakovitch and Turner (see \cite[Conjecture~6.2]{ren2023lee}). We have verified that this conjecture as well as Conjecture~\ref{conj:-CP2} conform with the data of $Kh(T(n,n))$ for $n\le8$.\medskip

In the rest of the section we present some calculations available for $\overline{\CP^2}$. 

The $p=q=k=1$ case of Proposition~\ref{prop:torus_knot_trace_Lee} gives $$gr_{2p^2+2p}\mathcal S_{0,-2p^2}^{Lee}(\overline{\CP^2};0)=\Q,\ p\ge0,$$ 
which by Theorem~\ref{thm:rank_ineq} implies that $\mathrm{rank}(\mathcal S_{0,-2p^2,2p^2+2p}^2(\overline{\CP^2};0))\ge1$.

On the other hand, one can check that the only contribution to the term $t^{-2p^2}q^{2p^2+2p}$ in \eqref{eq:-CP2;0} comes from the $n=p$ term in which $K_{2p}(t,q^{-1})=q^{-(4p^2-6p+1)}+q^{-(4p^2-6p+3)}+\cdots$ where $\cdots$ are terms with higher $t$-degrees (cf. \cite[Proposition~6.1]{khovanov2003patterns}, \cite[Theorem~2.1(2)]{ren2023lee}). Thus Conjecture~\ref{conj:-CP2} predicts $\mathcal S_{0,-2p^2,2p^2+2p}^2(\overline{\CP^2};0;\Q)=\Q$.

In fact, we can do better by directly resorting to \cite[Theorem~2.1(1)]{ren2023lee} to see that the generators of $Kh^{2(n+p)(n-p),q_{2n,2n}(2(n+p)(n-p))}(T(2n,2n))\cong\Z$ in the notation there after renormalization survive under the colimit of the $2$-handlebody formula \eqref{eq:-CP2_colimit}, thus descend to generators of 
\begin{equation}\label{eq:-CP2_las_0}
\mathcal S_{0,-2p^2,2p^2+2p}^2(\overline{\CP^2};0)\cong\Z,\ p\ge0.\footnote{More precisely, since $2$ is not invertible in $\Z$, Proposition~\ref{prop:2_hdby} only gives the result $\Z/N$ for some $N$. However $N$ is necessarily $0$ since we must have $\Z/N\otimes\Q=\Q$.}
\end{equation}
Similarly,
\begin{equation}\label{eq:-CP2_las_1}
\mathcal S_{0,-2p^2+2p,2p^2-1}^2(\overline{\CP^2};1)\cong\Z,\ p\ge1,
\end{equation}
and a generator comes from generators of $Kh^{2(n+p-1)(n-p),q_{2n-1,2n-1}(2(n+p-1)(n-p))}(T(2n-1,2n-1))\cong\Z$. This verifies Conjecture \ref{conj:-CP2} in these tri-gradings.

\begin{proof}[Proof of Proposition~\ref{prop:-CP^2}]
The statements for lasagna $s$-invariants and Lee skein lasagna modules follow from the $p=q=k=1$ case of Proposition~\ref{prop:torus_knot_trace_Lee} and Proposition~\ref{prop:34_hd}. The statement for Khovanov skein lasagna modules follows from \eqref{eq:-CP2_las_0}\eqref{eq:-CP2_las_1} for $\alpha=0,1$, and \eqref{eq:2-hdby_homology_mod_2} in general.
\end{proof}

\begin{proof}[Proof of Corollary~\ref{cor:emd_-kCP2}]
By the connected sum formula for lasagna $s$-invariants (Theorem~\ref{thm:s_prop}(4)) and Proposition~\ref{prop:-CP^2}, $s(k\overline{\CP^2};\alpha)=|\alpha|$. Thus by Corollary~\ref{cor:embedded_s} we have $s(X;\alpha)\ge|i_*\alpha|$. The equality when $\alpha=0$ follows from Corollary~\ref{cor:las_s_class_0}.

We have $gr_q(\mathcal S_{0,0}^{Lee}(X;\alpha))\ne0$ for $q=s(X;\alpha)$ by Definition~\ref{def:s}, and for $q=s(X;\alpha)-2$ if $\alpha\ne0\in H_2(X;\Q)$ by Theorem~\ref{thm:Lee_structure}(4). When $X$ is a $2$-handlebody, the second part of the statement thus follows from Theorem~\ref{thm:rank_ineq}.
\end{proof}

For the purpose of the next section, we prove the following local finiteness result for $\mathcal S_0^2(\overline{\CP^2};\Q)$.

\begin{Prop}\label{prop:-CP^2_upper_bound}
When evaluating at $q=1$, the Poincar\'e polynomial of $\mathcal S_0^2(\overline{\CP^2};0;\Q)$ is bounded above by \eqref{eq:-CP2;0}, and the Poincar\'e polynomial of $\mathcal S_0^2(\overline{\CP^2};1;\Q)$ is bounded above by \eqref{eq:-CP2;1}. In particular, the total dimension of $\mathcal S_{0,\ge h_0,*}^2(\overline{\CP^2};\alpha;\Q)$ is finite for every $h_0\in\Z$, $\alpha\in H_2(\overline{\CP^2})\cong\Z$.
\end{Prop}
\begin{proof}
\cite[Theorem~2.1(2)]{ren2023lee} implies that the highest $t$-degree of $t^{-\lfloor n^2/2\rfloor}K_n$ is at most $-(n-1)/2$, which tends to $-\infty$ as $n\to+\infty$. Thus the last finiteness statement for $\alpha=0,1$ follows from the upper bound statement. The case for general $\alpha$ follows from \eqref{eq:2-hdby_homology_mod_2}.

Let $L_n,M_n$ denote the Poincar\'e polynomial of $Kh(T(n,n);\Q),Kh(T(n,n);\Q)^{S_n}$, respectively. The argument in \cite[Section~4.3]{ren2023lee} implies $$L_n=\sum_{i=0}^{\lfloor n/2\rfloor}\dim(n-i,i)t^{2i(n-i)}q^{6i(n-i)}M_{n-2i}.$$
Define $M_n'$ inductively by $M_0'=M_0$, $M_1'=M_1$, $M_n'=t^{2n-2}q^{6n-8}M_{n-2}'+q^{n-1}K_n$ for $n\ge2$. Define $$L_n'=\sum_{i=0}^{\lfloor n/2\rfloor}\dim(n-i,i)t^{2i(n-i)}q^{6i(n-i)}M_{n-2i}'.$$ It follows from a calculation that Conjecture~6.1' in \cite{ren2023lee} is equivalent to the assertion that $M_n=M_n'$. In any case, it holds true that $L_n\le L_n'$.

By the $2$-handlebody formula \eqref{eq:-CP2_colimit} and taking into account the renormalizations, the Poincar\'e polynomial of $\mathcal S_0^2(\overline{\CP^2};0;\Q)$ is bounded above by 
\begin{equation}\label{eq:M_n_liminf}
\liminf_{n\to\infty}t^{-2n^2}q^{6n^2-2n}M_{2n}(t,q^{-1}),
\end{equation}
while \eqref{eq:-CP2;0} is equal to $$\lim_{n\to\infty}t^{-2n^2}q^{6n^2-2n}M_{2n}'(t,q^{-1}).$$ Here the limits are interpreted coefficientwise, and the terms in the second limit are nondecreasing.

Now we collapse the quantum grading by setting $q=1$. Assume, to the contrary, that the coefficient of some $t^m$ in \eqref{eq:M_n_liminf} at $q=1$ is larger than that in \eqref{eq:-CP2;0}. Denote by $A,B$ these coefficients, respectively, so $A>B$. Then $coef_{t^m}(t^{-2n^2}M_{2n}'(t,1))\le B$ for all $n$ and $coef_{t^m}(t^{-2n^2}M_{2n}(t,1))\ge A$ for all $n\ge N$ for some $N$. It follows that
\begin{align}\label{eq:coef_B}
coef_{t^{2n^2+m}}(L_{2n}'(t,1))=&\,\sum_{i=0}^n\dim(2n-i,i)coef_{t^m}(t^{-2(n-i)^2}M_{2n-2i}'(t,1))\nonumber\\\le&\,B\sum_{i=0}^n\dim(2n-i,i)=B\binom{2n}{n}.
\end{align}
Similarly, for $n>N$,
\begin{align}\label{eq:coef_A}
\,coef_{t^{2n^2+m}}(L_{2n}(t,1))=&\,\sum_{i=0}^n\dim(2n-i,i)coef_{t^m}(t^{-2(n-i)^2}M_{2n-2i}(t,1))\nonumber\\\ge&\,A\sum_{i=0}^{n-N}\dim(2n-i,i)=A\binom{2n}{n-N}.
\end{align}
Since $A>B$, for $n$ sufficiently large, \eqref{eq:coef_A} is larger than \eqref{eq:coef_B}, contradicting the fact that $L_{2n}\le L_{2n}'$. This proves the statement for $\mathcal S_0^2(\overline{\CP^2};0;\Q)$. The statement for $\mathcal S_0^2(\overline{\CP^2};1;\Q)$ is analogous.
\end{proof}

\subsection{\texorpdfstring{$\overline{\CP^2}$}{-CP2}-stably standard homotopy spheres}\label{sbsec:htp_S4}
We show that the Khovanov skein lasagna module, at least over $\Q$, is unable to detect potential exotic $4$-spheres that dissolve after taking connected sums with copies of $\overline{\CP^2}$. Note that homotopy $4$-spheres coming from Gluck twists dissolve after a single connected sum with $\overline{\CP^2}$.
\begin{Prop}
If $\Sigma$ is a homotopy $4$-sphere with $\Sigma\#k\overline{\CP^2}=k\overline{\CP^2}$ for some $k$, then $\mathcal S_0^2(\Sigma;\Q)=\Q$, concentrated in tri-degree $(0,0,0)$.
\end{Prop}
\begin{proof}
Applying $\mathcal S_0^2$ to $\Sigma\#k\overline{\CP^2}=k\overline{\CP^2}$ and using the connected sum formula (Proposition~\ref{prop:cntd_sum}), we obtain $\mathcal S_0^2(\Sigma;\Q)\otimes\mathcal S_0^2(\overline{\CP^2};\Q)^{\otimes k}\cong\mathcal S_0^2(\overline{\CP^2};\Q)^{\otimes k}$. Since $\mathcal S_0^2(\overline{\CP^2};\Q)$ is nonzero and satisfies a suitable finiteness condition given by Proposition~\ref{prop:-CP^2_upper_bound}, we conclude that $\mathcal S_0^2(\Sigma;\Q)=\Q$.
\end{proof}

\subsection{Induced map of link cobordisms in \texorpdfstring{$k\overline{\CP^2}$}{-kCP2}}\label{sbsec:Kh_CP2_cob}
Unless otherwise stated, in this section we work over a field $\mathbb F$, which will be dropped from the notation throughout.

We begin with a general construction. Suppose $X$ is a closed $4$-manifold and $L_0,L_1\subset S^3$ are two framed oriented links. A cobordism from $L_0$ to $L_1$ in $X$ is a properly embedded framed oriented surface $\Sigma\subset X^{\circ\circ}:=X\backslash(int(B^4)\sqcup int(B^4))$ with $\partial\Sigma=-L_0\sqcup L_1\subset S^3\sqcup S^3$. Thus $\Sigma$ is a skein in $X^{\circ\circ}$ rel $-L_0\sqcup L_1$ without input balls, which itself defines a lasagna filling with homological degree $0$, quantum degree $-\chi(\Sigma)$, and homology class degree $[\Sigma]$, which represents a class $x_\Sigma\in\mathcal S_{0,0,-\chi(\Sigma)}^2(X^{\circ\circ};-L_0\sqcup L_1;[\Sigma])$.

On the other hand, since $(X^{\circ\circ},-L_0\sqcup L_1)=(X,\emptyset)\#(B^4,-L_0)\#(B^4,L_1)$, the connected sum formula (Proposition~\ref{prop:cntd_sum}) and Proposition~\ref{prop:recover_Z} imply that over a field $\mathbb F$, 
\begin{align*}
\mathcal S_0^2(X^{\circ\circ};-L_0\sqcup L_1;[\Sigma])\cong&\,\mathcal S_0^2(X;[\Sigma])\otimes\mathcal S_0^2(B^4;-L_0)\otimes\mathcal S_0^2(B^4;L_1)\\
\cong&\,\mathcal S_0^2(X;[\Sigma])\otimes KhR_2(-L_0)\otimes KhR_2(L_1)\\
\cong&\,\mathcal S_0^2(X;[\Sigma])\otimes\mathrm{Hom}(KhR_2(L_0),KhR_2(L_1)).
\end{align*}
If $\theta\in\mathcal(S_0^2(X;[\Sigma]))_{h_0,q_0}^*$ is a dual Khovanov skein lasagna class, then we obtain a homomorphism $$KhR_2(\Sigma,\theta):=\theta(x_\Sigma)\colon KhR_2^{*,*}(L_0)\to KhR_2^{*+h_0,*+q_0-\chi(\Sigma)}(L_1).$$

In the case when $X=\overline{\CP^2}$, there are natural dual lasagna classes, one for each homology class $\alpha\in H_2(X)$. In fact, by Proposition~\ref{prop:-CP^2}, for each $\alpha\in H_2(\overline{\CP^2})$, we have $\mathcal S_{0,0,|\alpha|}^2(\overline{\CP^2};\alpha;\Z)=\Z$. Therefore over a field $\mathbb F$ and up to sign, $(\mathcal S_0^2(\overline{\CP^2};\alpha))_{0,-|\alpha|}^*\cong\mathbb F$ has a canonical dual lasagna generator $\theta_\alpha$. Thus for a cobordism $\Sigma\colon L_0\to L_1$ in $\overline{\CP^2}$, there is an induced map $$KhR_2(\Sigma):=KhR_2(\Sigma,\theta_{[\Sigma]})\colon KhR_2^{*,*}(L_0)\to KhR_2^{*,*-\chi(\Sigma)-|[\Sigma]|}(L_1)$$ defined up to sign. Translate back to the more familiar Khovanov homology and reverse the ambient orientation, we see if $\Sigma\colon L_0\to L_1$ is a cobordism in $\CP^2$ between unframed oriented links, then there is an induced map $Kh(\Sigma)\colon Kh(L_0)\to Kh(L_1)$ defined up to sign, with bidegree $(0,\chi(\Sigma)-[\Sigma]^2+|[\Sigma]|)$, at least over a field. This map only depends on the isotopy class of $\Sigma$ rel boundary.

In fact, we can also give the following more direct construction of $Kh(\Sigma)\colon Kh(L_0)\to Kh(L_1)$ over $\Z$, which is proposed to us by Hongjian Yang. By transversality, suppose $\Sigma$ intersects the core $\CP^1$ of $\CP^2$ transversely, in some $p$ points positively and $q$ points negatively. A tubular neighborhood of $\CP^1$ has boundary $S^3$ on which $\Sigma$ is the torus link $-T(p+q,p+q)_{p,q}\subset S^3$ (the torus link $-T(p+q,p+q)$ equipped with an orientation where $p$ of the strands are oriented against the other $q$ strands). Tubing this boundary $S^3$ with the negative boundary of $(\CP^2)^{\circ\circ}$ changes $\Sigma$ into a cobordism $\Sigma^\circ\colon L_0\sqcup(-T(p+q,p+q)_{p,q})\to L_1$ in $S^3\times I$, inducing a map
\begin{equation}\label{eq:Kh_intermediate_map}
Kh(\Sigma^\circ)\colon Kh(L_0)\otimes Kh(-T(p+q,p+q)_{p,q})\to Kh(L_1)
\end{equation}
up to sign with bidegree $(0,\chi(\Sigma^\circ))$ where $\chi(\Sigma^\circ)=\chi(\Sigma)-p-q$. On the other hand, \cite[Corollary~2.2]{ren2023lee} implies that $Kh^{0,-(p-q)^2+2\max(p,q)}(-T(p+q,p+q)_{p,q};\Z)\cong\Z$. Plugging a generator of it into \eqref{eq:Kh_intermediate_map} defines a map $$Kh(\Sigma)\colon Kh(L_0)\to Kh(L_1)$$ over $\Z$ up to sign with bidegree $(0,\chi(\Sigma^\circ)-(p-q)^2+2\max(p,q))=(0,\chi(\Sigma)-[\Sigma]^2+|[\Sigma]|)$ as $[\Sigma]=\pm(p-q)[\CP^1]$.

We argue briefly that this direct construction agrees with the previous one over a field. We begin with the lasagna construction. Realize $(\overline{\CP^2})^{\circ\circ}$ as $S^3\times I\#B^4$ with a $2$-handle attached to $\partial B^4$ along a $(-1)$-framed unknot and another $4$-handle attached. Then a suitable generalization of the $2$-handlebody formula (see \cite[Theorem~3.2]{manolescu2023skein}) gives 
\begin{align}\label{eq:2-hdby_CP2}
&\ \mathcal S_0^2((\overline{\CP^2})^{\circ\circ};-L_0\sqcup L_1;\alpha)\nonumber\\
\cong&\,\left(\bigoplus_r\mathcal S_0^2(S^3\times I\#B^4;-L_0\sqcup L_1\sqcup-T(\alpha^++r,\alpha^-+r)))\{-|\alpha|-2r\}\right)\bigg/\sim
\end{align}
for some equivalence relation $\sim$, and the element $x_\Sigma$ for a skein $\Sigma$ with $[\Sigma]=\alpha$ is represented by the class of $\Sigma^\circ$ as a lasagna filling in the right-hand side of \eqref{eq:2-hdby_CP2} ($\Sigma^\circ$ is a skein without input balls in $S^3\times I\#B^4$ rel some $-L_0\sqcup L_1\sqcup-T(p+q,p+q)_{p,q}$ obtained by cutting out a tubular neighborhood of a core $\overline{\CP^1}\subset\overline{\CP^2}$). The evaluation of $x_\Sigma$ by the dual lasagna class $\theta_{[\Sigma]}\in\mathcal S_0^2(\overline{\CP^2})^*$ is realized by regarding $\Sigma^\circ$ as a skein in $S^3\times I$ rel $-L_0\sqcup L_1$, where $-T(p+q,p+q)_{p,q}\subset\partial B^4$ is regarded as an input link on an input ball (so its orientation is reversed) with decoration given by a generator of $KhR_2^{0,-2\max(p,q)}(T(p+q,p+q)_{p,q})\cong\mathbb F$. Tubing this input ball to the negative boundary, then Proposition~\ref{prop:recover_Z} shows the equivalence to the direct construction.

In the case when a cobordism $\Sigma\colon L_0\to L_1$ in $\CP^2$ is given by a negative full twist along some $p+q$ strands of $L$ ($p$ of which oriented against the other $q$), the induced map $Kh(\Sigma)\colon Kh(L_0)\to Kh(L_1)$ can be described by putting a local torus link $-T(p+q,p+q)_{p,q}$ near $L_0$, decorated by a generator of $Kh^{0,-(p-q)^2+2\max(p,q)}(-T(p+q,p+q)_{p,q})$, and performing $p+q$ saddles between $-T(p+q,p+q)_{p,q}$ and $L_0$ to realize the twist.\medskip

We note that the same constructions of induced maps on $Kh$---both the lasagna approach and the direct approach---also apply to cobordisms in $k\CP^2$. For the lasagna approach, this is because the preferred dual lasagna classes $\theta_\alpha$ of $\overline{\CP^2}$ give rise to preferred dual lasagna classes $\otimes_{i=1}^k\theta_{\alpha_i}$ of $k\overline{\CP^2}$. For the direct approach, we simply cut out neighborhoods of every core $\CP^1$ in $k\CP^2$. We summarize as follows.

\begin{Prop}\label{prop:Kh_CP2_cob}
For an oriented link cobordism $\Sigma\colon L_0\to L_1$ in $k\CP^2$, there is an induced map $$Kh(\Sigma)\colon Kh(L_0)\to Kh(L_1)$$ defined over any coefficient field up to sign, with bidegree $(0,\chi(\Sigma)-[\Sigma]^2+|[\Sigma]|)$, which only depends on the isotopy class of $[\Sigma]$ rel boundary. It agrees with the classical induced map for link cobordisms in $S^3\times I$ when $k=0$. Moreover, if $\Sigma'\colon L_1\to L_2$ is another oriented link cobordism in some $k'\CP^2$, then $Kh(\Sigma'\circ\Sigma)=Kh(\Sigma')\circ Kh(\Sigma)$ up to sign.\qed
\end{Prop}
The two perspectives compensate each other. The direct construction has the advantage that the induced map is computable from diagrams; however, it is not a priori clear that it is independent of the choices of the core $\CP^1$'s and the paths used to tube the extra boundaries to the negative boundary. The lasagna construction is manifestly well-defined and satisfies the functoriality statement in Proposition~\ref{prop:Kh_CP2_cob}; yet it is not directly computable.

It would be interesting to investigate some properties of this construction. For example, consider a link $L\subset S^3$ and a collection of disjoint $2$-disks $D=\sqcup_{i=1}^k D_i \subset S^3$ intersecting $L$ transversely, with $\partial D$ disjoint from $L$, such that the algebraic intersection number between each $D_i$ and $L$ is zero. Then one might hope that twisting along the various $D_i$'s induces maps whose limit recovers the $\#^k(S^1\times S^2)$ homology \cite{rozansky2010categorification,willis2021khovanov} of the link $L'\subset\#^k(S^1\times S^2)$ obtained by regarding $L$ as a link in the $0$-surgeries along $\partial D$, or whose colimit recovers the dual of the $\#^k(S^1\times S^2)$ homology of $-L'$, the mirror image of $L'$.

\subsection{A conjecture for Stein manifolds}\label{sbsec:stein}
Some similarities between skein lasagna modules and other gauge/Floer-theoretic smooth invariants appear in our study. For example, the skein lasagna modules are sensitive to orientations, and there is an adjunction-type inequality (Theorem~\ref{thm:intro_genus_bound}) for lasagna $s$-invariants. However, a serious constraint for our computation is that we were not able to produce any $4$-manifold with $b_2^+>0$ that has finite lasagna $s$-invariants or nonvanishing Khovanov skein lasagna module. In fact, if one were able to find a simply-connected closed $4$-manifold $X$ with $b_2^+(X)>0$ that has $\mathcal{S}_0^2(X)\ne0$, then $X\#\overline{\CP^2}$ is an exotic closed $4$-manifold, as it has nonvanishing $\mathcal S_0^2$ while its standard counterpart (some $k\CP^2\#\ell\overline{\CP^2}$) has vanishing $\mathcal S_0^2$.

To pursue this thread, we propose the following conjecture, which seems plausible and desirable.

\begin{Conj}\label{conj:stein}
If $K$ is a knot and $X=X_n(K)$ for some $n<TB(K)$ (hence $X$ is Stein \cite{eliashberg1990topological,gompf1998handlebody}), then $s(X;0)=0$ and $s(X;1)=s(K)-n$. In particular, $\mathcal S_0^2(X)\ne0$. More generally, any Stein $2$-handlebody $X$ has $s(X;0)>-\infty$ and nonvanishing Khovanov skein lasagna module.
\end{Conj}
The part concerning $s(X;1)$ in Conjecture~\ref{conj:stein} implies the $n$-shake genus of $K$ for $n<TB(K)$ (or more generally $n<TB(K')$ for some $K'$ concordant to $K$) is bounded below by $s(K)/2$, generalizing Theorem~\ref{thm:shake_genus}(1). This seems to be the proper straightening of the bound from the Stein adjunction inequality \eqref{eq:Stein_shake_genus}.

By Proposition~\ref{prop:las_s_from_classical}, Conjecture~\ref{conj:stein} is equivalent to an assertion on the $s$-invariants of cables of $K$ (or some link in the $2$-handlebody case).

\printbibliography
\appendix
\section{Proof of Lemma~\ref{lem:full_comparison}}\label{sec:append}
We first rewrite Lemma~\ref{lem:full_comparison} in terms of the classical Khovanov homology. For a framed link $L$, let $Q(L)(t,q)$ denote the Poincar\'e polynomial of $Kh(L)$, and $\tilde Q(L)(t,q):=(tq^3)^{-\tfrac12w(L)}Q(L)$. Then by \eqref{eq:KhR_2}, $P(L)(t,q)=q^{-w(L)}Q(-L)(t,q^{-1})$ and $\tilde P(L)(t,q)=\tilde Q(-L)(t,q^{-1})$. Thus Lemma~\ref{lem:full_comparison} is equivalent to the following Khovanov version.
\begin{Lem}[Comparison lemma]\label{lem:full_comparison_Kh}
For $n\in\Z_{\ge0}^m$ and $p\in\Z^m$ with $p\ge w$, we have 
\begin{equation}\label{eq:compare_Kh}
\tilde Q(K^p(n))\le\sum_{\substack{0\le n'\le n\\2|n-n'}}R_{n,n',p-w}\tilde Q(K^w(n'))
\end{equation}
for some polynomials $R_{n,n',p-w}\in\Z[t^{\pm1/2},q^{\pm1/2}]$ independent of $K$ with $$R_{n,n',p-w}(t,q)=t^{-\sum_i\frac{(p_i-w_i)n'^2_i}{2}}q^{-\sum_i\frac{(p_i-w_i)(n'^2_i+2n'_i)}{2}}\prod_{j=1}^m\left(\sum_{r_j=0}^{\frac{n_j-n'_j}{2}}\dim\left(\frac{n_j+n'_j}{2}+r_j,\frac{n_j-n'_j}{2}-r_j\right)q^{-2r_j}\right)+\cdots$$
where $\cdots$ are terms with lower $t$-degrees.

Moreover, if \eqref{eq:compare_Kh} is sharp at the highest $t$-degree of the left-hand side, then for any $i=1,\cdots,m$ with $n_i\ge2$, the morphism on $Kh$ induced by the dotted annular creation map $K^p(n-2e_i)\to K^p(n)$ is injective in the highest nontrivial homological degree of $Kh(K^p(n))$.
\end{Lem}

Lemma~\ref{lem:full_comparison_Kh} can be reduced to an elementary piece. Suppose $J\cup L\subset S^3$ is a framed oriented link, where $J$ is a knot and $L$ is a link. Let $J^1$ be $J$ with framing one higher. For $n\ge0$, let $J(n)\cup L$ denote the link obtained from $J\cup L$ by replacing $J$ by an $n$-cable of itself, and $J^1(n)\cup L$ similarly.

\begin{Lem}[Elementary comparison lemma]\label{lem:comparison}
\begin{equation}\label{eq:compare_ele}
\tilde Q(J^1(n)\cup L)\le\sum_{\substack{0\le n'\le n\\2|n-n'}}R_{n,n'}\tilde Q(J(n')\cup L)
\end{equation}
for some polynomials $R_{n,n'}\in\Z[t^{\pm1/2},q^{\pm1/2}]$ independent of $J\cup L$ with $$R_{n,n'}(t,q)=t^{-\frac{n'^2}{2}}q^{-\frac{n'^2}{2}-n'}\sum_{r=0}^{\frac{n-n'}{2}}\dim\left(\frac{n+n'}{2}+r,\frac{n-n'}{2}-r\right)q^{-2r}+\cdots$$ where $\cdots$ are terms with lower $t$-degrees.

Moreover, if \eqref{eq:compare_ele} is sharp at the highest $t$-degree of the left-hand side, then the morphism on $Kh$ induced by the dotted annular creation map $J^1(n-2)\cup L\to J^1(n)\cup L$ is injective in the highest nontrivial homological degree of $Kh(J^1(n)\cup L)$.
\end{Lem}

It is easy to check that Lemma~\ref{lem:full_comparison_Kh} is obtained by iteratively applying Lemma~\ref{lem:comparison} to components of $K$ until the framings drop to $w$. Therefore we only prove Lemma~\ref{lem:comparison}.

Before starting the proof, we introduce an auxiliary family of links defined by St\v osi\'c \cite{stovsic2007homological,stovsic2009khovanov} and re-exploited by Tagami \cite{tagami2013maximal} (our notation aligns with \cite[Section~5]{ren2023lee}, which is different from St\v osi\'c's or Tagami's).

For $n,m\ge0$, $0\le i\le n-1$, let $D_{n,m}^i$ denote the braid closure of $(\sigma_1\cdots\sigma_{n-1})^m\sigma_1\cdots\sigma_i\in B_n$, regarded as an annular link. Thus $D_{n,m+1}^0=D_{n,m}^{n-1}$, $D_{n,n}^0$ is the annular torus link $T(n,n)$, and $D_{n,m}^i=\emptyset$ if $n=0$. For $i>0$, in the standard diagram of $D_{n,m}^i$ defined by the braid word, resolving the last letter $\sigma_i$ leads to the $0$-resolution which is exactly $D_{n,m}^{i-1}$, and some $1$-resolution, denoted $E_{n,m}^{i-1}$.

\begin{Lem}\label{lem:Stosic}
The annular link $E_{n,m}^{i-1}$ is isotopic to some $D_{n_0,m_0}^{i_0}$ or its flip (i.e. the braid closure of $\sigma_{n_0-i_0}\cdots\sigma_{n_0-1}(\sigma_1\cdots\sigma_{n_0-1})^{m_0}$, still denoted $D_{n_0,m_0}^{i_0}$ by an abuse of notation), possibly disjoint union with a local unknot. Moreover,\vspace{-5pt}
\begin{enumerate}[(1)]
\item $0\le n_0<n$, $2|n-n_0$.
\item If $m<n$, then we may assume $m_0<n_0$ or $n_0=0$.
\item $E_{n,n-1}^{n-2}\simeq D_{n-2,n-3}^{n-3}\sqcup U$ and $E_{n,n-1}^i\simeq D_{n-2,n-3}^i$ if $i<n-2$. Here, we define $D_{0,-1}^{-1}=D_{0,0}^0$.
\item If $m<n$, then the number of negative crossings in the $1$-resolution of $D_{n,m}^i$ oriented via $D_{n_0,m_0}^{i_0}(\sqcup U)$, denoted $\omega_{n,m}^i$, is less than or equal to $\frac{n^2-n_0^2}{2}-1$, with equality if and only if $m=n-1$.\vspace{-5pt}
\end{enumerate}
\end{Lem}
\begin{proof}[Sketch proof]
The fact that $E_{n,m}^{i-1}$ is isotopic to some $D_{n_0,m_0}^{i_0}$ or $D_{n_0,m_0}^{i_0}\sqcup U$ is checked carefully in \cite[Figure~15\textasciitilde20]{tagami2013maximal}, from which (1) and (2) are also clear. (3) is exactly a half of \cite[Lemma~5.1]{ren2023lee}.

For (4), when $m=n-1$ there are exactly $2n-3=\frac{n^2-(n-2)^2}{2}-1$ negative crossings as claimed in the paragraph after Lemma~5.1 in \cite{ren2023lee}.

When $m<n-1$, we count the contribution as the turnback arc in Figure~15 to Figure~18 of \cite{tagami2013maximal} travels along the longitude. Note that the whole isotopy $E_{n,m}^{i-1}\simeq D_{n_0,m_0}^{i_0}(\sqcup U)$ is decomposed into $\frac{n-n_0}{2}$ full longitude travels of the turnback arc and an extra final move according to the four types listed in Figure~15 to Figure~18 of \cite{tagami2013maximal}.

Interpret $n$ as the number of through strands in the bottom box and $m$ as the number of full $\sigma_1\cdots\sigma_{n-1}$ rotations. Then each full longitude travel decreases $n$ by $2$, decreases $m$ by $0$ or $2$ (other than the last longitude travel of Type 4, which will be discussed below).  The negative crossing contribution of this full longitude travel is $m+n-2+\epsilon$ if $m$ is decreased by $2$ and $m+\epsilon$ if $m$ is not decreased, where $\epsilon=0$ or $1$ is the contribution outside the bottom box. The last longitude travel has $\epsilon=0$. The negative crossing contribution from the final move is $0$ for type $1,2$, and at most $n-1$ for type $3,4$ (for these two types, the last longitude travel decreases $m$ by either $0$ or $1$, contributing $m$ negative crossings). In total, the number of negative crossings is at most
\[
\sum_{r=0}^{(n-n_0)/2-1}((m-2r)+n-2r-1)\le\sum_{r=0}^{(n-n_0)/2-1}(2n-4r-3)=\frac{n^2-n_0^2}{2}-\frac{n-n_0}2\le\frac{n^2-n_0^2}{2}-1.
\]
If the equality holds, we must have $m=n-2$ and $n_0=n-2$. This implies $i=n-1$ or $n-2$. One can check $E_{n,n-2}^{n-2}\simeq D_{n-2,n-4}^{n-3}\simeq E_{n,n-2}^{n-3}$, and the $1$-resolutions of $D_{n,n-2}^{n-1}$, $D_{n,n-2}^{n-2}$ have $2n-4,2n-5$ crossings, respectively, which are strictly less than $\frac{n^2-(n-2)^2}{2}-1$.
\end{proof}

\begin{proof}[Proof of Lemma~\ref{lem:comparison}]
Let $D_{n,m}^i(J)\cup L\subset S^3$ be the link obtained by replacing $J$ by its $D_{n,m}^i$ satellite. For $i>0$, the resolution of $D_{n,m}^i$ at the last letter $\sigma_i$ induces a resolution of $D_{n,m}^i(J)\cup L$ into $D_{n,m}^{i-1}(J)\cup L$ and $E_{n,m}^{i-1}(J)\cup L$, the latter of which is isotopic to some $D_{n_0,m_0}^{i_0}(J)\cup L$ possibly disjoint union with an unknot.

Fix a framed link diagram of $J\cup L$. Then $D_{n,m}^i(J)\cup L$ have (almost) standard diagrams by taking the blackboard $n$-cable of $J$ and inserting the standard diagrams of $D_{n,m}^i$ locally. Then, the number of negative crossings of $D_{n,m}^i(J)\cup L$ is $$n^2\cdot n_-(J)+n\cdot n_-(J,L)+n_-(L)$$ where $n_\pm(J),n_\pm(L),n_\pm(J,L)$ are the numbers of positive/negative crossings of $J$, of $L$, and between $J,L$, respectively. Since the orientation of $\frac{n-n_0}{2}$ of the strands are reversed in the $1$-resolution of $D_{n,m}^i(J)\cup L$, the number of negative crossings in the $1$-resolution is 
\begin{align*}
&\,\left((\frac{n+n_0}{2})^2+(\frac{n-n_0}{2})^2\right)n_-(J)+2\cdot\frac{n+n_0}{2}\cdot\frac{n-n_0}{2}\cdot n_+(J)+\omega_{n,m}^i\\
+&\,\frac{n+n_0}{2}n_-(J,L)+\frac{n-n_0}{2}n_+(J,L)+n_-(L),
\end{align*}
The difference between the numbers of negative crossings of the $1$-resolution of $D_{n,m}^i(J)\cup L$ and $D_{n,m}^i(J)\cup L$ itself is
\begin{equation}\label{eq:omega_JL}
\omega_{n,m}^i(J,L):=\frac{n^2-n_0^2}{2}w(J)+(n-n_0)\ell k(J,L)+\omega_{n,m}^i.
\end{equation}

The skein exact triangle on Khovanov homology corresponding to the resolution is
\begin{equation}\label{eq:Stosic_exact_triangle}
\begin{tikzcd}[column sep=-20pt]
Kh(E_{n,m}^{i-1}(J)\cup L)[\omega_{n,m}^i(J,L)+1]\{3\omega_{n,m}^i(J,L)+2\}\ar[rr]&&Kh(D_{n,m}^i(J)\cup L)\ar[dl]\\
&Kh(D_{n,m}^{i-1}(J)\cup L)\{1\},\ar[ul,"\text{[}1\text{]}"]&
\end{tikzcd}
\end{equation}

which implies that 
\begin{equation}\label{eq:Q_exact_triangle}
Q(D_{n,m}^i(J)\cup L)\le qQ(D_{n,m}^{i-1}(J)\cup L)+t^{\omega_{n,m}^i(J,L)+1}q^{3\omega_{n,m}^i(J,L)+2}Q(E_{n,m}^{i-1}(J)\cup L).
\end{equation}
Start with $J^1(n)\cup L=D_{n,n-1}^{n-1}(J)\cup L$ and apply \eqref{eq:Q_exact_triangle} or the identity $D_{n,m+1}^0=D_{n,m}^{n-1}$ iteratively. By Lemma~\ref{lem:Stosic}, each $E_{n,m}^{i-1}$ is some simpler $D_{n_0,m_0}^{i_0}$ (possibly with an unknot component which contributes a multiplicative factor of $q+q^{-1}$ to $Q$), and thus we must end with an inequality \begin{equation}\label{eq:Stosic_ineq}
Q(J^1(n)\cup L)\le\sum_{\substack{0\le n'\le n\\2|n-n'}}S_{n,n',J,L}Q(J(n')\cup L)
\end{equation}
for some polynomials $S_{n,n',J,L}$. These polynomials are determined as follows. Define a graph with vertex set $(\{(n,m,i)\,\colon\,0\le i<n,\ 0\le m<n\}\cup\{(0,0,0)\})/\sim$ where the equivalence relation identifies $(n,m,0)$ with $(n,m-1,n-1)$. For $i>0$, construct a directed edge $(n,m,i)\to(n,m,i-1)$ with weight $q$, and a directed edge $(n,m,i)\to(n_0,m_0,i_0)$ with weight $t^{\omega_{n,m}^i+1}q^{3\omega_{n,m}^i+2}$ if $E_{n,m}^{i-1}\simeq D_{n_0,m_0}^{i_0}$ and weight $t^{\omega_{n,m}^i+1}q^{3\omega_{n,m}^i+2}(q+q^{-1})$ if $E_{n,m}^{i-1}\simeq D_{n_0,m_0}^{i_0}\sqcup U$ (here, we demand $(n_0,m_0,i_0)=(0,0,0)$ if $D_{n_0,m_0}^{i_0}=\emptyset$ or $U$). The weight of a directed path from $(n,m,i)$ to $(n',m',i')$ is the product of its edge weights, with an extra multiplicative factor $(tq^3)^{\frac{n^2-n'^2}{2}w(J)+(n-n')\ell k(J,L)}$ to account for the renormalization factors in \eqref{eq:omega_JL}. Then $S_{n,n',J,L}$ is the sum of the path weights over directed paths from $(n,n,0)$ to $(n',0,0)$.

By Lemma~\ref{lem:Stosic}(4), the $t$-degree of every path $(n,n,0)\to(n',0,0)$ is at most $\frac{n^2-n'^2}2w(J)+(n-n')\ell k(J,L)+\frac{n^2-n'^2}{2}$, with equality achieved if and only if every edge on it of the form $(n,m,i)\to(n_0,m_0,i_0)$ that decreases $n$ has $m=n-1$. We can also enumerate all paths that achieve equality thanks to Lemma~\ref{lem:Stosic}(3). A few path-counting combinatorics then yield (cf. \cite[Figure~4]{ren2023lee}) $$S_{n,n',J,L}(t,q)=(tq^3)^{\frac{n^2-n'^2}{2}w(J)+(n-n')\ell k(J,K)}t^{\frac{n^2-n'^2}{2}}q^{\frac{3n^2-n'^2}{2}-n'}\sum_{r=0}^{\frac{n-n'}{2}}\dim\left(\frac{n+n'}{2}+r,\frac{n-n'}{2}-r\right)q^{-2r}+\cdots$$ where $\cdots$ are terms with lower homological degrees. Since $w(J^1(n)\cup L)=n^2(w(J)+1)+2n\ell k(J,L)+w(L)$ and $w(J(n')\cup L)=n'^2w(J)+2n'\ell k(J,L)+w(L)$, \eqref{eq:Stosic_ineq} renormalizes to yield \eqref{eq:compare_ele}, as desired.

Finally, if \eqref{eq:compare_ele} is sharp in the sense of the condition in the lemma, then the exact triangle \eqref{eq:Stosic_exact_triangle} at $(n,m,i)=(n,n-1,n-1)$ splits into a short exact sequence in the highest homological degree. In particular 
\begin{equation}\label{eq:ED_inj}
Kh(E_{n,n-1}^{n-2}(J)\cup L)\to Kh(D_{n,n-1}^{n-1}(J)\cup L)
\end{equation}
is injective in the highest homological degree of the right-hand side. Note $D_{n,n-1}^{n-1}(J)\cup L=J^1(n)\cup L$ and $E_{n,n-1}^{n-2}(J)\cup L\simeq J^1(n-2)\cup L\sqcup U$. By \cite[Lemma~5.2]{ren2023lee} we know that \eqref{eq:ED_inj} is the induced map by the saddle cobordism $J^1(n-2)\cup L\sqcup U\to J^1(n)\cup L$. Since the dotted annular creation cobordism $Kh(J^1(n-2))\to Kh(J^1(n))$ is obtained by restricting the domain of \eqref{eq:ED_inj} (to the part where $Kh(U)$ contributes a tensor summand $X$), it is also injective in the highest homological degree.
\end{proof}

\section{Lee canonical generators}\label{sec:Lee_generator}
We explain our choice of canonical generators $x_\mathfrak o\in KhR_{Lee}(L)$ in $\mathfrak{gl}_2$ Lee homology in Section~\ref{sbsec:KhR_2} and prove Proposition~\ref{prop:Lee_generator_properties}. To streamline the procedure, we define these canonical generators using the formalism in \cite{morrison2024invariants}; see especially their Section~3.1 and Section~3.2. Note that \cite{morrison2024invariants} is using a different renormalization convention from ours; see their Remark~2.3. We will also remark on a description of the canonical generators $x_\mathfrak o$ up to sign in terms of Rasmussen's generators \cite{rasmussen2010khovanov} in the usual Lee homology.

As homologically $\Z$-graded vector spaces, the $\mathfrak{gl}_2$ Lee homology canonically and functorially decomposes into (twisted) $\mathfrak{gl}_1$ homologies \cite[Equation~(1)]{morrison2024invariants}:
\begin{equation}\label{eq:Lee_=_gl_1}
KhR_{Lee}(L)\cong\bigoplus_\mathfrak o(KhR_1^{(2)}(L_+(\mathfrak o);\Q)\otimes KhR_1^{(-2)}(L_-(\mathfrak o);\Q))[-2\ell k(L_+(\mathfrak o),L_-(\mathfrak o))].
\end{equation}
Here, $[\cdot]$ denotes homological grading shifts, $\mathfrak o$ runs over orientations of $L$ as an unframed link, and $L_+(\mathfrak o)$ (resp. $L_-(\mathfrak o)$) is the sublink of $L$ consisting of components whose orientations agree (resp. disagree) with $\mathfrak o$. The functor $KhR_1$ is the Khovanov-Rozansky $\mathfrak{gl}_1$ homology. It is determined by canonical graded isomorphisms $KhR_1(L)\cong\Z$ for each framed oriented link $L\subset S^3$ where $\Z$ is concentrated in (homological, quantum) bidegree $(0,0)$, such that the vacuum state in $KhR_1(\emptyset)$ corresponds to $1\in\Z$; under these identifications, all Reidemeister I-induced framing-change maps induce $\mathrm{id}_\Z\colon\Z\to\Z$, and $KhR_1(\Sigma)=\mathrm{id}_\Z\colon\Z\to\Z$ for each framed oriented cobordism $\Sigma\colon L_0\to L_1$. The upper script $(\cdot)^{(d)}$, $d\in\Q^\times$ indicates a $d$-twist on morphisms according to the Euler characteristic; explicitly, one may renormalize so that there are canonical graded isomorphisms
\[KhR_1^{(d)}(L;\Q)\cong\Q\]
for any framed oriented link $L\subset S^3$ where $\Q$ concentrates in bidegree $(0,0)$, under which the vacuum in $KhR_1^{(d)}(\emptyset;\Q)$ corresponds to $1\in\Q$; under these identifications the Reidemeister I-induced framing-change maps each induce $\mathrm{id}_\Q\colon\Q\to\Q$, and $$KhR_1^{(d)}(\Sigma;\Q)=d^{n(\Sigma)}\mathrm{id}_\Q\colon\Q\to\Q$$ for any framed oriented link cobordism $\Sigma\colon L_0\to L_1$, where $$n(\Sigma)=\frac{-\chi(\Sigma)+\#\partial_+\Sigma-\#\partial_-\Sigma}{2}\in\Z$$ is a normalized Euler characteristic as defined in \eqref{eq:n_renormalization}, which is additive under gluing cobordisms.

In the formulation of \cite{morrison2024invariants}, the orientation $\mathfrak o$ in \eqref{eq:Lee_=_gl_1} corresponds to a coloring of components of $L$ by $\{\pm1\}$, so that $L_\pm(\mathfrak o)$ is colored by $\pm1$. The colors $1,-1$ are also regarded as colors by idempotents $\A=(1+X)/2,\B=(1-X)/2$ in $V_{Lee}=\Q[X]/(X^2-1)$, respectively.

Note that one does not recover the quantum $\Z/4$-grading and the quantum filtration on $KhR_{Lee}(L)$ from \eqref{eq:Lee_=_gl_1}.

\begin{Def}\label{def:Lee_canonical_generator}
The canonical Lee generator $x_\mathfrak o\in KhR_{Lee}(L)$, under the isomorphism \eqref{eq:Lee_=_gl_1}, is the element $$1\otimes1\in KhR_1^{(2)}(L_+(\mathfrak o);\Q)\otimes KhR_1^{(-2)}(L_-(\mathfrak o);\Q)$$ with homological degree shifted by $-2\ell k(L_+(\mathfrak o),L_-(\mathfrak o))$, in the $\mathfrak o$-th direct summand of the right-hand side of \eqref{eq:Lee_=_gl_1}.
\end{Def}

\begin{Rmk}\label{rmk:natural_renormalization}
There is a more natural renormalization for $KhR_1^{(d)}$, but which requires choosing a square root $d^{1/2}$ of $d$, and extending the coefficient field to $\Q(d^{1/2})$ if necessary. Namely, one may renormalize so that $KhR_1^{(d)}(L;\Q(d^{1/2}))\cong \Q(d^{1/2})$ and $KhR_1^{(d)}(\Sigma;\Q(d^{1/2}))=d^{-\chi(\Sigma)/2}\mathrm{id}_{\Q(d^{1/2})}$. We will not be using this renormalization as this would require that we work over $\Q(\sqrt2,i)$.
\end{Rmk}

The isomorphism \eqref{eq:Lee_=_gl_1} is functorial in $L$ in the sense that if $S\colon L_0\to L_1$ is a framed oriented cobordism, then $KhR_{Lee}(S)\colon KhR_{Lee}(L_0)\to KhR_{Lee}(L_1)$ is determined in the canonical bases by the matrix entries
\begin{equation}\label{eq:KhR_Lee_matrix}
KhR_{Lee}(S)_{x_{\mathfrak o_0},x_{\mathfrak o_1}}=\sum_{\substack{\mathfrak O|_{L_0}=\mathfrak o_0\\\mathfrak O|_{L_1}=\mathfrak o_1}}2^{n(S_+(\mathfrak O))}(-2)^{n(S_-(\mathfrak O))}=2^{n(S)}\sum_{\substack{\mathfrak O|_{L_0}=\mathfrak o_0\\\mathfrak O|_{L_1}=\mathfrak o_1}}(-1)^{n(S_-(\mathfrak O))},
\end{equation}
where $\mathfrak o_i$ runs over orientations of $L_i$, $i=0,1$, $\mathfrak O$ runs over orientations of $S$ compatible with $\mathfrak o_0,\mathfrak o_1$, and $S_\pm(\mathfrak O)$ denotes the union of components of $S$ whose orientations agree/differ with that of $\mathfrak O$.

\begin{Rmk}
More classically, if $D$ is a diagram of an oriented link $L$ and $\mathfrak o$ is an orientation of $L$, Rasmussen \cite{rasmussen2010khovanov} defined a cocycle $\mathfrak s_\mathfrak o(D)\in CKh_{Lee}(D)$, whose class $[\mathfrak s_\mathfrak o(D)]\in Kh_{Lee}(L)$ in the Lee homology of $L$ is independent of the choice of $D$ up to a scalar of the form $\pm 2^k$, $k\in\Z$. Better, Rasmussen \cite{rasmussen2005khovanov} defined a renormalized cocycle $\tilde{\mathfrak s}_\mathfrak o(D):=2^{(w(D)-r(D))/2}\mathfrak s_{\mathfrak o}(D)\in CKh_{Lee}(D)\otimes\Q(\sqrt2)$, where $w(D)$ is the writhe of $D$, and $r(D)$ is the number of Seifert circles in the oriented resolution of $D$. Its class $[\tilde{\mathfrak s}_{\mathfrak o}]=[\tilde{\mathfrak s}_\mathfrak o(D)]\in Kh_{Lee}(L)\otimes\Q(\sqrt2)$ is well-defined up to sign, and behaves nicely under cobordisms as described in \cite[Proposition~3.2]{rasmussen2005khovanov}\footnote{The equation in \cite[Proposition~3.2]{rasmussen2005khovanov} missed a $\tfrac12$ multiplicative factor on the exponent on the right-hand side.}, consistent with the renormalization in Remark~\ref{rmk:natural_renormalization}. One can check, analogously to the proof of item (4) below, that the canonical generator $x_\mathfrak o\in KhR_{Lee}(L)$ for a framed oriented link $L$ in Definition~\ref{def:Lee_canonical_generator} corresponds, up to sign, to the element $2^{\ell k(L_+(\mathfrak o),L_-(\mathfrak o))-\#L/2}[\tilde{\mathfrak s}_{-\mathfrak o}]\in Kh_{Lee}(-L)$ under the canonical (up to sign) isomorphism $KhR_{Lee}(L)\cong Kh_{Lee}(-L)$ (cf. \eqref{eq:KhR_2}; be aware that this isomorphism may involve conjugations on the Frobenius algebra $V_{Lee}$, see \cite{beliakova2023functoriality}), with changes in gradings and filtration suppressed.
\end{Rmk}

\begin{proof}[Proof of Proposition~\ref{prop:Lee_generator_properties}]
(3) follows from the definition. (5) is a reformulation of \eqref{eq:KhR_Lee_matrix}.

To prove the other items, one needs to understand the isomorphism \eqref{eq:Lee_=_gl_1} more explicitly, as described in \cite{morrison2024invariants}. We sketch the arguments below. With appropriate (intricate) sign fixes, one may also prove these in terms of Rasmussen's generators.

(1) The canonical isomorphism $KhR_{Lee}(\emptyset)\cong KhR_1^{(2)}(\emptyset;\Q)\otimes KhR_1^{(-2)}(\emptyset;\Q)$ is determined by demanding that the vacuum corresponds to the vacuum. The canonical generator $x_\emptyset\in KhR_{Lee}(\emptyset)$ thus corresponds to the vacuum $1\in\Q$.

(2) For $(U,\mathfrak o)$ the $0$-framed oriented unknot, let $D$ be the spanning disk of $U$ in $S^3$, with interior pushed slightly into $int(B^4)$. Then the canonical isomorphism $KhR_{Lee}(U)\cong V_{Lee}\{-1\}$ of homologically graded, quantum filtered vector spaces is determined by identifying the class in $KhR_{Lee}(U)$ born from the $\A,\B$-colored disk $D$ with $\A,\B$ (with filtration degree shifts), respectively. On the other hand, the class represented by the $\A$-colored (resp. $\B$-colored) disk corresponds to $1\otimes1$ in the $\mathfrak o$-summand (resp. $\bar{\mathfrak o}$-summand) on the right-hand side of \eqref{eq:Lee_=_gl_1}. This proves the statements for the $0$-framed unknot. For a general $n$-framed unknot $(U^n,\mathfrak o)$, the statements follow from the fact that Reidemeister-I-induced framing-change maps induce a canonical isomorphism $KhR_{Lee}(U^n)\cong KhR_{Lee}(U)\{-n\}$, and that \eqref{eq:Lee_=_gl_1} intertwines with the Reidemeister-I-induced maps on the two sides.

(4) For $L=\emptyset$, the conclusion follows from (1) (note that $0\in\Q$ is homogeneous of arbitrary degree). For $L$ the $n$-framed unknot, the conclusion follows from (2).

Next, we check the statement for $L$ the positive (resp. negative) Hopf link where both components have zero framings, and $\mathfrak o$ an orientation on $L$ making it a negative (resp. positive) Hopf link. The corresponding canonical generator $x_\mathfrak o$ is denoted $h_+$ (resp. $h_-$) in \cite{morrison2024invariants}, which involves a scalar choice. Note that $x_\mathfrak o=h_+$ determines $x_{\bar{\mathfrak o}}$ by symmetry, and $h_+$ determines $h_-$ by \eqref{eq:KhR_Lee_matrix}. Moreover, $q_{\Z/4}(x_\mathfrak o\pm x_{\bar{\mathfrak o}})$ is independent of the scalar choice of $h_+$. Thus, for convenience, we may pick the diagram of $L$ consisting of two counterclockwise oriented crossingless circles, with exactly $2$ positive (resp. negative) crossings between them, and assign any $x_{\mathfrak o}$ for the purpose of proving (4). The advantage of this choice is that there is a planar isotopy of the diagram exchanging the two components of $L$, sending $\mathfrak o$ to $\bar{\mathfrak o}$. With this choice of diagram, one may write down explicit representatives for $x_{\mathfrak o}$ and $x_{\bar{\mathfrak o}}$ analogously to \cite[Example~3.7]{morrison2024invariants} and verify the claims about quantum $\Z/4$ degrees.

If $L=L_1\sqcup L_2$ is a split link, $\mathfrak o_1,\mathfrak o_2$ are orientations on $L_1,L_2$, and $\mathfrak o$ is the corresponding orientation on $L$, then $$x_{\mathfrak o}\pm x_{\bar{\mathfrak o}}=x_{\mathfrak o_1}\otimes x_{\mathfrak o_2}\pm x_{\overline{\mathfrak o_1}}\otimes x_{\overline{\mathfrak o_2}}=\tfrac12((x_{\mathfrak o_1}+x_{\overline{\mathfrak o_1}})\otimes(x_{\mathfrak o_2}\pm x_{\overline{\mathfrak o_2}})+(x_{\mathfrak o_1}-x_{\overline{\mathfrak o_1}})\otimes(x_{\mathfrak o_2}\mp x_{\overline{\mathfrak o_2}})).$$ Hence if $(L_1,\mathfrak o_1)$ and $(L_2,\mathfrak o_2)$ both satisfy the claims on quantum $\Z/4$ gradings, so does $(L,\mathfrak o)$.

Finally, for any pair $(L,\mathfrak o)$ where $L$ is a framed oriented link, and $\mathfrak o$ is an orientation of $L$ as an unframed link, there exists a cobordism $S\colon L\to L'$ so that
\begin{itemize}
\item $L'$ is a disjoint union of some unknots with some framings, and some $0$-framed positive and negative Hopf links;
\item there exists a unique orientation $\mathfrak O$ of $S$ compatible with $\mathfrak o$ on $L$, whose restriction on $L'$, denoted $\mathfrak o'$, does not agree with the orientation of any Hopf link disjoint summands or their opposites.
\end{itemize}
In particular, $(L',\mathfrak o')$ satisfies the claims on quantum $\Z/4$ degrees thanks to the previous arguments. On the other hand, $KhR_{Lee}(S)$ is homogeneous of quantum $\Z/4$ degree $-\chi(S)$, so 
\begin{align*}
q_{\Z/4}(x_{\mathfrak o}\pm x_{\overline{\mathfrak o}})&=q_{\Z/4}(KhR_{Lee}(S)(x_{\mathfrak o}\pm x_{\overline{\mathfrak o}}))+\chi(S)\\
&=q_{\Z/4}(2^{n(S)}(-1)^{n(S_-(\mathfrak O))}(x_{\mathfrak o'}\pm(-1)^{n(S)}x_{\overline{\mathfrak o'}}))+\chi(S)\\
&=q_{\Z/4}(x_{\mathfrak o'}\pm x_{\overline{\mathfrak o'}})+2n(S)+\chi(S)\\
&=-w(L)-\#L-1\pm1,
\end{align*}
as desired.
\end{proof}

\end{document}

%% file: preamble.tex
\usepackage[utf8]{inputenc}
\usepackage{amsfonts,amsthm,amsmath,amssymb,thmtools,comment}
\usepackage{mathtools}
\usepackage{stmaryrd}
\usepackage{float}
\usepackage{array,diagbox}
\usepackage{graphicx}
\usepackage{enumerate}
\usepackage{color}
\usepackage{tikz-cd}
\usepackage{tikz}
\usetikzlibrary{arrows.meta,positioning,calc,fit,backgrounds}
\usetikzlibrary{patterns}
\usepackage[
    maxbibnames=99,
    backend=biber,
    style=alphabetic,
    sorting=nyt,
    giveninits=true
]{biblatex}
\DeclareFieldFormat{pages}{#1}
\renewbibmacro{in:}{%
  \ifentrytype{article}
    {}
    {\bibstring{in}%
     \printunit{\intitlepunct}}}
\DeclareFieldFormat
[article,inbook,incollection,inproceedings,patent,thesis,unpublished]
  {title}{\mkbibemph{#1}}
\DeclareFieldFormat{journaltitle}{#1\isdot}
\DeclareFieldFormat[article]{volume}{\mkbibbold{#1}}
\DeclareFieldFormat[article]{number}{\bibstring{number}\addnbspace #1}

\renewbibmacro*{journal+issuetitle}{%
  \usebibmacro{journal}%
  \setunit*{\addspace}%
  \iffieldundef{series}
    {}
    {\newunit
     \printfield{series}%
     \setunit{\addspace}}%
  \printfield{volume}%
  \setunit{\addspace}%
  \usebibmacro{issue+date}%
  \setunit{\addcomma\space}%
  \printfield{number}%
  \setunit{\addcolon\space}%
  \usebibmacro{issue}%
  \setunit{\addcomma\space}%
  \printfield{eid}
  \newunit}

\newtheoremstyle{mystyle}
  {}
  {}
  {\itshape}
  {}
  {\bfseries}
  {.}
  { }
  {\thmname{#1}\thmnumber{ #2}\thmnote{ (#3)}}

\theoremstyle{mystyle}
\newtheorem{Thm}{Theorem}[section]
\newtheorem{Lem}[Thm]{Lemma}
\newtheorem{Cor}[Thm]{Corollary}
\newtheorem{Prop}[Thm]{Proposition}
\newtheorem{Conj}[Thm]{Conjecture}

\theoremstyle{definition}
\newtheorem{Def}[Thm]{Definition}
\newtheorem{Ex}[Thm]{Example}

\theoremstyle{remark}
\newtheorem{Rmk}[Thm]{Remark}

\declaretheoremstyle[%
  spaceabove=3pt,
  spacebelow=10pt,
  headfont=\normalfont\itshape,%
  postheadspace=.5em,%
  qed=\qedsymbol%
]{mystyle2}

\newcommand{\R}{\mathbb{R}}
\newcommand{\Z}{\mathbb{Z}}

\newcommand{\Q}{\mathbb{Q}}

\newcommand{\CP}{\mathbb{CP}}

\newcommand{\A}{\mathbf{a}}
\newcommand{\B}{\mathbf{b}}

\newcommand{\ILtikzpic}[2][]{
\vcenter{\hbox{\begin{tikzpicture}[#1]
#2
\end{tikzpicture}}}
}

\newcommand{\drawover}[2][thick]{
\draw[line width=2mm,white] #2
\draw[#1] #2
}

\usepackage[hypertexnames=false]{hyperref}
\usepackage[margin=1in]{geometry}